\documentclass[final,onefignum,onetabnum]{siamart190516}
\usepackage{graphicx}
\usepackage{amsmath,multirow}
\usepackage{amssymb}
\usepackage{amscd}
\usepackage{amsfonts}
\usepackage{enumerate}
\usepackage{mathrsfs}
\usepackage{xypic}
\usepackage{stmaryrd}
\usepackage{graphicx}
\usepackage{fancybox}
\usepackage{color}
\usepackage{amsmath}
\usepackage{exscale}
\usepackage{enumitem,array}%
\usepackage{amssymb}
\usepackage{amsmath}
\usepackage{amssymb}
\usepackage{amscd}
\usepackage{amsfonts}
\usepackage{enumerate}
\usepackage{mathrsfs}
\usepackage{xy}
\usepackage{stmaryrd}
\usepackage{graphicx}
\usepackage{fancybox}
\usepackage{color}
\usepackage{multirow}
\usepackage{exscale}
\usepackage{enumitem,array}%
\usepackage{subfigure}
\usepackage{extarrows}
\usepackage{amsfonts}
\usepackage{xypic}
\usepackage{tikz}
\usepackage{algorithm}
\usepackage{algorithmicx}
\usepackage{algpseudocode}
\usepackage{manyfoot}
\usepackage{verbatim}
\usepackage{graphicx}
\usepackage{amsopn}
\usetikzlibrary{arrows,chains,matrix,positioning,scopes}
\makeatletter
\tikzset{join/.code=\tikzset{after node path={%
\ifx\tikzchainprevious\pgfutil@empty\else(\tikzchainprevious)%
edge[every join]#1(\tikzchaincurrent)\fi}}}
\makeatother
\tikzset{>=stealth',every on chain/.append style={join},
         every join/.style={->}}
\tikzstyle{labeled}=[execute at begin node=$\scriptstyle,
   execute at end node=$]

\newsiamremark{remark}{Remark}
\newsiamremark{hypothesis}{Hypothesis}
\crefname{hypothesis}{Hypothesis}{Hypotheses}
\newsiamthm{claim}{Claim}
\usepackage{amssymb,amsmath,geometry,color}
\usepackage{epstopdf}
\usepackage{epsfig,multirow}
\usepackage{color,cases}

\newcommand{\fe}{\mathrm{e}}

\newcommand{\eps}{\varepsilon}

\newcommand{\bT}{{\mathbb T}}

\newcommand{\norm}[1]{\left\Vert#1\right\Vert}
\def\ii{\textmd{i}}

\def\hh{\triangle t}

\headers{Uniformly accurate numerical schemes for the NLDE}{L. Wang, B. Wang, J. Li}

\title{Fourth-order uniformly accurate integrators with long time near conservations for the nonlinear Dirac equation in the nonrelativistic regime
\thanks{Submitted to the editors DATE.
\funding{B. Wang is supported by NSFC 12371403. }}}

\author{Lina Wang\thanks{School of Mathematics and Statistics, Xi'an Jiaotong University, 710049 Xi'an, China
  (wanglina@stu.xjtu.edu.cn).}
  \and Bin Wang\thanks{Corresponding author.
  School of Mathematics and Statistics, Xi'an Jiaotong University, 710049 Xi'an, China
  (wangbinmaths@xjtu.edu.cn).}
  \and Jiyong Li\thanks{School of Mathematics Science, Hebei Key Laboratory of Computatioanal Mathematics  and Applications, Hebei International Joint Research Center for Mathematics and Interdisciplinary Science, Hebei Normal University, Shijiazhuang 050024,  China
  (ljyong406@163.com).}
 }

\begin{document}

\maketitle

% REQUIRED
\begin{abstract}
In this paper, we propose two novel fourth-order integrators that exhibit uniformly high accuracy and long-term near conservations  for solving the nonlinear Dirac equation (NLDE) in the nonrelativistic regime. In this regime, the solution of the NLDE exhibits highly oscillatory behavior in time, characterized by a wavelength of $\mathcal{O}(\eps^{2})$ with a small parameter $ \eps>0$. To ensure uniform temporal accuracy, we employ a two-scale approach in conjunction with exponential integrators, utilizing operator decomposition techniques for the NLDE. The proposed methods are rigorously proved to achieve fourth-order uniform accuracy in time for all $\eps\in (0,1]$. Furthermore, we successfully incorporate symmetry into the integrator, and the long-term near conservation properties are analyzed through the modulated Fourier expansion. The proposed schemes are readily extendable to linear Dirac equations incorporating magnetic potentials, the dynamics of traveling wave solutions and the two/three-dimensional Dirac equations.   The validity of all theoretical findings and  extensions is numerically substantiated through     a series of numerical experiments.
\end{abstract}

% REQUIRED

% REQUIRED

\section{Introduction}

The Dirac equation, first proposed by Paul Dirac in 1928, serves as a fundamental framework for describing all massive spin-1/2 particles, including electrons, positrons, and so on \cite{PD1}. As a cornerstone of quantum physics, this equation elegantly unifies the principles of quantum mechanics and special relativity. In recent years, the Dirac equation has gained renewed significance due to its applicability in diverse theoretical studies, such as the electronic properties of graphite and graphene \cite{SGM,CG,NGMJ}, interactions of molecules with intense laser fields \cite{GLB1,GLB2}, and Bose-Einstein condensation phenomena \cite{BaoCai,HC}. Consequently, the numerical solution of the Dirac equation has garnered substantial attention from researchers across various disciplines.

In the study of the Dirac equation, the choice of parameters gives rise to a variety of regimes, among which are the classical regime, the massless regime, the semi-classical regime \cite{CLY, WHJY}, and the non-relativistic regime (\cite{BCJT16,BCJT,XST}).  In this work, we focus on the nonlinear Dirac equation (NLDE) in the nonrelativistic regime, specifically in the absence of a magnetic potential \cite{BCY}.    As demonstrated in \cite{BCJY, BY}, the three-dimensional  Dirac equation, expressed in its four-component form to describe the time evolution of spin-1/2 massive particles, can be reduced to a two-component form in two dimensions   and one dimension  under appropriate assumptions on the electromagnetic potentials. Therefore, here we focus on the two-component form  and it is noted that all the integrators developed in this paper can be straightforwardly extended to the Dirac equation of a four-component form.
 In one or two spatial dimensions, the equation can be expressed in a two-component form, characterized by a complex-valued vector wave function
 $\Phi:=\Phi(t,\mathbf{x})=(\phi_{1}(t,\mathbf{x}),\phi_{2}(t,\mathbf{x}))^{\intercal}\in \mathbb{C}^{2}$:
\begin{equation}\label{equ-1}
\begin{aligned}
  &\ii\partial_{t}\Phi=\bigg(-\frac{\ii}{\varepsilon}\sum\limits_{j=1}^{d}\sigma_{j}\partial_{j}+\frac{\sigma_{3}}{\varepsilon^2}\bigg)\Phi
  +V(\mathbf{x})\Phi+\mathbf{F}(\Phi)\Phi, \ d=1, 2, \ t>0, \ \Phi(0,\mathbf{x})=\Phi_{0}(\mathbf{x}), \ \mathbf{x}\in\mathbb{R}^{d},
\end{aligned}
\end{equation}
where $\ii=\sqrt{-1}$ is the imaginary unit, $t$ is time, $\varepsilon\in (0,1]$ is a dimensionless parameter inversely proportional to the speed of light,
$\mathbf{x} = (x_{1},\ldots, x_{d})^{\intercal}\in \mathbb{R}^d$ is the spatial coordinate vector, $\partial_{j}=\frac{\partial}{\partial x_{j}}\ (j=1,\ldots,d)$,
$V(\mathbf{x})$ is a real-valued function denoting the external electric potential, and
$\sigma_{1}, \sigma_{2}, \sigma_{3}$ are the Pauli matrices defined as
\begin{equation}\label{equ-3}
\begin{matrix}
\sigma_{1}=\left(\begin{array}{cc} 0 & 1 \\ 1 & 0 \\ \end{array}\right),
\end{matrix}
\quad
\begin{matrix}
\sigma_{2}=\left(\begin{array}{cc} 0 & -\ii \\ \ii & 0 \\ \end{array}\right),
\end{matrix}
\quad
\begin{matrix}
\sigma_{3}=\left(\begin{array}{cc} 1 & 0 \\ 0 & -1 \\ \end{array}\right).
\end{matrix}
\end{equation}
The commonly used form for nonlinearity $\mathbf{F}(\Phi)$ in \eqref{equ-1} is
\begin{align}\label{equ-4}
\mathbf{F}(\Phi)=\lambda_{1}(\Phi^{*}\sigma_{3}\Phi)\sigma_{3}+\lambda_{2}|\Phi|^{2}I_{2},
\end{align}
with $|\Phi|^{2}=\Phi^{*}\Phi$, where $\lambda_{1}, \lambda_{2}\in \mathbb{R}$ are two given real constants, $\Phi^{*}=\overline{\Phi}^{\intercal}$ is the complex conjugate transpose of $\Phi$. The   nonlinearity is inspired by two distinct physical contexts: the Soler model in quantum field theory, where
$\lambda_{1}\neq 0$ and $\lambda_{2}=0$ \cite{FLR,FS}, and the Bose-Einstein condensation (BEC) with chiral confinement and spin-orbit coupling, where
$\lambda_{1}=0$ and $\lambda_{2}\neq 0$ \cite{CEL,HC,LWC}. It is well known that the solution of  NLDE  \eqref{equ-1} conserves the total mass
\begin{align}\label{equ-6}
M(\Phi):= \int_{\mathbb{R}^{d}}|\Phi(t,\mathbf{x})|^{2}d\mathbf{x} = \int_{\mathbb{R}^{d}}\sum\limits_{j=1}^{2}|\phi_{j}(t,\mathbf{x})|^{2}
d\mathbf{x}   \equiv M(\Phi_0),
\end{align}
and the total energy
\begin{align}\label{equ-7}
E(\Phi):&=\int_{\mathbb{R}^{d}}\bigg(-\frac{\ii}{\varepsilon}\sum\limits_{j=1}^{d}\Phi^{*}\sigma_{j}\partial_{j}\Phi+\frac{1}{\varepsilon^{2}}\Phi^{*}\sigma_{3}\Phi +V(\mathbf{x})|\Phi|^{2}+\frac{\lambda_{1}}{2}(\Phi^{*}\sigma_{3}\Phi)^{2}+\frac{\lambda_{2}}{2}|\Phi|^{4} \bigg)d\mathbf{x} \equiv E(\Phi_0).
\end{align}

%On the analytical side, the existence or diversity of bound states or standing wave solutions is analyzed in \cite{BCDM,BD,CV,DES,ES}. In particular, for the choice of $F(\Phi)$ with $\lambda_{1}=-1, \lambda_{2}=0$ and $d=1$, $V(x)=0$, there exists exact soliton solutions to the Dirac equation \eqref{equ-1} \cite{M,T79,ST,XST}.
%Many numerical methods have been proposed to solve the Dirac equation \eqref{equ-1}, including the finite difference time domain (FDTD) methods \cite{BCJ,BHM,HPA,NSG}, time splitting and time-splitting Fourier pseudospectral (TSFP) method \cite{BCY,BCJ,BL,FS89,HJMSZ}, exponential wave integrator Fourier pseudospectral (EWI-FP) methods \cite{BCJ,L1,L3} and the Runge-Kutta discontinuous Galerkin methods \cite{HL}. In fact, when $0<\eps\ll 1$, the solution of Dirac equation \eqref{equ-1} propagates waves with wavelength at $\mathcal{O}(1)$ in space and $\mathcal{O}(\eps^{2})$ in time.
\begin{figure}[t!]
$$\begin{array}{cc}
\psfig{figure=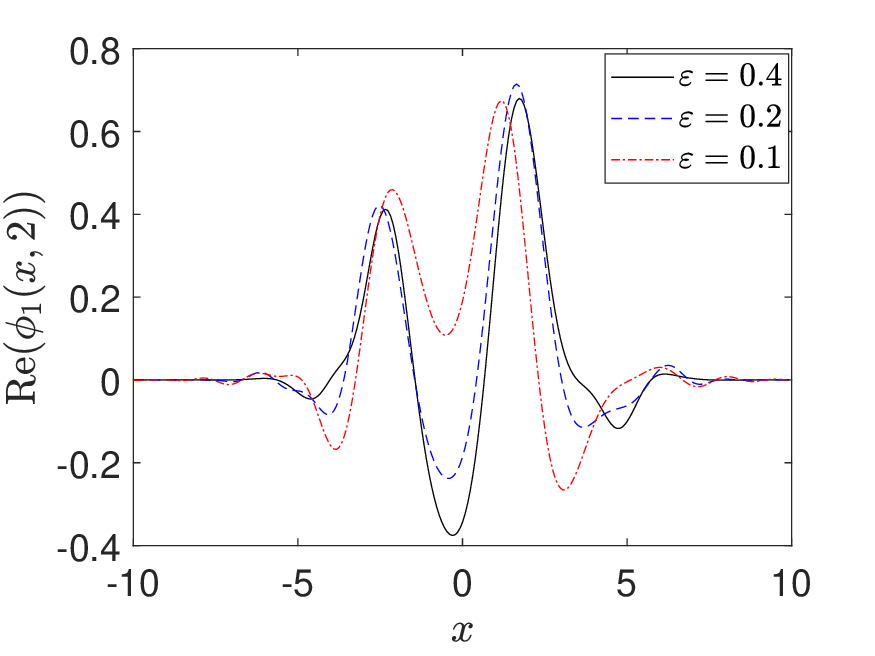,height=3.0cm,width=6.0cm}
\psfig{figure=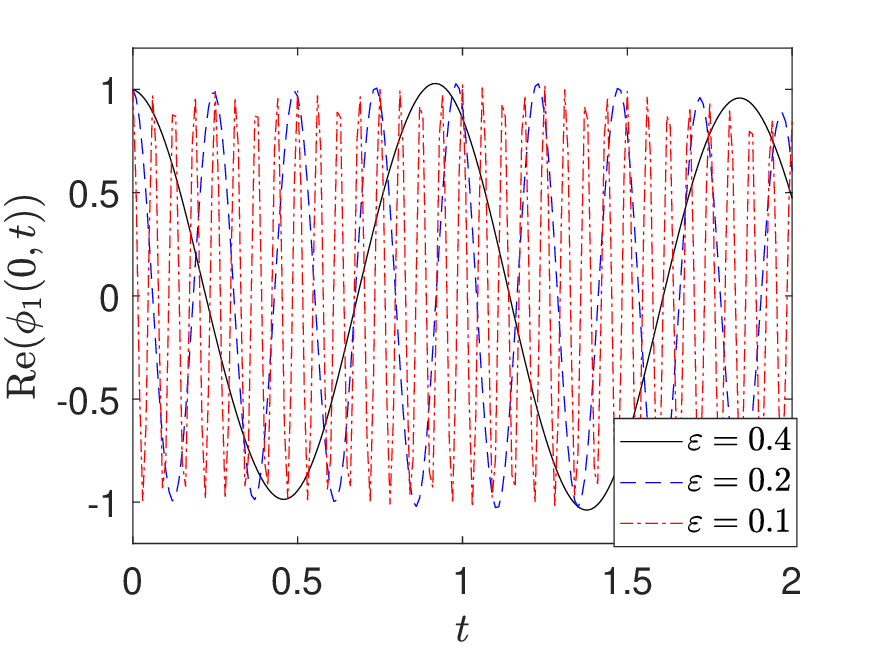,height=3.0cm,width=6.0cm}
\end{array}$$
\caption{The solution $\phi_{1}(t=2,x)$ and $\phi_{1}(t,x=0)$ of the Dirac equation \eqref{equ-1} with $d=1$ for different $\eps$ (Re($\phi_{1}$) denotes the real part of $\phi_{1}$).}\label{fig-3-1-4}
\end{figure}

From an analytical perspective, the existence and diversity of bound states or standing wave solutions have been extensively studied in \cite{BCDM,BD,CV,DES,ES}. Notably, for the specific choice of $F(\Phi)$ with $\lambda_{1}=-1, \lambda_{2}=0$ and $d=1$, $V(x)=0$, exact soliton solutions to the Dirac equation \eqref{equ-1} have been established \cite{M,T79,ST,XST}.
On the numerical front, a variety of methods have been developed to solve the Dirac equation \eqref{equ-1} such as finite difference time domain (FDTD) methods \cite{BCJY,BHM,HPA,NSG}, time-splitting and time-splitting Fourier pseudospectral (TSFP) methods \cite{BCY,BCJY,BY,BL,FS89,HJMSZ}, exponential wave integrator Fourier pseudospectral (EWI-FP) methods \cite{BCJY,L1}, and Runge-Kutta discontinuous Galerkin methods \cite{HL}. It is worth noting that when
$0<\eps\ll 1$, the solution of the Dirac equation \eqref{equ-1} exhibits wave propagation with spatial wavelengths of $\mathcal{O}(1)$ and temporal wavelengths of $\mathcal{O}(\eps^{2})$. To further illustrate this behavior, Figure \ref{fig-3-1-4} presents the numerical solution of the Dirac equation \eqref{equ-1} with $d=1$, $V(x)=\frac{1-x}{1+x^{2}}$, $\lambda_{1}=-1$, $\lambda_{2}=0$ and $\Phi_{0}(x)=\left(\exp(-x^{2}/2),\exp(-(x-1)^{2}/2)\right)^{\intercal}$ for different $\eps$. It can be clearly observed that the small temporal wavelength of the solution induces high temporal oscillations, posing significant challenges for both theoretical analysis and classical numerical discretization in the nonrelativistic regime where $0<\eps\ll 1$. To   ensure the accuracy of numerical solutions, the time step and spatial mesh size in traditional  methods have to depend on $\eps$. Consequently, developing numerical methods that permit the use of $\eps$-independent time step  is both essential and highly challenging to effectively address these temporal oscillations.

Recently, significant advancements have been achieved in the numerical  solution of highly oscillatory  problems through the  development of uniformly accurate (UA) schemes \cite{BCZ,BDZ,BZ,CCLM,Zhao}.  Both the time stepsize and the accuracy are independent of $\eps$, which is the primary advantage of UA methods.  Improved   error bounds on  splitting methods for the Dirac equation with weakly nonlinear or small potentials have been analyzed in \cite{BCF,BFY}.  Utilizing two-scale formulation, two numerical schemes with first- and second-order uniform accuracies  are constructed for  the nonlinear Dirac equation  in the nonrelativistic regime \cite{LMZ}. The work \cite{BCJT16} proposed a novel method  which   achieves first-order uniform accuracy  for all $\eps\in(0,1]$ and optimal  second-order accuracy in the regimes when either $\eps=O(1)$ or $0<\eps\lesssim \triangle t$ for a time stepsize $\triangle t$.  Another method with  similar accuracy   has been proposed in  \cite{CW2} also for  the nonlinear Dirac equation.  Additionally, uniform first- and second-order Picard iteration approaches   have been applied to the nonlinear Dirac equation in \cite{CW}, and UA methods up to third order have been formulated for the linear Dirac equation in  \cite{CW1}.
Despite these great advancements, the existing UA methods are limited to second-order temporal accuracy for the nonlinear Dirac equation and third-order temporal accuracy for the linear case, and their formulation and analysis are computationally intensive.  Moreover, while various structure-preserving  methods have been investigated for the mass and energy conservation properties of the Dirac equation \cite{L1,L2}, a thorough long-time analysis of UA schemes for the nonlinear Dirac equation is notably absent. In summary, for the nonlinear Dirac equation  \eqref{equ-1} in the nonrelativistic limit, it seems that all the UA schemes are constrained to second-order accuracy, and a detailed long-time analysis of their conservation properties   is absent.

The objective of this work  is to develop and analyze a family of  high order  uniformly accurate integrators and the main  contributions  are as follows.
a) This paper provides a general framework for the construction  of high order UA algorithms  for the NLDE.  Compared with   the existing UA methods, the accuracy of the proposed  integrators are improved to fourth order, which is   very competitive  in the numerical computation of the NLDE.
  b) Utilizing the robust modulated Fourier expansion technique, we conduct a long-time analysis of the developed symmetric UA algorithm. The analysis reveals that the novel method   maintains the energy and mass conservations of the NLDE   over long times.
%c) Moreover, we should note that all the methods and analysis can be extended to  the multi-frequency oscillatory Hamiltonian ODEs/PDEs and
%both the improved uniform accuracy and long-time analysis  still   hold. This means that this paper provides a general framework for the construction and analysis of new methods for second-order highly oscillatory systems.

The article is structured as follows. In Section \ref{sec-2}, we   present  the formulation of   novel  schemes. Section \ref{sec-3} studies the fourth-order uniform error bounds and   long-time near conservations  of the proposed methods. In Section \ref{sec-4}, five numerical experiments are conducted for various one-dimensional and two-dimensional   Dirac equations, numerically demonstrating  the  accuracy and long-time behaviour. The last section is devoted to the conclusions of this paper.

\section{Formulation of the method}\label{sec-2}

Without loss of generality, we restrict this section to the one-dimensional case ($d=1$) of \eqref{equ-1}. The extension to the two-dimensional case, as well as to the four-component form in three dimensions for tensor grids, is straightforward.

\subsection{Transformation of the equation} \label{sec-2.1}
For integer $m>0$, $\Omega=[a,b]$, we denote by $H^{m}(\Omega)$ the standard Sobolev space.
%\begin{equation}\label{equ-2-9-2}
%\norm{f}_{m}^{2}=\sum\limits_{l=-\infty}^{\infty}(1+|\mu_{l}|^{2})^{m}|\hat{f}_{l}|^{2}\ \mbox{for} \ f(x)=\sum\limits_{l=-\infty}^{\infty}\widehat{f}_{l}e^{i\mu_{l}(x-a)},\ \mu_{l}=\frac{2\pi l}{b-a},
%\end{equation}
%with the Fourier transform coefficients is defined as
%$\widehat{f}_{l}=\frac{1}{b-a}\int\limits_{a}^{b}f(x)e^{-i\mu_{l}(x-a)}dx.$
%For the two-component Dirac equation \eqref{equ-1}, we introduce the  notation
%$
%\norm{U}_{m}^{2}=\norm{U_{1}}_{m}^{2}+\norm{U_{2}}_{m}^{2}\ \textmd{for} \ U=(U_{1},U_{2})^{\intercal}\in[H^{m}(\Omega)]^{2}.
%$
Introduce the Dirac operator \cite{BCY}
\begin{align}\label{equ-8}
Q^{\eps}=-\ii\varepsilon\sigma_{1}\partial_{x}+\sigma_{3}, \ x\in\mathbb{R}.
\end{align}
In the phase space (Fourier domain), $Q^{\eps}$  is diagonalizable and can be decomposed as
\begin{equation}\label{equ-10}
Q^{\eps}=\sqrt{Id-\eps^{2}\Delta}\Pi_{+}^{\eps}-\sqrt{Id-\eps^{2}\Delta}\Pi_{-}^{\eps},
\end{equation}
where $\Delta=\partial_{xx}$ represents the one-dimensional Laplace operator, $Id$ denotes the identity operator, and $\Pi_{\pm}^{\eps}: (H^{m}(\mathbb{R}))^{2}\rightarrow (H^{m}(\mathbb{R}))^{2} $ are projector operators
\begin{align}\label{equ-11}
\Pi_{+}^{\eps}=\frac{1}{2}[Id+(Id-\eps^{2}\Delta)^{-1/2}Q^{\eps}], \ \ \ \ \Pi_{-}^{\eps}=\frac{1}{2}[Id-(Id-\eps^{2}\Delta)^{-1/2}Q^{\eps}].
\end{align}
It can be verified that the projection operators $\Pi_{+}^{\eps}$ and $\Pi_{-}^{\eps}$ satisfy  $\Pi_{+}^{\eps}+\Pi_{-}^{\eps}=Id,$ $\Pi_{+}^{\eps}\Pi_{-}^{\eps}=\Pi_{-}^{\eps}\Pi_{+}^{\eps}=0$ and $(\Pi_{\pm}^{\eps})^{2}=\Pi_{\pm}^{\eps}$. Furthermore, through Taylor expansion, we obtain
\begin{equation}\label{equ-12}
\Pi_{\pm}^{\eps}=\Pi_{\pm}^{0}\pm\eps \mathcal{R}_{1}=\Pi_{\pm}^{0}\mp \ii\frac{\eps}{2}\sigma_{1}\partial_{x}\pm\eps^{2} \mathcal{R}_{2}, \ \ \ 
\Pi_{+}^{0}=\text{diag}(1,0), \ \ \  \Pi_{-}^{0}=\text{diag}(0,1),
\end{equation}
with $\mathcal{R}_{1} : (H^{m}(\mathbb{R}))^{2}\rightarrow (H^{m-1}(\mathbb{R}))^{2}$ for $m\geq 1$ and $\mathcal{R}_{2} : (H^{m}(\mathbb{R}))^{2}\rightarrow (H^{m-2}(\mathbb{R}))^{2}$ for $m\geq 2$ which are uniformly bounded w.r.t $\eps$. Denote
\begin{align}\label{equ-13}
\mathcal{D}^{\eps}=\frac{1}{\eps^{2}}(\sqrt{(Id-\eps^{2}\Delta)}-Id)=-(\sqrt{(Id-\eps^{2}\Delta)}+Id)^{-1}\Delta.
\end{align}
This is an uniformly bounded    mapping from $(H^{m}(\mathbb{R}))^{2}\rightarrow (H^{m-2}(\mathbb{R}))^{2}$ for $m\geq 2$. Moreover, it satisfies the inequality
$0\leq \mathcal{D}^{\eps} \leq -\frac{1}{2}\Delta$ for all $\eps>0$. Consequently, the operator
$e^{\frac{\ii t}{\eps^{2}}Q^{\eps}}$ can be expressed as
\begin{align}\label{equ-2-9-1}
e^{\frac{\ii t}{\eps^{2}}Q^{\eps}}=e^{\frac{\ii t}{\eps^{2}}(\sqrt{Id-\eps^{2}\Delta}\Pi_{+}^{\eps}-\sqrt{Id-\eps^{2}\Delta}\Pi_{-}^{\eps})}
=e^{\frac{\ii t}{\eps^{2}}}e^{\ii t \mathcal{D}^{\eps}}\Pi_{+}^{\eps}+e^{-\frac{\ii t}{\eps^{2}}}e^{-\ii t \mathcal{D}^{\eps}}\Pi_{-}^{\eps}.
\end{align}
Using the operator $Q^{\eps}$,   the NLDE \eqref{equ-1} can be rewritten as
\begin{align}\label{equ-9}
&\partial_{t}\Phi(t,x)=\frac{-\ii Q^{\eps}}{\eps^{2}}\Phi(t,x)-\ii\bigg(V(x)\Phi(t,x)+\mathbf{F}(\Phi(t,x))\Phi(t,x)\bigg), \ x\in\Omega, \\ \nonumber
&\Phi(0,x)=\Phi_{0}(x), \ x\in\overline{\Omega}, \ \Phi(t,a)=\Phi(t,b), \ t\in[0,T],
\end{align}
where $\Omega=(a,b)$ is a bounded domain with periodic boundary conditions.

To filter out the dominant oscillations in   \eqref{equ-9}, we introduce a transformed variable
\begin{align}\label{equ-2-9-4}
\Psi(t,x)=(\psi_{1}(t,x),\psi_{2}(t,x)):=e^{\frac{\ii tQ^{\eps}}{\eps^{2}}}\Phi(t,x), \ x\in\Omega.
\end{align}
Then the system \eqref{equ-9} is transformed into
\begin{equation}\label{equ-2-9-5}
\partial_{t}\Psi(t,x)=e^{\frac{\ii tQ^{\eps}}{\eps^{2}}}g(e^{-\frac{\ii tQ^{\eps}}{\eps^{2}}}\Psi(t,x)),\ \Psi(0,x)=\Phi_{0}(x),\ t\in(0,T],
\end{equation}
where $g(\Phi)=-\ii\big(V(x)\Phi(t,x)+\mathbf{F}(\Phi(t,x))\Phi(t,x)\big)$. By \eqref{equ-2-9-1}, it follows that
\begin{equation}\label{equ-2-19-1}
\partial_{t}\Psi(t,x)=\big(e^{\frac{\ii t}{\eps^{2}}}e^{\ii t \mathcal{D}^{\eps}}\Pi_{+}^{\eps}+e^{-\frac{\ii t}{\eps^{2}}}e^{-\ii t \mathcal{D}^{\eps}}\Pi_{-}^{\eps}\big)
g\big(\big(e^{-\frac{\ii t}{\eps^{2}}}e^{-\ii t \mathcal{D}^{\eps}}\Pi_{+}^{\eps}+e^{\frac{\ii t}{\eps^{2}}}e^{\ii t \mathcal{D}^{\eps}}\Pi_{-}^{\eps}\big)\Psi(t,x) \big),
\end{equation}
where  inserting \eqref{equ-8} into \eqref{equ-11} gives  $$
\Pi_{+}^{\eps}=\frac{1}{2}\left(\begin{array}{cc} 1+\delta & -\ii\eps\partial_{x}\delta \\
-\ii\eps\partial_{x}\delta  & 1-\delta \\ \end{array}\right),\ \
\Pi_{-}^{\eps}=\frac{1}{2}\left(\begin{array}{cc} 1-\delta & \ii\eps\partial_{x}\delta \\
\ii\eps\partial_{x}\delta & 1+\delta \\ \end{array}\right)
$$
with $\delta=(1-\eps^{2}\Delta)^{-\frac{1}{2}}$.

\subsection{Two-scale formulation and its initial value}\label{sec-2.2}
It can be observed that \eqref{equ-2-9-5} involves two distinct time scales, the fast scale $t/\eps^{2}$ and the slow scale $t$. For small $\eps$, the fast scale $\tau$ varies significantly more rapidly than the slow scale $t$. To address this, we derive a two-scale formulation by explicitly separating the fast scale $t/\eps^{2}$   from the slow scale $t$. Specifically, we introduce an augmented system of the Dirac equation, governed by the augmented solution $U(t,\tau,x)$, which is $2\pi$ periodic in $\tau$. Here, $U(t,\tau,x)=(U_{1}(t,\tau,x),U_{2}(t,\tau,x))$ coincides with the original solution $\Psi(t,x)$ when $\tau=t/\eps^{2}$, i.e., $\Psi(t,x)=U\big(t,t/\eps^{2},x\big)$. Set
\begin{align}\label{equ-2-9-6}
G(t,\tau,U)=\big(e^{\ii\tau}e^{\ii t \mathcal{D}^{\eps}}\Pi_{+}^{\eps}+e^{-\ii\tau}e^{-\ii t \mathcal{D}^{\eps}}\Pi_{-}^{\eps}\big)
g\bigg(\big(e^{-\ii\tau}e^{-\ii t \mathcal{D}^{\eps}}\Pi_{+}^{\eps}+e^{\ii\tau}e^{\ii t \mathcal{D}^{\eps}}\Pi_{-}^{\eps}\big)U \bigg).
\end{align}
According to \eqref{equ-2-19-1},   we   obtain the following two-scale equation of NLDE
\begin{equation}\label{equ-TS-1}
\partial_{t}U(t,\tau,x)+\frac{1}{\eps^{2}}\partial_{\tau}U(t,\tau,x)=G(t,\tau,U(t,\tau,x)), \ t\in[0,T], \ x\in\Omega, \ \tau\in\mathbb{T}_{\tau},
\end{equation}
where $\mathbb{T}_{\tau}:=\mathbb{R}/(2\pi\mathbb{Z})$ denotes the torus. It is noted that the initial value only needs to satisfy  $U(0,0,x)=\Phi_{0}(x)$ and for  $U(0,\tau,x)$, it can be designed as follows.

For a periodic function $v(\cdot)$ on $\bT_{\tau}$, we define the averaging, differentiation and its inversion operators as
 $$\Pi v:=\frac{1}{2\pi}\int_0^{2\pi}v(\tau)d\tau, \quad
 L:=\partial_\tau, \quad
 L^{-1}v:=(I-\Pi)\int_0^\tau v(s)ds,\quad
  A:=L ^{-1} (I-\Pi).$$
The  solution $U(t,\tau,x)$ of \eqref{equ-TS-1} is expressed through the Chapman-Enskog expansion \cite{CCLM,Zhao}
 \begin{equation}\label{equ-2-9-7}
U(t,\tau,x)=\underline{U}(t,x)+h(t,\tau,x),
\end{equation}
where $\underline{U}(t,x):=\Pi U(t,\tau,x)$ is the averaged part and $h(t,\tau,x)$ is the correction part which satisfies
$ \Pi h(t,\tau,x)=0.$ Acting the operator $\Pi $ to both sides of \eqref{equ-TS-1}, we derive the following result
\begin{equation}\label{equ-2-9-8}
\partial_{t} \underline{U}(t, x)=\Pi G(t,\tau,\underline{U}(t,x)+h(t,\tau,x)).
\end{equation}
Subtracting the above from \eqref{equ-TS-1} leads to
\begin{equation}\label{equ-2-9-9}
 \partial_{t}h(t,\tau,x)+\frac{1}{\eps^{2}}\partial_{\tau} h(t,\tau,x)
=(I- \Pi) G(t,\tau,\underline{U}(t,x)+h(t,\tau,x)),
\end{equation}
which by inverting $L=\partial_\tau$ can be rewritten as
\begin{equation}\label{equ-2-9-10}
h(t,\tau,x)=\eps^{2}\big( A G(t,\tau,\underline{U}(t,x)+h(t,\tau,x))- L^{-1}\left(\partial_{t} h(t,\tau,x)\right)\big).
\end{equation}
From \eqref{equ-2-9-10}, it is evident that $h=\mathcal{O}(\eps^{2})$ holds formally, motivating the consideration of the following asymptotic expansion:
 \begin{equation}\label{equ-2-9-11}
h(t,\tau,x)=\eps^{2} h_{1}(t,\tau,\underline{U}(t,x))+\eps^{4} h_{2}(t,\tau,\underline{U}(t,x))+\eps^{6} h_{3} (t,\tau,\underline{U}(t,x))+\mathcal{O}(\eps^{8}),
\end{equation}
with some $h_{1}, h_{2}, h_{3}=\mathcal{O}(1)$ to be determined.

Substituting \eqref{equ-2-9-11} into \eqref{equ-2-9-10} and expanding $G(t,\tau,\cdot)$ as a Taylor series around $\underline{U}$, we have
\begin{equation}\label{equ-2-9-12}
\begin{aligned}
\eps^{2} h_{1}&+\eps^{4} h_{2}+\eps^{6} h_{3}+\mathcal{O}(\eps^{8})=
\eps^{2} A G(t,\tau,\underline{U})+\eps^{2} A \partial_{u} G(t,\tau,\underline{U})(\eps^{2} h_{1}+\eps^{4} h_{2}) \\
&+\frac{1}{2}\eps^{2} A \partial_{u}^{2} G(t,\tau,\underline{U})\big(\eps^{2} h_{1},\eps^{2} h_{1}\big)-\eps^{2} L^{-1}(\eps^{2}  \partial_{t}h_{1}+\eps^{4} \partial_{t} h_{2})+\mathcal{O}(\eps^{8}).
\end{aligned}
\end{equation}
By comparing the coefficients of $\eps^{2j}$ for  $j=1,2,3$,  we derive the following relationships:
\begin{equation}\label{equ-2-9-14}
\begin{aligned}
h_{1}(t,\tau,\underline{U}):=&AG(t,\tau,\underline{U}), \\
h_{2}(t,\tau,\underline{U}):=&A\partial_{u}G(t,\tau,\underline{U})AG(t,\tau,\underline{U})-A^{2}\partial_{t}G(t,\tau,\underline{U})-A^{2}\partial_{u}G(t,\tau,\underline{U})
\Pi G(t,\tau,\underline{U}), \\
h_{3}(t,\tau,\underline{U}):=&A\partial_{u}G(t,\tau,\underline{U})h_{2}+\frac{1}{2}A\partial_{u}^{2}G(t,\tau,\underline{U})(h_{1},h_{1})-A^{2}\partial_{ut}G(t,\tau,\underline{U})
AG(t,\tau,\underline{U}) \\
&-A^{2}\partial_{u}^{2}G(t,\tau,\underline{U})(\Pi G(t,\tau,\underline{U}),AG(t,\tau,\underline{U}))-A^{2}\partial_{u}G(t,\tau,\underline{U})A\partial_{t}G(t,\tau,\underline{U}) \\
&-A^{2}\partial_{u}G(t,\tau,\underline{U})A\partial_{u}G(t,\tau,\underline{U})\Pi G(t,\tau,\underline{U}) \\
&-A^{2}\partial_{u}G(t,\tau,\underline{U})(\Pi\partial_{u}G(t,\tau,\underline{U}),AG(t,\tau,\underline{U}))+A^{3}\partial_{t}^{2}G(t,\tau,\underline{U}) \\
&+2A^{3}\partial_{tu}G(t,\tau,\underline{U})\Pi G(t,\tau,\underline{U})+A^{3}\partial_{u}^{2}G(t,\tau,\underline{U})(\Pi G(t,\tau,\underline{U}), \Pi G(t,\tau,\underline{U})) \\
&+A^{3}\partial_{u}G(t,\tau,\underline{U})\Pi\partial_{t}G(t,\tau,\underline{U})+A^{3}\partial_{u}G(t,\tau,\underline{U})\Pi\partial_{u}G(t,\tau,\underline{U})
\Pi G(t,\tau,\underline{U}).
\end{aligned}
\end{equation}
Since $\Pi$  and $ A$ are bounded on $C^0(\mathbb{T}_{\tau};H^{\sigma})$ for any $\sigma$,
it follows  that $h_{1}, h_{2}, h_{3}=\mathcal{O}(1)$ under suitable Sobolev norm, and this ensures the consistency of the ansatz \eqref{equ-2-9-11} and the derivation process.

We now proceed to derive the complete initial value  $U(0,\tau,x)$. Utilizing the Chapman-Enskog expansion \eqref{equ-2-9-7} and \eqref{equ-2-9-11} at the initial time, we obtain that
$$U(0,\tau,x)=\underline{U}(0)+\eps^{2} h_{1}(0,\tau,\underline{U}(0))+\eps^{4} h_{2}(0,\tau,\underline{U}(0))+\eps^{6} h_{3}(0,\tau,\underline{U}(0))+\mathcal{O}(\eps^{8}),$$
the initial data $\underline{U}(0)$ will be determined such that it is consistent with the asymptotic order established above.
By the initial restriction
$\Phi_{0}=U(0,0,x),$
the proper $\underline{U}(0)$ can be determined via an iterative procedure:
\begin{equation}\label{equ-2-10-1}
\underline{U}^{[0]}:=\Phi_{0},\ \underline{U}^{[r]}=\Phi_{0}-\sum_{l=1}^r\eps^{2(r-l+1)} h_{l}(0,0,\underline{U}^{[r-l]}),\ r=1,2,3,\  \ \mbox{where}\
\underline{U}^{[r]}-\underline{U}(0)=\mathcal{O}(\eps^{2(r+1)}).
\end{equation}
Therefore, the complete set of initial conditions for the two-scale equation \eqref{equ-TS-1} is formulated as follows:
\begin{equation}\label{init-1}
U(0,\tau,x)=\underline{U}^{[3]}+\eps^{2} h_{1}(0,\tau,\underline{U}^{[2]})+\eps^{4} h_{2}(0,\tau,\underline{U}^{[1]})+\eps^{6} h_{3}(0,\tau,\underline{U}^{[0]})=:U_0(\tau,x),\quad
\tau\in\bT_\tau.
\end{equation}

%\begin{remark}\label{rem-1}
%It is clearly that $G(t,\tau,U(t,\tau,x))$ in two-scale equation \eqref{equ-TS-1} satisfies assumption in Proposition \ref{prop-1} based on the uniform  bound of operators $\Pi_{\pm}$, $\mathcal{D}^{\eps}$.
%\end{remark}

\subsection{Full discretization}\label{sec-2.3}
In this part, we present the full discretization.
To begin, we employ the Fourier spectral method in the spatial $x$ direction (as described in \cite{STW}) for solving \eqref{equ-2-19-1}.
Choose the mesh size $h:=\triangle x=(b-a)/N_{x}$ with $N_{x}$ a positive even integer, and denote grid points as $x_{j}:=a+jh$ for $j=0,1,\ldots,N_{x}$. Specially, the first and second order Fourier differential matrices are introduced as
\begin{equation}\label{equ-2-19-2}
\begin{aligned}
&A_{kj}^{(1)}:=\left\{
\begin{array}{ll}
\frac{(-1)^{k+j}}{2}\cot\big(\frac{(k-j)\pi}{N_{x}}\big),  & k\neq j, \\
0,  & k=j,
\end{array}
\right. \ \
A_{kj}^{(2)}:=\left\{
\begin{array}{ll}
\frac{(-1)^{k+j}}{-2}\sin^{-2}\big(\frac{(k-j)\pi}{N_{x}}\big),   & k\neq j, \\
\frac{N_{x}^{2}}{-12}-\frac{1}{6},   & k=j,
\end{array}
\right.
\end{aligned}
\end{equation}
with $k, j=0,\ldots, N_{x}-1$. Thence the operators $\mathcal{D}^{\eps}$ and $\Pi_{\pm}$ are replaced  by
\begin{align*}
&\mathcal{D}^{\eps}_{f}=\frac{1}{\eps^{2}}\textmd{diag} \Big((I-\eps^{2}A_{kj}^{(2)})^{1/2}-I,(I-\eps^{2}A_{kj}^{(2)})^{1/2}-I\Big),\\
&\Pi_{\pm,f}^{\eps}=\frac{1}{2}\left(\begin{array}{cc} I\pm(I-\eps^{2}A_{kj}^{(2)})^{-1/2} & \mp \ii\eps A_{kj}^{(1)}(I-\eps^{2}A_{kj}^{(2)})^{-1/2} \\
\mp \ii\eps A_{kj}^{(1)}(I-\eps^{2}A_{kj}^{(2)})^{-1/2}  & I\mp(I-\eps^{2}A_{kj}^{(2)})^{-1/2} \\ \end{array}\right),
\end{align*}
where $I=I_{N_{x}\times N_{x}}$ is $N_{x}\times N_{x}$ identity matrix, $\mathcal{D}^{\eps}_{f}$ and $\Pi_{\pm,f}^{\eps}$ are $2\times 2$  matrix blocks and each matrix block is an $N_{x}\times N_{x}$ matrix.
Then the system \eqref{equ-2-19-1} can be discretized into the following ODEs:
\begin{equation}\label{equ-2-20-1}
\frac{\textmd{d}}{\textmd{d}t}W(t)=\big(e^{\frac{\ii t}{\eps^{2}}}e^{\ii t \mathcal{D}^{\eps}_{f}}\Pi_{+,f}^{\eps}+e^{-\frac{\ii t}{\eps^{2}}}e^{-\ii t \mathcal{D}^{\eps}_{f}}\Pi_{-,f}^{\eps}\big)
\Lambda\big(\big(e^{-\frac{\ii t}{\eps^{2}}}e^{-\ii t \mathcal{D}^{\eps}_{f}}\Pi_{+,f}^{\eps}+e^{\frac{\ii t}{\eps^{2}}}e^{\ii t \mathcal{D}^{\eps}_{f}}\Pi_{-,f}^{\eps}\big)W(t) \big),
\end{equation}
where $W(t)=[\psi_{1}(x_{1},t),\ldots,\psi_{1}(x_{N_{x}},t),\psi_{2}(x_{1},t),\ldots,\psi_{2}(x_{N_{x}},t)]^{\intercal}$, $\Lambda(\Phi_{f})=-\ii\big(V_{f}+\mathbf{F}(\Phi_{f}(t))\big)\Phi_{f}(t)$, and $V_f$ denotes the discrete set of potential values ${V(x_j)}_{j=0,1,\ldots,N_{x}-1}$ along the spatial grids.

 For the system \eqref{equ-2-20-1}, according to Subsection \ref{sec-2.2}, its two-scale form  can be expressed as
\begin{equation}\label{equ-2-21-1}
\partial_{t}Z(t,\tau)+\frac{1}{\eps^{2}}\partial_{\tau}Z(t,\tau)=\Upsilon(t,\tau,Z(t,\tau)),
\end{equation}
with
$$
\Upsilon(t,\tau,Z(t,\tau))=\big(e^{\ii\tau}e^{\ii t \mathcal{D}^{\eps}_{f}}\Pi_{+,f}^{\eps}+e^{-\ii\tau}e^{-\ii t \mathcal{D}^{\eps}_{f}}\Pi_{-,f}^{\eps}\big)
\Lambda\big(\big(e^{-\ii\tau}e^{-\ii t \mathcal{D}^{\eps}_{f}}\Pi_{+,f}^{\eps}+e^{\ii\tau}e^{\ii t \mathcal{D}^{\eps}_{f}}\Pi_{-,f}^{\eps}\big)Z(t,\tau) \big).
$$
When $\tau=t/\eps^{2}$, it is obtained that  $Z(t,\tau)=W(t)$.
Let $N_{\tau}$ be an even positive integer  and define $\triangle\tau=\frac{2\pi}{N_{\tau}}$, $\tau_{k}=\frac{2\pi}{N_{\tau}}k$ with $k=0,1,\ldots, N_{\tau}-1$. For any periodic function $v(\tau)$ on $[0,2\pi)$, denote by $P_{\mathcal{M}} : L^{2}([0,2\pi))\rightarrow Y_{\mathcal{M}}$ the standard projection operator
$(P_{\mathcal{M}}v)(\tau)=\sum\limits_{l\in\mathcal{M}}\widehat{v_{l}}e^{\ii l\tau}$ with $l\in \mathcal{M} :=\left\{ -\frac{N_{\tau}}{2}, -\frac{N_{\tau}}{2}+1,\ldots, \frac{N_{\tau}}{2}-1 \right\}$ and by  $I_{\mathcal{M}} : C([0,2\pi))\rightarrow Y_{\mathcal{M}}$ the trigonometric interpolation operator
$(I_{\mathcal{M}}v)(\tau)=\sum\limits_{l\in\mathcal{M}}\widetilde{v_{l}}e^{\ii l\tau}$,
where $\widehat{v_{l}}$ and $\widetilde{v_{l}}$ are respectively defined as
$
\widehat{v}_{l}=\frac{1}{2\pi}\int\limits_{0}^{2\pi}v(\tau)\fe^{-\ii l\tau} d\tau$ and $ \widetilde{v}_{l}=\frac{1}{N_{\tau}}\sum\limits_{k=0}^{N_{\tau}-1}v(\tau_{k})\fe^{-\ii l\tau_{k}}.
$
Then the Fourier spectral method is given in the $\tau$ direction by finding the trigonometric polynomials
$
Z^{\mathcal{M}}(t,\tau)=P_{\mathcal{M}} Z(t,\tau) =\left(Z_{1,j}^{\mathcal{M}}(t,\tau),Z_{2,j}^{\mathcal{M}}(t,\tau)\right)_{j=1,2,\ldots,N_{x}}
$
such that
\begin{align}\label{equ-2-21-4}
\partial_{t}Z^{\mathcal{M}}(t,\tau)+\frac{1}{\eps^{2}}\partial_{\tau}Z^{\mathcal{M}}(t,\tau)=\Upsilon(t,\tau,Z^{\mathcal{M}}(t,\tau)).
\end{align}
Consider  the notation $\widehat{\mathbf{Z}}=[\widehat{\mathbf{Z}}_{1}; \widehat{\mathbf{Z}}_{2}]$ with
$\widehat{\mathbf{Z}}_{m}:=(\widehat{Z}^l_{m,j})_{j=1,2,\ldots,N_{x},l\in\mathcal{M}}$ for  $m=1,2$,
where $(\widehat{Z}^l_{m,j})_{l\in\mathcal{M}}$ are  the Fourier coefficients of $Z^{\mathcal{M}}_{m,j}$.
Then, we obtain
\begin{equation}\label{equ-2-21-2}
\frac{\textmd{d}}{\textmd{d}t}\widehat{\mathbf{Z}}(t)=\ii\Theta\widehat{\mathbf{Z}}(t)+\Xi(t,\widehat{\mathbf{Z}}(t)),
\end{equation}
  where $\widehat{\mathbf{Z}}$ is a $2N_{x}\times N_{\tau}$  dimensional vector,  $\Theta :=\mbox{diag}(\Theta_{1},\Theta_{2},\ldots,\Theta_{2N_{x}})$ with
$\Theta_{1}=\Theta_{2}=\ldots=\Theta_{2N_{x}} :=\frac{1}{\eps^{2}}\mbox{diag}\left(\frac{N_{\tau}}{2},\frac{N_{\tau}}{2}-1,\ldots,-\frac{N_{\tau}}{2}+1 \right)$  and
\begin{equation}\label{equ-2-21-3}
\Xi(t,\widehat{\mathbf{Z}})=\mathfrak{F}\left((-\ii)\mathcal{C}_{-}\left(V_f+F(\mathcal{C}_{+}
\mathfrak{F}^{-1}\widehat{\mathbf{Z}})\right)(\mathcal{C}_{+}\mathfrak{F}^{-1}\widehat{\mathbf{Z}}) \right).
\end{equation}
Here $\mathfrak{F}$ denotes the discrete Fast Fourier Transform (FFT), and
\begin{align*}
&\mathcal{C}_{\pm}:=\mbox{diag}\left(e^{\mp\ii\tau_{k}}e^{\mp\ii t \mathcal{D}^{\eps}_{f}}\Pi_{+,f}^{\eps}+e^{\pm\ii\tau_{k}}e^{\pm\ii t \mathcal{D}^{\eps}_{f}}\Pi_{-,f}^{\eps} \right)_{k=0,1,\ldots,N_{\tau}-1}.
\end{align*}

Based on the Duhamel formula for \eqref{equ-2-21-2}
\begin{equation}\label{equ-2-10-2}
\widehat{\mathbf{Z}}(t_{n}+\xi)=\fe^{\xi \ii\Theta}\widehat{\mathbf{Z}}(t_{n})+\int_{0}^{\xi}\fe^{(\xi-\theta)\ii\Theta}
\Xi(t_{n}+\theta,\widehat{\mathbf{Z}}(t_{n}+\theta))d\theta,\quad 0\leq \xi\leq \triangle t,\quad n\geq0,
\end{equation}
  we adopt the framework of the $s$-stage exponential integrator (\cite{HO06,HO10}):
\begin{equation}\label{equ-2-10-3}
\begin{aligned}
&\widehat{\mathbf{Z}}^{n,j}=\fe^{c_{j}\ii\Theta\triangle t}\widehat{\mathbf{Z}}^{n}+\triangle t
\sum_{k=1}^{s} a_{jk} (\ii\Theta\triangle t)\Xi(t_{n}+c_{k}\triangle t,\widehat{Q}^{n,k}),\quad j=1,2,\ldots,s,\\
&\widehat{\mathbf{Z}}^{n+1}=\fe^{\ii\Theta\triangle t}\widehat{\mathbf{Z}}^{n}+\triangle t
\sum_{j=1}^{s} b_{j} (\ii\Theta\triangle t)\Xi(t_{n}+c_{j}\triangle t,\widehat{Q}^{n,j}),\quad n\geq0,
\end{aligned}
\end{equation}
where $\triangle t$ is a time stepsize, $c_j\in[0,1]$ and $a_{jk}(z), b_j(z)$ are bounded coefficient functions of $z\in \mathbb{C}$ for  $j,k=1,2,\ldots,s$ and $s\in\mathbb{N}_+$.
If the coefficients satisfy the following conditions:
\begin{equation}\label{equ-2-10-5}
\begin{array}[c]{ll}
c_{\rho}=1-c_{s+1-\rho},\ \ b_{\rho}(z)=\fe^{z}b_{s+1-\rho}(-z),\ \
a_{j\rho}(z)=\fe^{c_{j}z}b_{s+1-\rho}(-z)-a_{s+1-j,s+1-\rho}(-z),
\end{array}
\end{equation}
where  $\rho,j=1,2,\ldots,s$, then the integrator \eqref{equ-2-10-3} is time symmetric.

Based on these symmetry conditions and the stiff order conditions of exponential integrators  (\cite{HO06,HO10}),  we construct the following two numerical schemes, which will  be proved to have fourth-order accuracy in the next section.

$\bullet$ \textbf{Symmetric scheme}
We consider $s=3$ and choose the following coefficients
  \begin{equation}\label{sori4 coe}
\begin{aligned} &a_{31}=a_{32}=a_{33}=0,\ \
a_{21}=-\frac{1}{4} \varphi_{2,2} + \frac{1}{2}\varphi_{3,2},\ \ a_{22}=\varphi_{2,2} -\varphi_{3,2},  \ \ a_{23}=\frac{1}{2}\varphi_{1,2}-\frac{3}{4}\varphi_{2,2} +\frac{1}{2}\varphi_{3,2},\\
&a_{11}=b_{1}= 4\varphi_3  - \varphi_2 ,\ \
 a_{12}=b_{2}=4\varphi_2 -8\varphi_3,\ \ a_{13}=b_{3}=\varphi_1-3\varphi_2 +4\varphi_3,
\end{aligned}
\end{equation}
where $\varphi_{i,j}=\varphi_{i,j}(\ii\Theta\triangle t)=\varphi_{i}(c_{j}\ii\Theta\triangle t)$ with  $\varphi_{\rho}(z):=\int_0^1
\theta^{\rho-1}\frac{\fe^{(1-\theta)z}}{(\rho-1)!}d\theta.$ This scheme is symmetric and we denote it by SEP-TS4.

$\bullet$ \textbf{Explicit scheme}  Now we consider explicit schemes  and the
coefficients are given as
\begin{equation} \label{eori4 coe}
\begin{aligned}
&c_{1}=0,  \ c_{2}=c_{3}=c_{5}=\frac{1}{2}, \  c_{4}=1, \ a_{2,1}=\frac{1}{2}\varphi_{1,2}, \  a_{3,1}=\frac{1}{2}\varphi_{1,3}-\varphi_{2,3}, \   a_{3,2}=\varphi_{2,3}, \\
&a_{4,1}=\varphi_{1,4}-2\varphi_{2,4}, \ \qquad  a_{4,2}=a_{4,3}=\varphi_{2,4}, \ \ \qquad  a_{5,1}=\frac{1}{2}\varphi_{1,5}-2a_{5,2}-a_{5,4}, \\
&a_{5,2}=\frac{1}{2}\varphi_{2,5}-\varphi_{3,4}+\frac{1}{2}\varphi_{2,4}-\frac{1}{2}\varphi_{3,5}, \ \ a_{5,3}=a_{5,2}, \ \ a_{5,4}=\frac{1}{4}\varphi_{2,5}-\varphi_{5,2}, \\
&b_{1}=\varphi_{1}-3\varphi_{2}+4\varphi_{3}, \ \ \ b_{2}=b_{3}=0, \ \ \ b_{4}=-\varphi_{2}+4\varphi_{3}, \ \ \ b_{5}=4\varphi_{2}-8\varphi_{3}.
\end{aligned}
\end{equation}
This scheme is not symmetric but is explicit. It is referred as EEP-TS4.

Based on the preceding derivations, we are now prepared to introduce the fully discrete  integrators for solving the original Dirac equation \eqref{equ-1}.

\begin{definition}\label{ts-code}(\textbf{Fully discrete  integrators})
For the nonlinear Dirac equation \eqref{equ-1} in the nonrelativistic regime, choosing the time step size $\triangle t$, positive even number $N_{x}$ and $N_{\tau}$, then the SEP-TS4 $\&$ EEP-TS4 integrators are defined as follows.
\begin{itemize}
  \item Derive the fourth order  initial data $Z(0,\tau)$ by inputting $\Phi_{0}$ into \eqref{init-1} and then set
  $\widehat{\mathbf{Z}}^{0}=[\widehat{\mathbf{Z}}_{1}^{0}; \widehat{\mathbf{Z}}_{2}^{0}]:=\textmd{FFT}\Big(\big(Z^{0}(\tau_k)\big)_{k=0,1,\ldots,N_{\tau}-1}\Big)$.

  \item Compute $\widehat{\mathbf{Z}}^{n+1}$ using EP-TS4 schemes \eqref{equ-2-10-3} with coefficients \eqref{sori4 coe} or \eqref{eori4 coe}, for $n=0,1,\ldots,T/\triangle t-1$.

  \item Then the nemerical solution $\Phi^{n+1}\approx \Phi(t_{n+1},x)$ of the NLDE \eqref{equ-1} is formulated as
\begin{align}\label{equ-3-10-1}
\Phi^{n+1}=\fe^{-\frac{\textmd{i} t_{n+1}}{\eps^{2}}Q^{\eps}}Z^{n+1}, \ n=0,1,\ldots,T/\triangle t-1,
\end{align}
where $Z^{n+1}$ is obtained by $Z^{n+1}=\sum\limits_{l\in\mathcal{M}}(\widehat{\mathbf{Z}}^{n+1})_{l}\fe^{\textmd{i} lt_{n+1}/\eps^{2}}$.
\end{itemize}
 The details  of the whole procedure   are given in  \textbf{Algorithm 2.1}.
\end{definition}

\floatname{algorithm}{Algorithm}
\renewcommand{\algorithmicrequire}{\textbf{Input:}}  % Use Input in the format of Algorithm
\renewcommand{\algorithmicensure}{\textbf{Output:}} % Use Output in the format of Algorithm
\begin{algorithm}[t!]
\caption{ Pseudo-code of fully discrete integrators for the NLDE \eqref{equ-1} }
\label{algorithmofORI}
\begin{algorithmic}[1]
 \Require Initial value $\Phi_{0}(x)$; Time stepsize $\triangle t$; End time $T$.
      \Ensure Numerical solution $\Phi^{n+1}(x)\approx \Phi((n+1)\triangle t,x), n=0:\frac{T}{\triangle t}-1$.
\State {$\underline{U}^{[0]}(x_{j})=\Phi_{0}(x_{j})_{j=0,1,\ldots,N_{x}-1}$}
\For{$r=1:3$}
\State{$\underline{U}^{[r]}=\Phi_{0}-\sum_{l=1}^r\eps^{2(r-l+1)} h_{l}(0,0,\underline{U}^{[r-l]})$.}
\EndFor
 \State {$U(0,\tau,x_{j})=\underline{U}^{[3]}+\eps^{2} h_{1}(0,\tau,\underline{U}^{[2]})+\eps^{4} h_{2}(0,\tau,\underline{U}^{[1]})+\eps^{6} h_{3}(0,\tau,\underline{U}^{[0]}).$}
 \State{$Z^{0}(\tau)=(U_0(\tau,x_{j}))_{j=0,1,\ldots,N_{x}-1}$ $\%$ the initial value of the equation \eqref{equ-2-21-1}}
 \State{$\widehat{\mathbf{Z}}^{0}:=\textmd{FFT}\Big(\big(Z^{0}(\tau_k)\big)_{k=0,1,\ldots,N_{\tau}-1}\Big)$ $\%$ the initial value of the equation \eqref{equ-2-21-2}}
 \For{$n=0:\frac{T}{\triangle t}-1$}
 \For{$j=1:s (s=3, 5)$}
  \State{$\widehat{\mathbf{Z}}^{n,j}=\fe^{c_{j}\ii\Theta\triangle t}\widehat{\mathbf{Z}}^{n}+\triangle t
\sum_{k=1}^{s} a_{jk} (\ii\Theta\triangle t)\Xi(t_{n}+c_{k}\triangle t,\widehat{Q}^{n,k}).$}
 \EndFor
 \State{$\widehat{\mathbf{Z}}^{n+1}=\fe^{\ii\Theta\triangle t}\widehat{\mathbf{Z}}^{n}+\triangle t
\sum_{j=1}^{s} b_{j} (\ii\Theta\triangle t)\Xi(t_{n}+c_{j}\triangle t,\widehat{Q}^{n,j}).$}
 \State{$\Phi^{n+1}=\fe^{\frac{-\ii(n+1)\triangle t}{\eps^{2}}Q^{\eps}}Z^{n+1}((n+1)\triangle t/\eps^{2}).$}
\EndFor
\end{algorithmic}
\end{algorithm}

In practice, the integrals appeared   in the computing   Fourier transform coefficients  are not suitable. Consequently, the  scheme \eqref{equ-2-10-3}  is typically replaced by interpolation schemes in practical computations. Choosing $\mathbf{Z}^{0}=\big(Z^{0}(\tau_k)\big)_{k=0,1,\ldots,N_{\tau}-1}$, then the Fourier pseudospectral (FP) discretization  reads
\begin{align}\label{equ-3-14-1}
Z^{n+1}(\tau)=\sum\limits_{l\in \mathcal{M}}(\widetilde{\mathbf{Z}}^{n+1})_{l}\fe^{\ii l\tau},
\end{align}
where
\begin{equation}\label{equ-3-14-2}
\begin{aligned}
&\widetilde{\mathbf{Z}}^{n,j}=\fe^{c_{j}\ii\Theta\triangle t}\widetilde{\mathbf{Z}}^{n}+\triangle t
\sum_{k=1}^{s} a_{jk} (\ii\Theta\triangle t)\Xi(t_{n}+c_{k}\triangle t,\widetilde{\mathbf{Z}}^{n,k}),\quad j=1,2,\ldots,s,\\
&\widetilde{\mathbf{Z}}^{n+1}=\fe^{\ii\Theta\triangle t}\widetilde{\mathbf{Z}}^{n}+\triangle t
\sum_{j=1}^{s} b_{j} (\ii\Theta\triangle t)\Xi(t_{n}+c_{j}\triangle t,\widetilde{\mathbf{Z}}^{n,j}),\quad n\geq0,
\end{aligned}
\end{equation}
with the coefficients \eqref{sori4 coe} or \eqref{eori4 coe}.

\section{Uniform accuracy and long-time analysis}\label{sec-3}
In this section, we shall study the  uniform accuracy and long-time behaviour of the proposed  integrators.

\subsection{Error estimates}
In order to obtain the error bounds for the  integrators \eqref{equ-2-10-3}, we assume that the electric potential and the exact solution to the NLDE \eqref{equ-9} satisfies:
$$ \norm{V}_{W_{p}^{\nu+m_{0},\infty}}\lesssim 1,\ \
  \norm{\Phi}_{L^{\infty}([0,T];(H_{p}^{\nu+m_{0}})^{2})}\lesssim 1,
$$
where $\nu>d/2+8$, $m_{0}\geq 1$, and $W_{p}^{m,\infty}=\{u| u\in W^{m,\infty}(\Omega),\partial_{x}^{l}u(a)=\partial_{x}^{l}u(b), l=0, \ldots, m-1\}$, $H_{p}^{m}(\Omega)=\{u|u\in H^{m}(\Omega),\partial_{x}^{l}u(a)=\partial_{x}^{l}u(b), l=0, \ldots, m-1\}$ for $m\in\mathbb{N}$.
For simplicity, in this section we use  $A\lesssim B$ to denote that
there exists a generic constant $C>0$ independent of $\eps, \triangle t,\triangle x, \triangle \tau$ such that $A\leq C B$.
\begin{theorem}\label{theo-1} (\textbf{Uniform accuracy})
Under the assumptions stated above, there exist $t_{0}>0$, $h_{0}>0$ and $\tau_{0}>0$ independent of $\eps$ such that for any $0< \eps \leq 1$, when $0<\triangle t\leq t_{0}$, $0<\triangle x\leq h_{0}$ and $0<\triangle \tau\leq \tau_{0}$ ,the global errors
$e_{\Phi}^{n}:=\Phi(t_{n},\cdot)-\Phi^n$  of fully discrete  integrators are bounded by
\begin{equation}\label{equ-2-10-6}
\begin{aligned}
 &\norm{\fe_{\Phi}^{n}}_{H^{1}}  \lesssim \triangle t^4 +\triangle x^{m_{0}}+ \triangle\tau^{\nu-1} ,\qquad \   n=0,1,\ldots, T/\triangle t,
\end{aligned}
\end{equation}
where $\nu>8+d/2$ and $m_{0}\geq 1$.
\end{theorem}
\begin{remark}
The   error estimates in Theorem \ref{theo-1} also hold for the scheme \eqref{equ-3-14-1}-\eqref{equ-3-14-2}, and the proof follows a similar way as given in  \cite{FXY,WZ}.
\end{remark}
\begin{proof}
Here we only prove the global error of the SEP-TS4 integrator, and the proof of EEP-TS4 scheme is similar.

The property of two-scale differential equation will be used in the error analysis and we first study this aspect.
In an analogous way as in \cite{CCLM,LWZ}, the solution of the two-scale system \eqref{equ-TS-1} with the initial value \eqref{init-1}
can be estimated as
$$
\norm{U(t,\tau,x)}_{L_{\tau}^{\infty}(H^{\nu+m_{0}})}\leq C, \ \forall t\in [0,T],
$$
where the constant $C>0$ depends on $\norm{\Phi_{0}}_{L_{\tau}^{\infty}(H^{\nu+m_{0}})}$ but is independent of $\eps$.
Moreover,  the derivatives of $U$ w.r.t. $t$ are bounded as
\begin{equation}\label{der bound}
\norm{\partial_{t}^{k}U(t,\tau,\cdot)}_{L_{\tau}^{\infty}(H^{\nu+m_{0}-2k})}\leq C, \ \forall t\in [0,T], \ k=1,2,3,4.
\end{equation}

Then we study the error brought by the  Fourier spectral method in the spatial $x$.
By estimates on projection error \cite{STW}, we can obtain that
$$%\begin{align}\label{err-1}
 \norm{\Psi(t_{n},x)-W(t_{n})}_{H^{\nu}} \lesssim \triangle x ^{m_{0}}.$$
%\end{align}
In order to study the error brought by Fourier spectral method on $\tau$, define the error function by $$\fe_{Z}^{n}(\tau):=Z(t_{n},\tau)-P_{\mathcal{M}}Z^{n}$$ and the projected error as
$$\fe^{n}_{\mathcal{M}}(\tau):=P_{\mathcal{M}}Z(t_{n},\tau)-Z^{n}_{\mathcal{M}}(\tau),$$ where  we use the notation $Z^{n+1}_{\mathcal{M}}(\tau)=\sum\limits_{l\in\mathcal{M}}(\widehat{Z}^{n+1})_{l}\fe^{\ii l\tau}$.  Since $Z(t_{n},\tau)$ is an argumented function of $W(t_{n})$ and $\norm{W(t_{n})}_{H^{\nu}}=\norm{Z(t_{n},t_{n}/\eps^{2})}_{H^{\nu}}$, which implies that $Z(t_{n},\tau)\in H^{\nu}$. By applying the triangle inequality and estimates on projection error, we  obtain that
\begin{align}\label{err-2}
\norm{\fe^{n}_{Z}}_{H^{1}}&\leq \norm{\fe^{n}_{\mathcal{M}}}_{H^{1}}+
\norm{Z(t_{n},\tau)-P_{\mathcal{M}}Z(t_{n},\tau)}_{H^{1}} \lesssim\norm{\fe^{n}_{\mathcal{M}}}_{H^{1}}+\triangle \tau^{\nu-1}.
\end{align}
Based on the transformation of the equation and the formulation of integrators, we get
\begin{equation}\begin{aligned}\label{err-ff}
\norm{\fe_{\Phi}^{n}}_{H^{1}}&=\norm{\Phi(t_{n},\cdot)-\Phi^n}_{H^{1}}
=\norm{e^{\frac{-\ii t_{n}}{\eps^{2}}Q^{\eps}} \Psi(t_{n},\cdot)-\fe^{\frac{-\ii t_{n}}{\eps^{2}}Q^{\eps}}P_{\mathcal{M}}Z^n}_{H^{1}}\\
&\lesssim \norm{ \Psi(t_{n},\cdot)- P_{\mathcal{M}}Z^n}_{H^{1}} \\
&=\norm{ \Psi(t_{n},\cdot) -W(t_{n})+W(t_{n})-Z(t_{n},t_{n}/\eps^2)+Z(t_{n},t_{n}/\eps^2)-P_{\mathcal{M}}Z^n}_{H^{1}}\\
& \leq \norm{\Psi(t_{n},\cdot)-W(t_{n})}_{H^{1}}+\norm{\fe^{n}_{Z}}_{H^{1}}  \leq \norm{\Psi(t_{n},\cdot)-W(t_{n})}_{H^{\nu}}+\norm{\fe^{n}_{Z}}_{H^{1}} \\
& \lesssim \norm{\fe^{n}_{\mathcal{M}}}_{H^{1}}+\triangle x ^{m_{0}}+\triangle \tau^{\nu-1}.
\end{aligned}
\end{equation}
Therefore, the estimate of error $e_{\Phi}^{n}$ is converted to the estimate of error $\fe^{n}_{\mathcal{M}}$.

For the error  $\fe^{n}_{\mathcal{M}}(\tau) =P_{\mathcal{M}}Z(t_{n},\tau)-Z^{n}_{\mathcal{M}}(\tau)=
\sum\limits_{l\in\mathcal{M}}(\widehat{Z}(t_{n}) -\widehat{Z}^{n})_{l}\fe^{il\tau}$,
we need to study the error $\widehat{\mathbf{Z}}(t_{n}) -\widehat{\mathbf{Z}}^{n}$.
In order to simplify the notation, we set $\widehat{\Gamma}(t)=\Xi(t,\widehat{\mathbf{Z}}(t))$.
Substituting the solution of  \eqref{equ-2-21-2} into \eqref{equ-2-10-3}, we write
\begin{equation}\label{equ-2-10-4}
\begin{aligned}
&\widehat{\mathbf{Z}}(t_{n}+c_{j}\triangle t)=\fe^{c_{j}\ii\Theta\triangle t }\widehat{\mathbf{Z}}(t_{n})+\triangle t
\sum_{k=1}^{3} a_{jk} (\ii\Theta\triangle t )\widehat{\Gamma}(t_{n}+c_{k}\triangle t)+\delta^{n,j},\\
&\widehat{\mathbf{Z}}(t_{n+1})=\fe^{\ii\Theta\triangle t }\widehat{\mathbf{Z}}(t_{n})+\triangle t
\sum_{j=1}^{3} b_{j} (\ii\Theta\triangle t)\widehat{\Gamma}(t_{n}+c_{j}\triangle t)+\delta^{n+1},
\end{aligned}
\end{equation}
where $\delta^{n,j}$ and $\delta^{n+1}$ denote the remainders for $j=1,2,3$ and $n=0,1,\ldots,T/\Delta t$. Define
 \begin{equation*}
\begin{aligned} &  \Psi_{\rho}(z):=\varphi_{\rho}(z)-\sum_{j=1}^{3}b_j(z)\frac{c_j^{\rho-1}}{(\rho-1)!},\ \Psi_{\rho,j}(z):=\varphi_{\rho}(c_jz)c_j^{\rho}-\sum_{k=1}^{3}a_{jk}(z)\frac{c_{k}^{\rho-1}}{(\rho-1)!},
 \end{aligned}
\end{equation*}
for $j=1,2,3$ and $\rho=1,2,\ldots.$
Subtracting \eqref{equ-2-10-4} from \eqref{equ-2-10-2}  with $\xi=\triangle t$ and applying the Taylor expansion to $\widehat{\Gamma}(t)$,  $\delta^{n+1}$ can be derived as
\begin{equation}\label{2-10-5}
\begin{aligned}
\delta^{n+1}=& \triangle t \int_{0}^{1}   \fe^{(1-\theta)\ii\Theta\triangle t}  \sum\limits_{\rho=1}^{4}\frac{(\theta  \triangle t)^{\rho-1}}{(\rho-1)!}\frac{\textmd{d}^{\rho-1}}{\textmd{d} t^{\rho-1}} \widehat{\Gamma} (t_{n}) {\rm d}\theta\\&- \triangle t \sum\limits_{j =1}^{3}   {b}_{j}\big(\ii\Theta\triangle t\big) \sum\limits_{\rho=1}^{4}\frac{ c_{j}^{\rho-1}\triangle t^{\rho-1}}{(\rho-1)!}\frac{\textmd{d}^{\rho-1}}{\textmd{d} t^{\rho-1}}\widehat{\Gamma} (t_{n})+\widetilde{\delta_{4}^{n+1}} \\
=&  \sum\limits_{\rho=1}^{4}\triangle t^\rho \Psi_{\rho}\big(\ii\Theta\triangle t\big)  \frac{\textmd{d}^{\rho-1}}{\textmd{d} t^{\rho-1}}\widehat{\Gamma} (t_{n})+\widetilde{\delta_{4}^{n+1}}.
\end{aligned}
\end{equation}
Here, $\widetilde{\delta_{4}^{n+1}}$ represents the truncation error of the Taylor expansions, which is bounded as $\widetilde{\delta_{4}^{n+1}}=\mathcal{O}(\Delta t^{5})$ based on \eqref{der bound}.
Using the choice \eqref{sori4 coe}, it is easy to check that $\Psi_{\rho}=0$ for $\rho=1,2,3,4$. Thus one gets
$$\delta^{n+1}=\widetilde{\delta_{4}^{n+1}}=\mathcal{O}(\Delta t^{5}).$$
Similarly, analogous results can be derived as
$$\delta^{n,j}=\sum\limits_{\rho=1}^{3}\triangle t^{\rho} \Psi_{\rho,j}\big(\ii\Theta \triangle t\big)   \frac{\textmd{d}^{\rho-1}}{\textmd{d} t^{\rho-1}}\widehat{\Gamma}(t_{n})+\widetilde{\delta_{3}^{n,j}}=\widetilde{\delta_{3}^{n,j}}
$$
with  $\widetilde{\delta_{3}^{n,j}}=\mathcal{O}(\Delta t^{4})$.  Following the similar error estimates of exponential integrators, we can get that the time error   is
$$
\norm{\widehat{\mathbf{Z}}(t_{n}) -\widehat{\mathbf{Z}}^{n}}_{H^{1}}\lesssim\triangle t^{4}.$$
This, together with equation \eqref{err-ff}, gives the global error \eqref{equ-2-10-6}.
\end{proof}

\subsection{Long-time analysis}
Now we turn to the long-time analysis and to this end, consider  the    notations (see \cite{LWZ}):
\begin{equation*}\displaystyle
\begin{array}[c]{ll}
&
\mathbf{k}=\left(k_{-\frac{N_{\tau}}{2},1},\ldots,
k_{\frac{N_{\tau}}{2}-1,1},k_{-\frac{N_{\tau}}{2},2},\ldots,
 k_{\frac{N_{\tau}}{2}-1,2},\ldots,k_{-\frac{N_{\tau}}{2},2N_{x}},\ldots,
k_{\frac{N_{\tau}}{2}-1,2N_{x}}\right),\\
&\displaystyle |\mathbf{k}|=\sum_{l=1}^{2N_{x}}\sum_{j=-N_{\tau}/2}^{N_{\tau}/2-1}
|k_{j,l}|,\ \
 \boldsymbol{\omega}=(\textmd{diagonal elements of}\  \Theta),\ \ \mathbf{k}\cdot \boldsymbol{\omega}=\sum_{l=1}^{2N_{x}}\sum_{j=-N_{\tau}/2}^{N_{\tau}/2-1}
k_{j,l}\omega_{j,l}.
\end{array}
\end{equation*}
Using the notation
$\mathcal{Q}= \{\mathbf{k}\in \mathbb{Z}^{D}:  \textmd{there exists an}\ l\in \{1,\ldots,2N_{x}\}\
 \textmd{such that}\
|\mathbf{k}_{:,l}|=|\mathbf{k}|\}$, we define  the resonance module
$\mathcal{M}= \{\mathbf{k}\in {\mathcal{Q}}:\ \mathbf{k}\cdot
\boldsymbol{\omega}=0\}$. Let $\mathcal{K}$
represent a set of representatives for the equivalence classes in  ${\mathcal{Q}}/ \mathcal{M}$,
selected such that for every  $\mathbf{k}\in\mathcal{K}$, the sum $|\mathbf{k}|$ is   the smallest within its equivalence class $[\mathbf{k}] = \mathbf{k} +\mathcal{M}$, and it is also required that
 $-\mathbf{k}\in\mathcal{K}.$
  For any integer $N>0$, we introduce the following notation:
\begin{equation}\begin{array}{ll}\mathcal{N}_N={\{\mathbf{k}\in\mathcal{K}:  |\mathbf{k}| \leq N \}},\ \ \
\mathcal{N}_N^*=\mathcal{N} \bigcup \{\langle 0\rangle_l\}_{l=1,2,\ldots,2N_{x}},
\end{array} \label{NN}\end{equation}
where $\langle j\rangle_l$ is the unit coordinate vector in  $ \mathbb{R}^{D}$,   represented as $(0, \ldots , 0, 1, 0, \ldots,0)^{\intercal}$  with the sole non-zero entry of $1$ located at the   $(j,l)$-th position. Further details on these notations can be found in \cite{LWZ}.

\begin{theorem}(\textbf{Long time near conservations})\label{Long-time thm}
For the initial value  \eqref{init-1} of the two-scale equation \eqref{equ-TS-1}, it is assumed to have the bound  $0<\delta_0:=\norm{\Phi_{0}}_{L_{\tau}^{\infty}(H^{\nu+m_{0}})}<1$.
Moreover, we assume the following  non-resonance requirement $
|\sin\big(\frac{1}{2}\hh \omega_{j,1}\big)| \geq c_1 \sqrt{\eps}
\   \textmd{for}\     j=-\frac{N_{\tau}}{2},-\frac{N_{\tau}}{2}+1,\ldots,\frac{N_{\tau}}{2}-1,$ and the condition  on the time stepsize: $\hh/\sqrt{\eps} \geq c_2 > 0.$
The symmetric numerical approximation \eqref{equ-3-10-1} produced by  SEP-TS4   has the following time near conservations
\begin{equation}\label{er}
\begin{aligned}
  &\frac{\eps^2}{\delta_0^2} |E(\Phi^{n})- E(\Phi^{0})|= {\mathcal{O}(\eps^{2} \delta_0^4)}+ \mathcal{O}(\delta_{\mathcal{F}}),\  \frac{1}{\delta_0^2}|M(\Phi^{n})- M(\Phi^{0})| = {\mathcal{O}(\eps^{2}\delta_0^4)}+ \mathcal{O}(\delta_{\mathcal{F}}),
\end{aligned}
\end{equation}
where $   0\leq n\hh\leq  \eps^{-2}\delta_0^{-N+5}$, $N$ is an arbitrarily large integer  given in  \eqref{NN}, $\mathcal{O}$ denotes the term whose Euclidean norm is bounded, and $\delta_{\mathcal{F}}$ denotes   the error brought by the Fourier pseudospectral discretizations in $x$ and $\tau$.
The constants symbolized by $\mathcal{O}$ can depend on $c_1, c_2,   N, N_{\tau}, N_{x}$ but  are independent of $n, \Delta t, \eps, \delta_0$.   %Here  $N$ is given in  \eqref{H2} and it is noted that since it can be arbitrarily large,  the energy conservation law holds  for a long time.
\end{theorem}
 \begin{proof}
This result is established within the framework of modulated Fourier expansion (MFE) \cite{hairer2000,hairer2006,WZ}. The proof is provided for the energy conservation property and  the derivation for mass conservation follows similarly.

We initially seek the modulated Fourier expansion (MFE) of the numerical solution $\widehat{\mathbf{Z}}^n$ produced by \eqref{equ-2-10-3}:
\begin{equation}
\begin{array}{ll}
\widehat{\mathbf{Z}}^n=\Phi_{\textmd{MFE}}( t ):= \sum\limits_{\mathbf{k}\in\mathcal{N}^*_{\infty}}
\mathrm{e}^{\mathrm{i}(\mathbf{k} \cdot\boldsymbol{\omega} )
 t }\alpha^{\mathbf{k}}( t ), \ \   t =n \hh,
\end{array} \label{MFE-1}%
\end{equation}
 where $\mathcal{N}^*_{\infty}$  represents the limit of $\mathcal{N}_N^*$ as $N \to +\infty$. Analogous to the analysis of local errors, we demonstrate that the error between $\widehat{\mathbf{Z}}^{n,j}$ and $\Phi_{\textmd{MFE}}( t + c_i\hh )$ is bounded by $\mathcal{O}( \hh^{5}) \Phi_{\textmd{MFE}} ^{(4)}( t + \theta_i \hh)$ for some $\theta_i \in [0, c_i]$ and $i = 1, 2, 3$. Consequently, one has
\begin{equation*} \begin{array}[c]{ll}%
   \widehat{\mathbf{Z}}^{n,i}=\Phi_{\textmd{MFE}}( t +c_i\hh )+C  \hh^5\mathcal{D}^4\Phi_{\textmd{MFE}}( t +\theta_i\hh ),\ \ i=1,2,3,
\end{array}
\end{equation*}
where $C$ denotes an error constant that is independent of $n, \hh, \eps$, and $\mathcal{D}$ represents the differential operator as \cite{hairer2006}.

Upon substituting   \eqref{MFE-1} into  \eqref{equ-2-10-3} and introducing the operator $\mathcal{L}$
\begin{equation*} \begin{array}[c]{ll}\mathcal{L}(\hh \mathcal{D}):=&\big(\fe^{\hh \mathcal{D}} -\fe^{\ii\Theta\triangle t} \big)
\Big(b_{1}(\ii\Theta\triangle t)(\fe^{c_1\hh \mathcal{D}}+C \hh^5 \mathcal{D}^4\fe^{\theta_1\hh \mathcal{D}})+b_{2}(\ii\Theta\triangle t)\\&(\fe^{c_2\hh \mathcal{D}}+C \hh^5 \mathcal{D}^4\fe^{\theta_2\hh \mathcal{D}})+b_{3}(\ii\Theta\triangle t)(\fe^{c_3\hh \mathcal{D}}+C \hh^5 \mathcal{D}^4\fe^{\theta_3\hh \mathcal{D}})\Big)^{-1},
\end{array}\end{equation*}
we obtain $\mathcal{L}(\hh \mathcal{D}) \Phi_{\textmd{MFE}}( t )=\hh f(\Phi_{\textmd{MFE}})$, where $f(\Phi_{\textmd{MFE}})=\Xi(t,\Phi_{\textmd{MFE}})$.
By expanding the nonlinear term into its Taylor series, we derive
\begin{equation*}
\begin{aligned}&\mathcal{L}(\hh \mathcal{D})\Phi_{\textmd{MFE}}( t )
=\hh    \sum\limits_{\mathbf{k}\in\mathcal{N}^*_{\infty}}\mathrm{e}^{\mathrm{i}(\mathbf{k}
\cdot \boldsymbol{\omega})  t } \sum\limits_{m\geq
1}\frac{ f  ^{(m)}(0)
}{m!}
\sum\limits_{\mathbf{k}^1+\ldots+\mathbf{k}^m=\mathbf{k}} \Big[\alpha^{\mathbf{k}^1}\cdot\ldots\cdot
\alpha^{\mathbf{k}^m}\Big]( t ).
\end{aligned} %
\end{equation*}
In the subsequent development, we establish the modulation system for the coefficients $\alpha^{\mathbf{k}}( t )$. This is accomplished by substituting the ansatz given in   \eqref{MFE-1} into the relevant expressions and equating the coefficients associated with $\mathrm{e}^{\mathrm{i}(\mathbf{k} \cdot\boldsymbol{\omega})  t }$. Consequently, we arrive at the following results:
\begin{equation*}
\begin{aligned}
&\mathcal{L}(\hh \mathcal{D}+\mathrm{i}(\mathbf{k} \cdot\boldsymbol{\omega} )
\hh)\alpha^{\mathbf{k}}( t )=\hh   \sum\limits_{m\geq
1}\frac{f^{(m)}(0)}{m!}\sum\limits_{\mathbf{k}^1+\ldots+\mathbf{k}^m=\mathbf{k}} \Big[\alpha^{\mathbf{k}^1}\cdot\ldots\cdot
\alpha^{\mathbf{k}^m}\Big]( t ).
\end{aligned} %
\end{equation*}
 This constitutes the modulation system for the coefficients $\alpha^{\mathbf{k}}( t )$ of the modulated Fourier expansion (MFE). Based on this result and similar arguments given in   \cite{hairer2000,hairer2006,WZ},  we demonstrate that the numerical solution derived from the SEP-TS4 scheme can be represented by   \begin{equation}\begin{array}{ll}
\widehat{\mathbf{Z}}^n=\sum\limits_{\mathbf{k}\in\mathcal{N}_N^*}
\mathrm{e}^{\mathrm{i}(\mathbf{k} \cdot \boldsymbol{\omega})
 t }\alpha^{\mathbf{k}}( t )+R_{N}( t ),
           \end{array} \label{MFE-ERKN-0}%
\end{equation}
where $\alpha^{\mathbf{k}}$ are  smooth coefficient functions at $ t =n\hh$ and the term $R_{N}( t )$ represents the remainder resulting from the truncation process. The coefficient functions $\alpha^{\mathbf{k}}$ are bounded by
\begin{equation}%
\begin{array}{ll}
  \alpha_{j,l}^{\langle j\rangle_l}( t )=\mathcal{O}(\delta_0),\ \
  \dot{\alpha}_{j,l}^{\langle j\rangle_l}( t )=\mathcal{O}(\eps^{5/2}   \delta_0),\ \
 \alpha_{j,l}^{\mathbf{k}}( t )=\mathcal{O}\big(\eps^{2}   \min(\delta_0^{|\mathbf{k}|},\delta_0^{3}) \big),\ \    {\mathbf{k}\neq \langle j\rangle_l,}
\end{array} %
\label{coefficient func}%
\end{equation}
where $j\in \{-N_\tau/2,-N_\tau/2+1,\ldots,N_\tau/2-1\}$ and  $l=1,2,\ldots,2N_{x}$.
  Standard convergence estimates (similar to the proof of Theorem \ref{theo-1}) yield a  bound for the remainder term $R_{N}$:
\begin{equation}\begin{array}{ll}
&R_{N}( t )=\mathcal{O}( t   \eps^2\delta_0^{N+1}).
\end{array}\label{remainder}%
\end{equation}
%where the constant  symbolised by   $\mathcal{O}$ is independent of $n, \hh, \eps$, but depends on the constants $N_{\tau},c$ appeared in Theorem \ref{Long-time thm}.

The construction of the coefficients $\alpha^{\mathbf{k}}$ along with the remainder bound given by \eqref{remainder} yields an almost-invariant quantity. Furthermore, this almost-invariant is demonstrated to closely approximate the energy \eqref{equ-7}, with the magnitude of the discrepancy governed by the bound \eqref{coefficient func} on $\alpha^{\mathbf{k}}$. The analysis of this aspect is similar to  that of \cite{LWZ,WZ23}, and hence we only state the results here.

The synthesis of $\alpha^{\mathbf{k}}$ gives rise to an almost-invariant $\mathcal{H}( t )$, which exhibits small variation as it evolves from $0$ to $t$
 $$\mathcal{H}( t )=
\mathcal{H}(0)+\mathcal{O}( t   \eps^2\delta_0^{N+1}).$$
Furthermore, this almost-invariant is closely aligned with the energy $E$, and their  relationship is given by
$$
\mathcal{H}( t _n)=  E(\Phi^{n}) +\mathcal{O}(  \delta_0^6)+\mathcal{O}(  \delta_{\mathcal{F}}).
$$
Having established the preceding preparations,  we are now in a position to deduce the error in energy conservation.
\begin{equation*}
\begin{aligned}
  E(\Phi^{n})&= \mathcal{H}( t _n) +\mathcal{O}(  \delta_0^6)+ \mathcal{O}(  \delta_{\mathcal{F}}) %\\&=   \mathcal{H}( t _{n-1})+\hh\mathcal{O}( \eps^{2}\delta_0^{N+1})+\mathcal{O}(\eps^{4(2-\alpha)} \delta_0^4)+  \mathcal{O}(  \delta_{\mathcal{F}})\\
 =   \mathcal{H}( t _{n-1})+\hh\mathcal{O}(\eps^{2}\delta_0^{N+1})+\mathcal{O}( \delta_0^6)+  \mathcal{O}(  \delta_{\mathcal{F}})\\
  &=   \mathcal{H}( t _{n-2})+2\hh\mathcal{O}(\eps^{2}\delta_0^{N+1})+\mathcal{O}( \delta_0^6)+  \mathcal{O}(  \delta_{\mathcal{F}})=\ldots\\
  &=    \mathcal{H}( t _{0})+n\hh\mathcal{O}(\eps^{2}\delta_0^{N+1})+\mathcal{O}( \delta_0^6)+  \mathcal{O}(  \delta_{\mathcal{F}})= E(\Phi^{0})+\mathcal{O}(  \delta_0^6)+ \mathcal{O}(  \delta_{\mathcal{F}}),
\end{aligned}
\end{equation*}
as long as $n\hh \eps^{2}\delta_0^{N+1} \leq \delta_0^6$. The demonstration of the nearly conserved energy is thereby concluded.
\end{proof}
\section{Numerical experiments}\label{sec-4}
In this section, we apply the proposed two fourth order uniformly accurate integrators to different kinds of Dirac equation. We will test the accuracy and long-term energy and mass performance of each integrator.
\subsection{One-dimensional Dirac equation} \label{sec-1d}
 We first consider the equation \eqref{equ-1} with  one dimension.
We take a bounded domain $\Omega=(a,b)$ and assume periodic boundary conditions.% Let $\triangle t>0$ be the time step, and denote $t_{n}=n\triangle t$ for $n=0,1,\ldots$. The mesh size is chose as $h:=\triangle x=\frac{b-a}{N_{x}}$ and $\triangle\tau=2\pi/N_{\tau}$ with $N_{x}$ and $N_{\tau}$ being two even positive integers, then the grid points in $x$ and $\tau$ can be denoted as $
% x_{j}:=a+jh, \ j=0,1,\ldots,N; \ \tau_{j}=j\triangle\tau, \ j=0,1,\ldots,N_{\tau}.$

\vspace{0.15cm}
 \textbf{Problem 1. Nonlinear Case.}
In this example,  we consider the bounded domain $\Omega=(-32,32)$ and the electric potential
$
V(x)=\frac{x-1}{x^{2}+1}.$
The nonlinearity is defined as
$
\mathbf{F}(\Phi)=(\Phi^{*}\sigma_{3}\Phi)\sigma_{3}$
and  the initial data $\Phi_{0}=(\phi_{1},\phi_{2})$ in \eqref{equ-3}  is selected  as
$
\phi_{1}(0,x)=\fe^{-\frac{x^{2}}{2}},    \phi_{2}(0,x)=\fe^{-\frac{(x-1)^{2}}{2}}.
$

\begin{figure}[t!]
$$\begin{array}{cc}
\psfig{figure=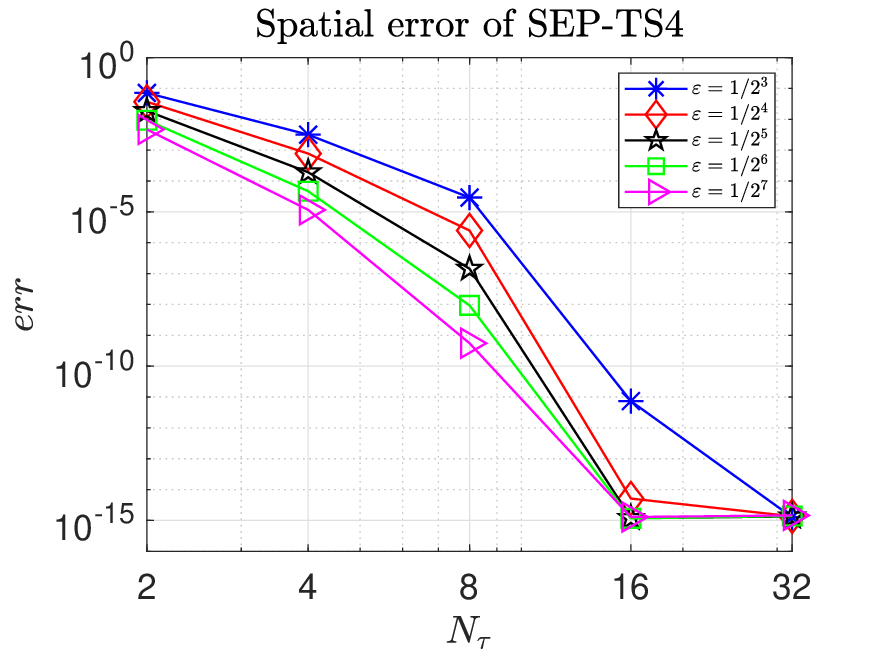,height=3.0cm,width=6.0cm}
\psfig{figure=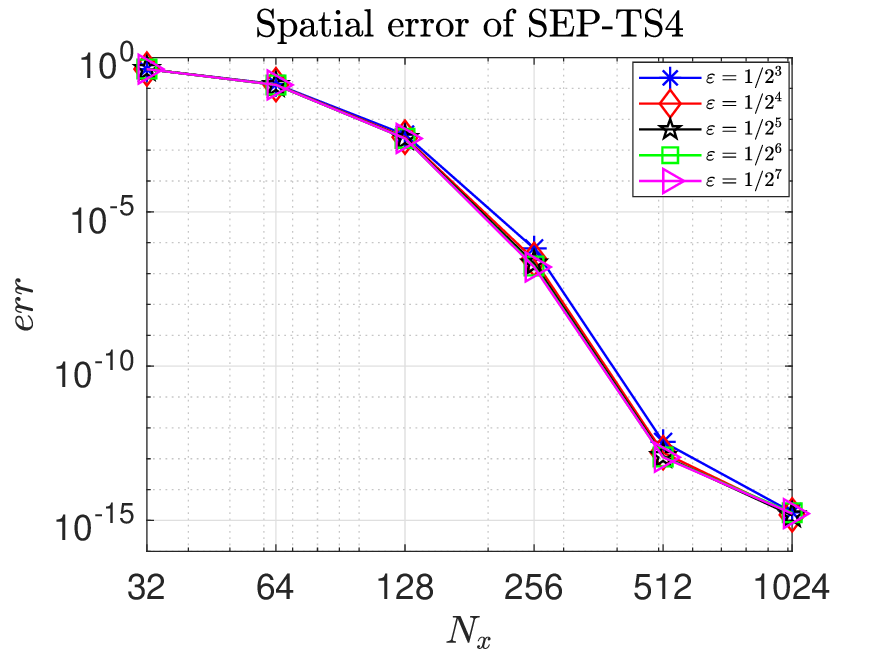,height=3.0cm,width=6.0cm}
\end{array}$$
\caption{Problem 1. Spatial error of NLDE \eqref{equ-1} in 1D at $t=1$ under different $\eps$  in $\tau$-direction  (left) and   in $x$-direction  (right).}\label{fig-1}
\end{figure}

We first test the spatial discretization error of the proposed SEP-TS4 which is given in Figure \ref{fig-1}. Here  the reference solution is obtained by SEP-TS4 with $\triangle t=0.01$, $N_{x}=2^{11}$  and $N_{\tau}=2^{8}$. The results of EEP-TS4 are similar and are omitted here. From these results, it can be observed that the spatial discritization in the $x$ and $\tau$ directions has spectral accuracy  for all $\eps\in (0,1]$. In the following experiments, we fixed $N_{x}=2^{10}$ and $N_{\tau}=2^{6}$, which is enough to achieve machine accuracy.

\begin{figure}[t!]
$$\begin{array}{cc}
\psfig{figure=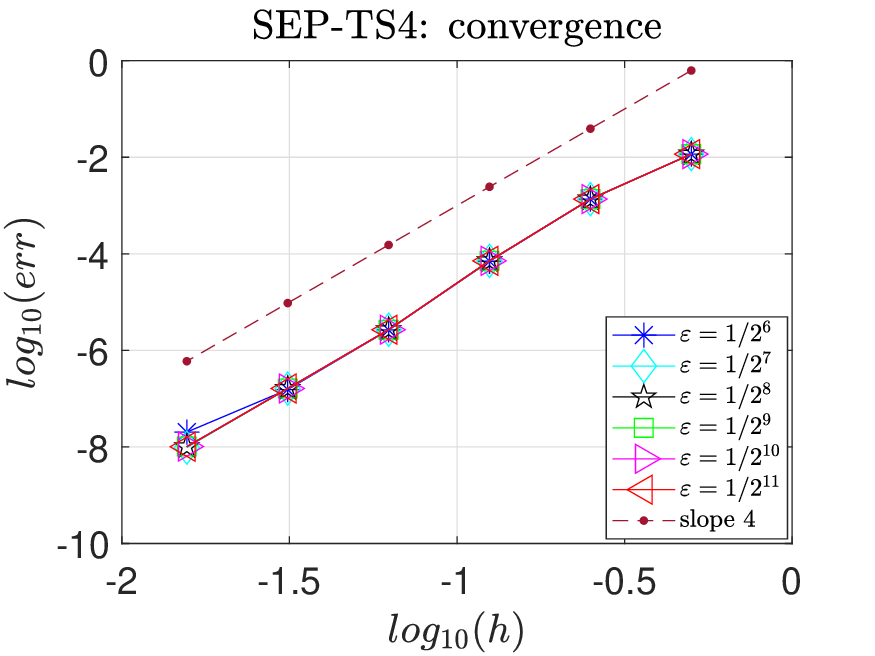,height=3cm,width=6cm}
\psfig{figure=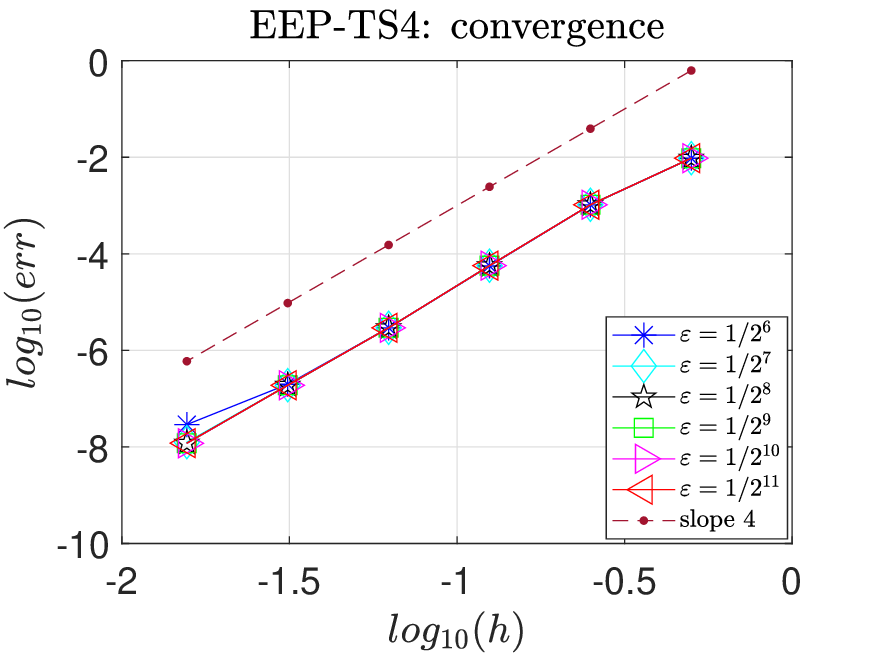,height=3cm,width=6cm}
\end{array}$$
\caption{Problem 1. Temporal error of NLDE \eqref{equ-1} in 1D at $t=1$ under different $\eps$.}\label{fig-2}
\end{figure}

Subsequently, we test the temporal errors, with the exact solution obtained via the second-order Strang splitting method (\cite{BCY}) using an ultra-fine time step   $10^{-6}$. The temporal error
$
err=\frac{\norm{\Phi^{n}-\Phi(t_{n},\cdot)}{l^{\infty}}}{\norm{\Phi(t_{n},\cdot)}_{l^{\infty}}}
$
at $t=1$ for varying $\eps$  are presented in Figure \ref{fig-2}. The numerical results show that the temporal errors for both SEP-TS4 and EEP-TS4 exhibit a fourth-order convergence rate, in alignment with the theoretical bounds \eqref{equ-2-10-6}.

\begin{figure}[t!]
$$\begin{array}{cc}
\psfig{figure=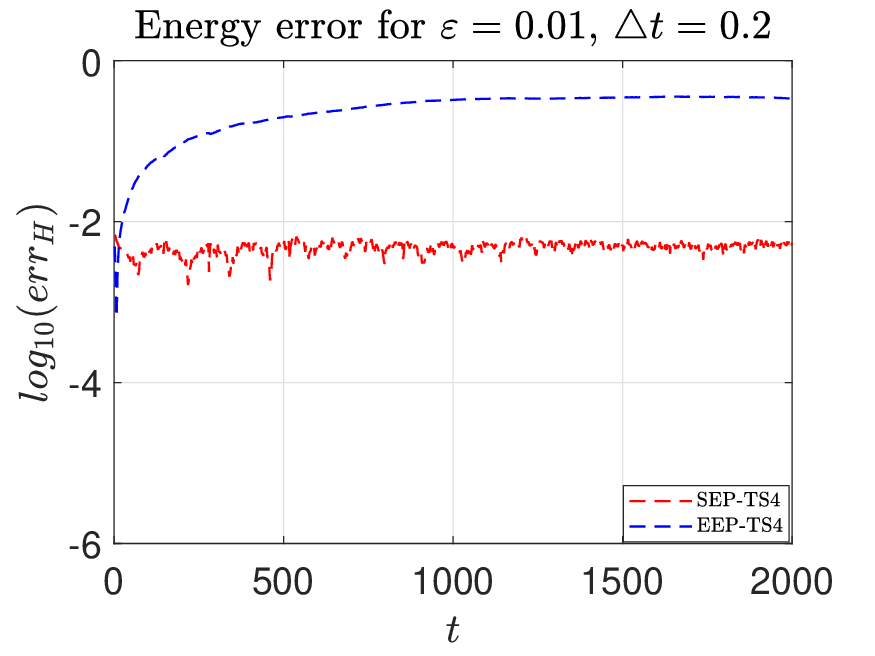,height=3cm,width=5cm}
\psfig{figure=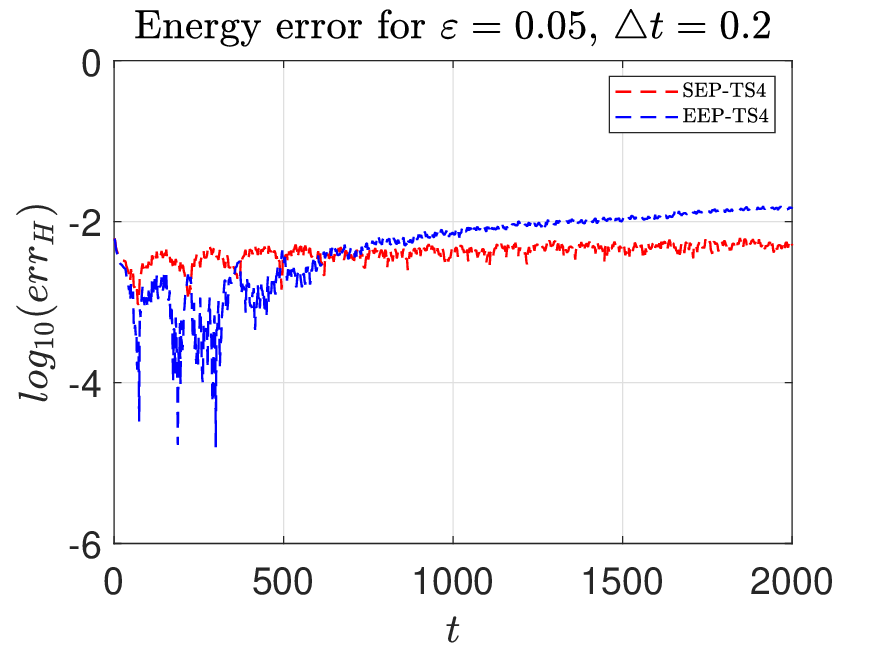,height=3cm,width=5cm}
\psfig{figure=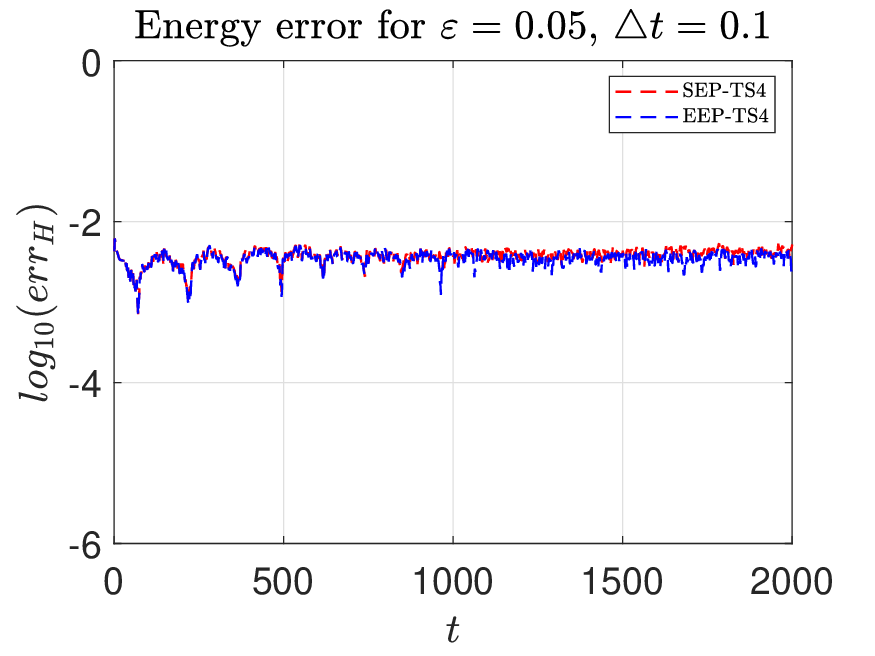,height=3cm,width=5cm}\\
\psfig{figure=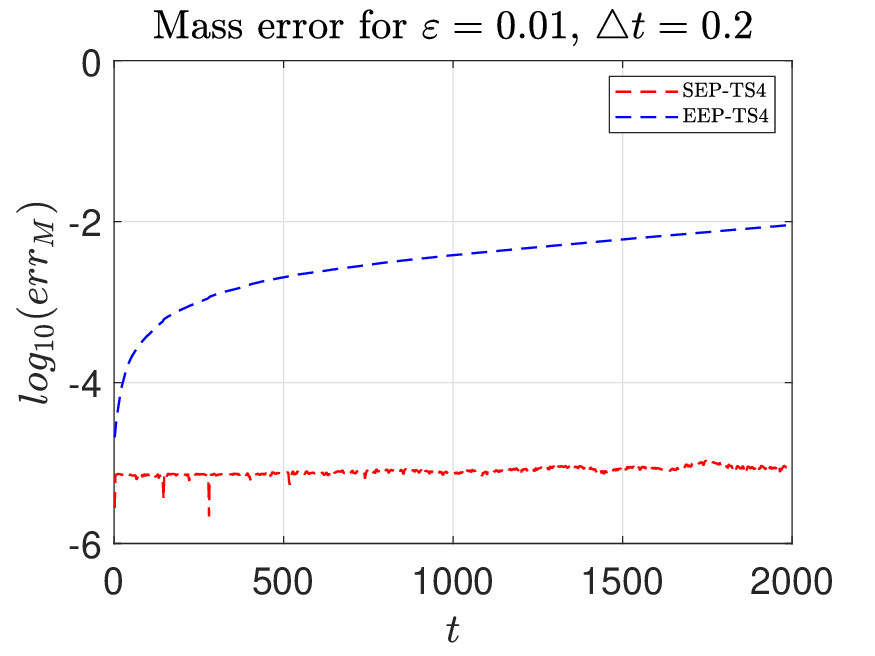,height=3cm,width=5cm}
\psfig{figure=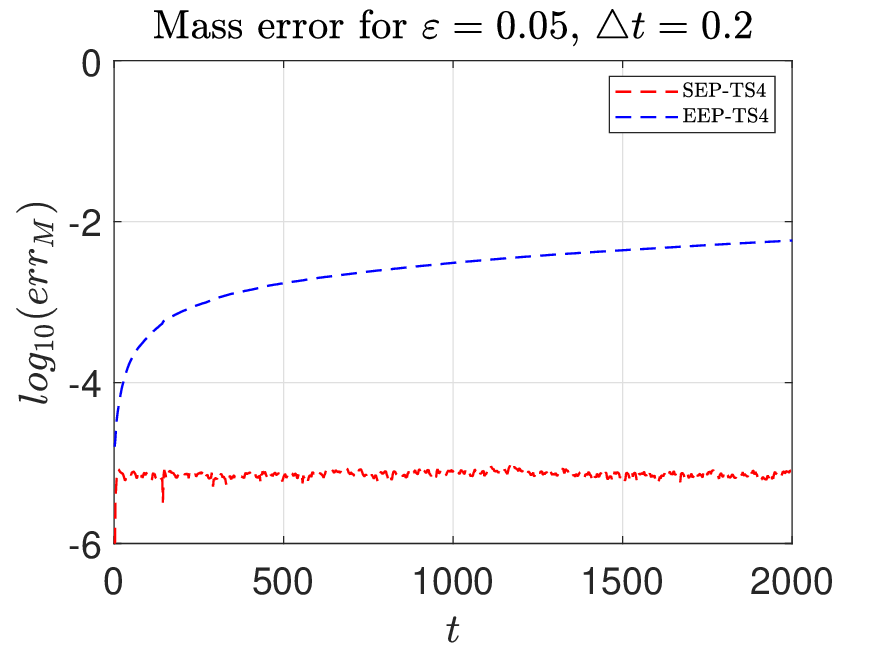,height=3cm,width=5cm}
\psfig{figure=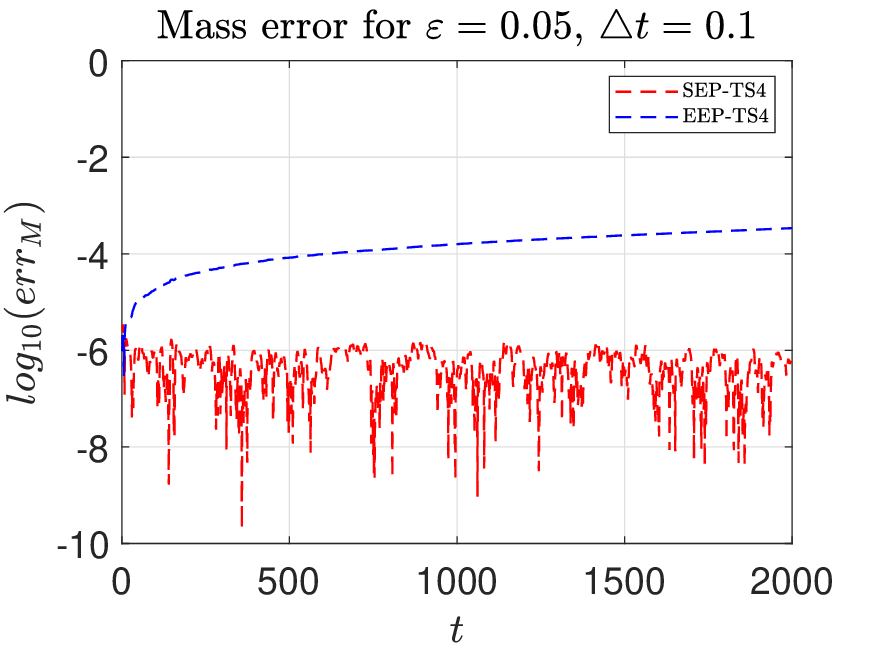,height=3cm,width=5cm}
\end{array}$$
\caption{Problem 1. Energy error (top) and mass error (bottom) of NLDE \eqref{equ-1} in 1D   under different $\eps$ and $\triangle t$.}\label{fig-3}
\end{figure}

Then we study the long-term performance of the EP-TS4 integrators by investigating their numerical energy $H(\Phi^{n})$ and mass $M(\Phi^{n})$. The conservation errors
$
err_{H}=\frac{|H(\Phi^{n})-H(\Phi^{0})|}{|H(\Phi^{0})|}$ and $err_{M}=\frac{|M(\Phi^{n})-M(\Phi^{0})|}{|M(\Phi^{0})|}
$
 are displayed  in Figure \ref{fig-3}. The result clearly illustrates that the symmetric SEP-TS4 integrator maintains  perfect long-term conservation properties, whereas the non-symmetric EEP-TS4 integrator exhibits a noticeable drift in both energy and mass.

%\begin{figure}[t!]
%$$\begin{array}{cc}
%\psfig{figure=N1-energy-err-3,height=3cm,width=6cm}
%\psfig{figure=N1-mass-err-3,height=3cm,width=6cm}\\
%\psfig{figure=N1-energy-err-4,height=3cm,width=6cm}
%\psfig{figure=N1-mass-err-4,height=3cm,width=6cm}
%\end{array}$$
%\caption{Energy error (left) and mass error (right) of EP-TSs at $t_{n}=2000$ with $\eps=0.05$ under different $\triangle t$.}\label{fig-8}
%\end{figure}

%In this example, we choose the bounded domain $\Omega=(-32,32)$, the same electric potential   and space discritization $N_{x}=2^{10}$, $N_{\tau}=2^{6}$ like Example 1.We take nonlinearity
%$\mathbf{F}(\Phi)=(\Phi^{*}\sigma_{3}\Phi)\sigma_{3}+|\Phi|^{2}I_{2},$
%i.e., $\lambda_{1}=1$, $\lambda=1$, and the initial data $\Phi_{0}=(\phi_{1},\phi_{2})$ in \eqref{equ-3} is given as
%$
%\phi_{1}(0,x)=\fe^{-\frac{x^{2}}{2}}, \quad \phi_{2}(0,x)=\fe^{-\frac{3(x-1)^{2}}{2}}, \ x\in\mathbb{R}.$
%Figure \ref{fig-4} exhibit the corresponding energy and mass error with $\triangle t=0.1$ under different $\eps=0.01/5$ and $\eps=0.01/10$.

%\begin{figure}[t!]
%$$\begin{array}{cc}
%\psfig{figure=N11-energy-err-1.eps,height=3cm,width=6cm}
%\psfig{figure=N11-mass-err-1.eps,height=3cm,width=6cm} \\
%\psfig{figure=N11-energy-err-2.eps,height=3cm,width=6cm}
%\psfig{figure=N11-mass-err-2.eps,height=3cm,width=6cm}
%\end{array}$$
%\caption{Energy error (left) and mass error (right) of EP-TSs at $t_{n}=2000$ with $\triangle t=0.1/2$ under different $ep$.}\label{fig-4}
%\end{figure}

\vspace{0.15cm}
 \textbf{Problem 2. Dynamics of traveling waves.}
When $\eps=1$, $V(x)=0$, $\lambda_{2}=0$, the NLDE \eqref{equ-1} in 1D with $\lambda_{1}=-1$ is
\begin{align}\label{equ-3-6-3}
i\partial_{t}\Phi=\bigg(-\frac{\ii}{\varepsilon}\sigma_{1}\partial_{x}+\frac{1}{\varepsilon^2}\sigma_{3}\bigg)\Phi
  -\Phi^{*}\sigma_{3}\Phi\sigma_{3}\Phi, \ \mathbf{x}\in\mathbb{R},
\end{align}
which admits soliton solutions with velocity $v$ initially placed as $x_{0}$ as
\begin{equation}\label{equ-3-6-1}
\begin{matrix}
\Phi^{ss}(x-x_{0},t)=\left(\begin{array}{cc} \sqrt{\frac{\gamma+1}{2}} & \text{sign}(v)\sqrt{\frac{\gamma-1}{2}} \\
\text{sign}(v)\sqrt{\frac{\gamma-1}{2}} & \sqrt{\frac{\gamma+1}{2}} \\ \end{array}\right)\Phi^{sw}(\tilde{t},\tilde{x}).
\end{matrix}
\end{equation}
Here $\gamma=1/\sqrt{1-v^{2}}$, $\tilde{x}=\gamma(x-x_{0}-vt)$, $\tilde{t}=\gamma(t-v(x-x_{0}))$, and $\Phi^{sw}$ is the standing wave defined as
\begin{equation}\label{equ-3-6-2}
\begin{matrix}
\Phi^{sw}(x,t)\equiv\left(\begin{array}{c} \phi_{1}^{sw}(x,t) \\ \phi_{2}^{sw}(x,t) \end{array}\right)
=\left(\begin{array}{c} A(x) \\ iB(x) \end{array}\right)\fe^{-\ii\omega t}, \ 0<\omega\leq1,
\end{matrix}
\end{equation}
where
\begin{align*}
A(x)=\frac{\sqrt{-\frac{2}{\lambda_{1}}(1-\omega^{2})(1+\omega)}\cosh(\sqrt{1-\omega^{2}}x)}{1+\omega\cosh(2\sqrt{1-\omega^{2}}x)}, \ \
B(x)=\frac{\sqrt{-\frac{2}{\lambda_{1}}(1-\omega^{2})(1-\omega)}\cosh(\sqrt{1-\omega^{2}}x)}{1+\omega\cosh(2\sqrt{1-\omega^{2}}x)}.
\end{align*}

\begin{figure}[t!]
$$\begin{array}{cc}
\psfig{figure=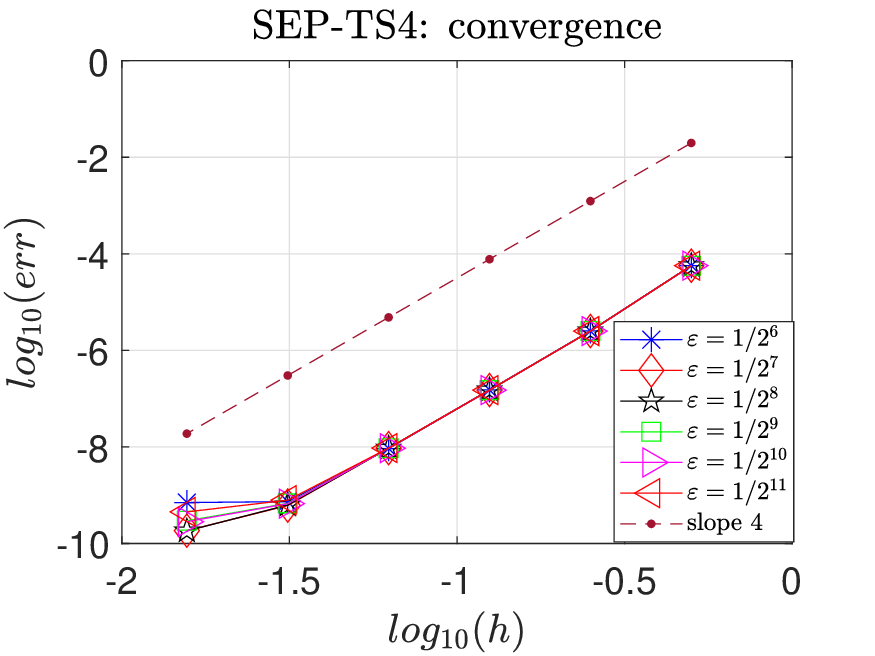,height=3cm,width=6cm}
\psfig{figure=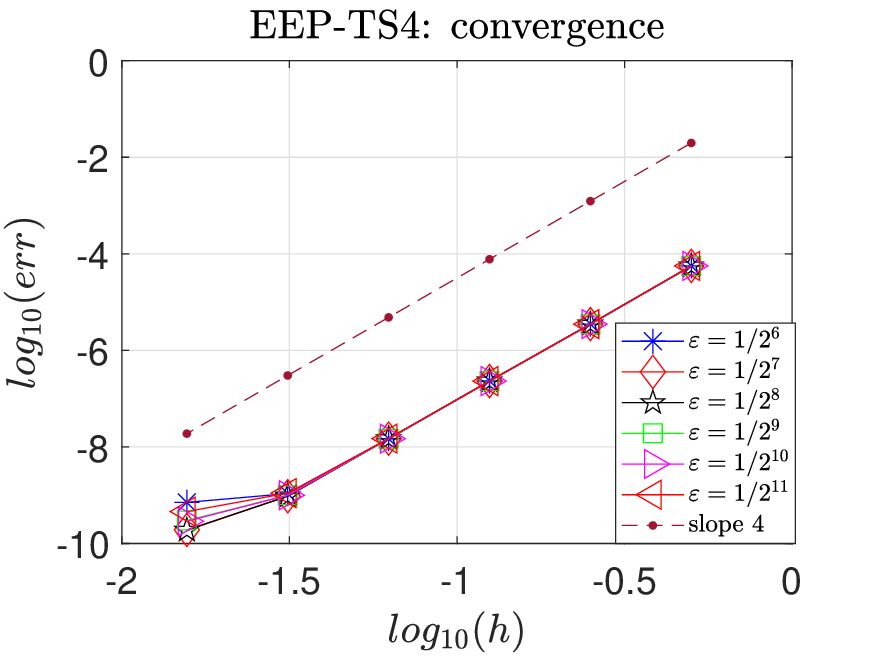,height=3cm,width=6cm}
\end{array}$$
\caption{Problem 2. Temporal error of NLDE \eqref{equ-3-6-3} in 1D at $t=1$ under different $\eps$.}\label{fig-3-10-5}
\end{figure}

\begin{figure}[t!]
$$\begin{array}{cc}
\psfig{figure=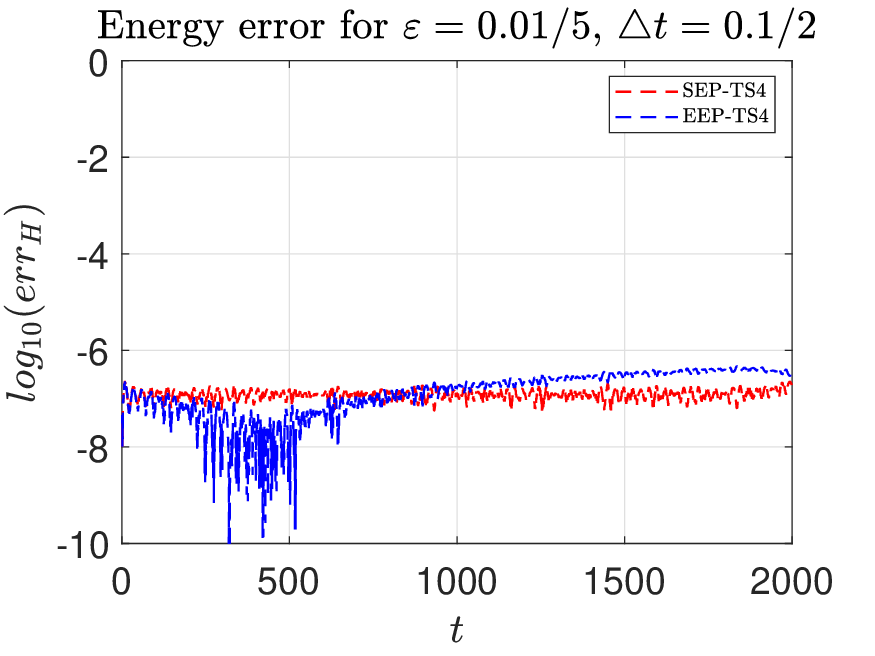,height=3cm,width=5cm}
\psfig{figure=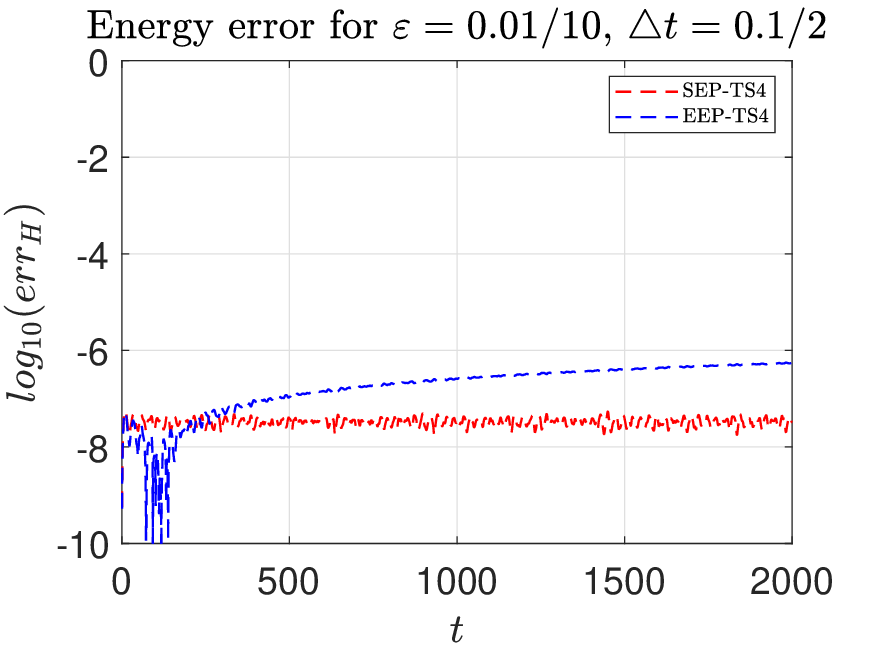,height=3cm,width=5cm}
\psfig{figure=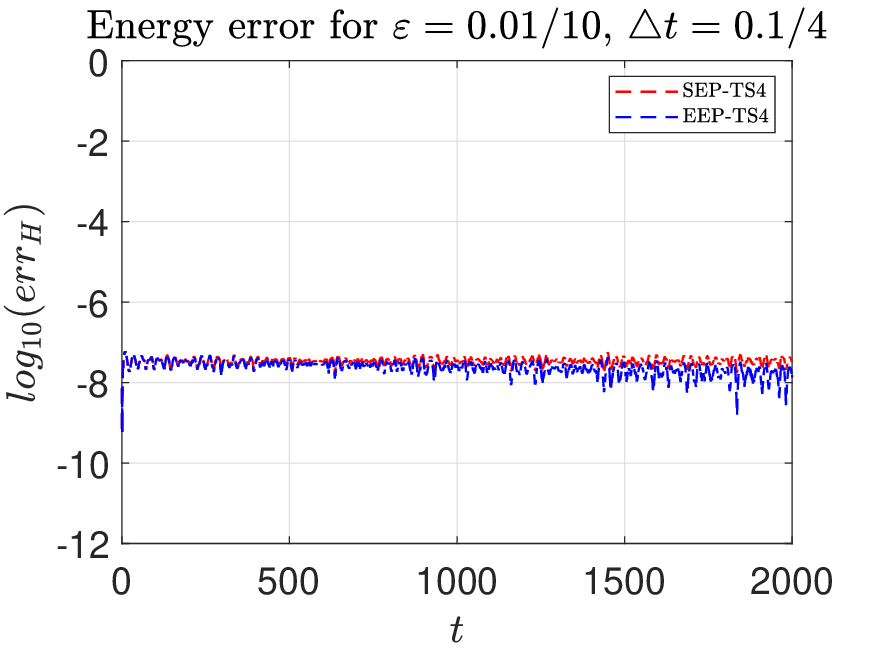,height=3cm,width=5cm}\\
\psfig{figure=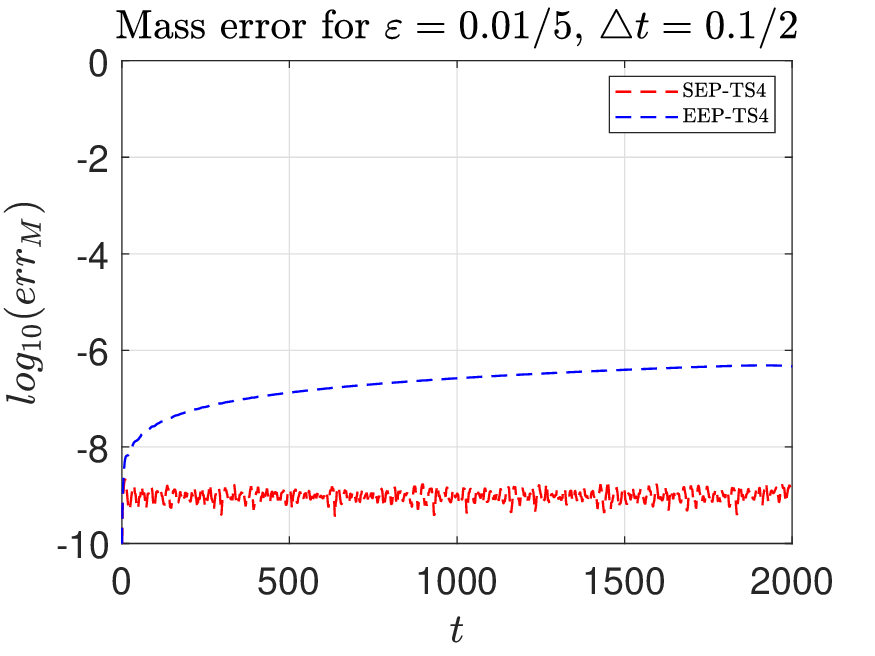,height=3cm,width=5cm}
\psfig{figure=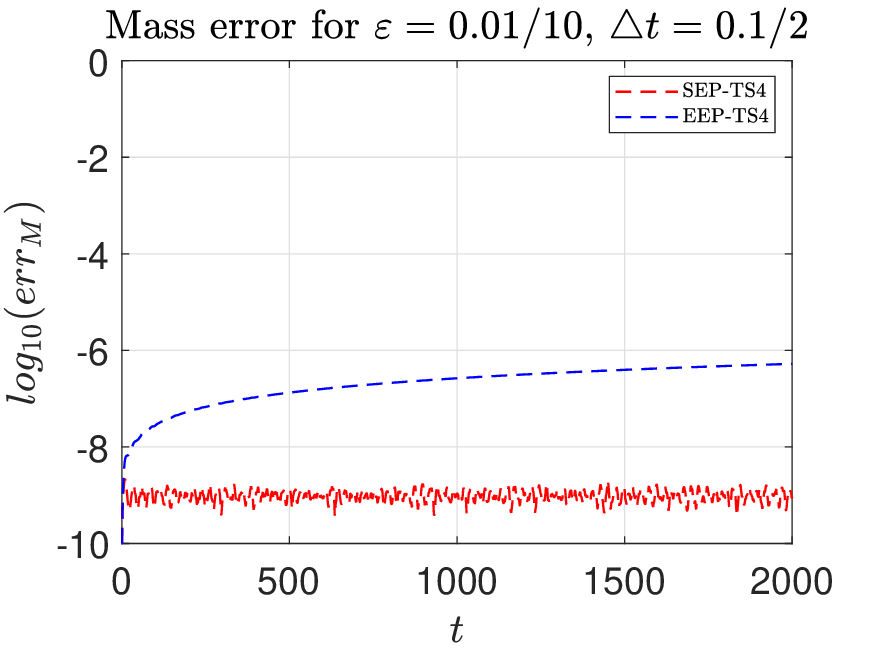,height=3cm,width=5cm}
\psfig{figure=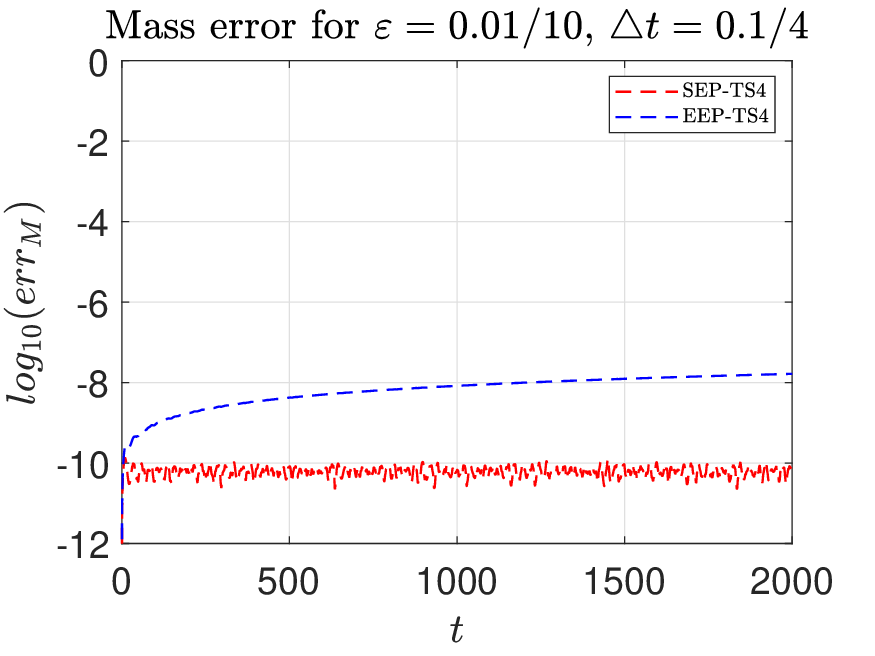,height=3cm,width=5cm}
\end{array}$$
\caption{Problem 2. Energy error (top) and mass error (bottom) of NLDE \eqref{equ-3-6-3} in 1D  with different $\eps$ and $\triangle t$.}\label{fig-3-10-6}
\end{figure}

\begin{figure}[t!]
$$\begin{array}{cc}
\psfig{figure=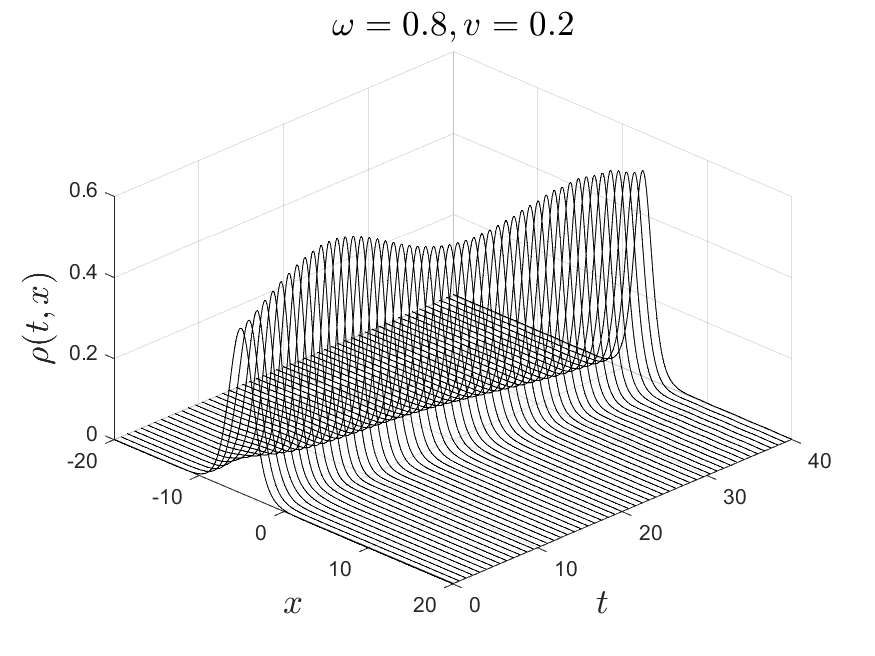,height=3.5cm,width=6cm}
\psfig{figure=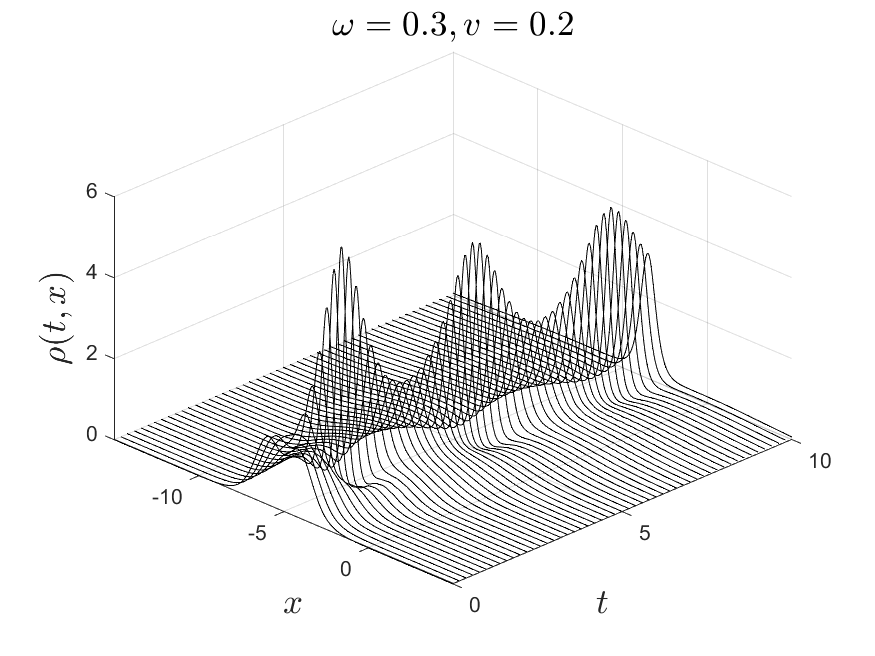,height=3.5cm,width=6cm}
\end{array}$$
\caption{Problem 2. Evolution of the density $\rho(t,x)=|\Phi(t,x)|^{2}$ of NLDE \eqref{equ-3-6-3} initially one-humped (left) and two-huped waves (right) with $\eps=0.01$.}\label{fig-3-6-1}
\end{figure}
In this study, we adopt these soliton solutions as initial conditions to investigate the dynamics of the Dirac equation \eqref{equ-3-6-3} in the nonrelativistic limit with $\eps\ll 1$. All numerical simulations are conducted on a bounded domain $\Omega=(-32,32)$ with spatial mesh size $N_{x}=2^{10}$, $N_{\tau}=2^{6}$.

It is well-established that soliton solutions exhibit two distinct profiles \cite{ST,XST}: the one-humped soliton for $\omega\in[1/2,1)$ and the two-humped soliton for $\omega\in(0,1/2)$. Firstly, the temporal errors and long term performances of Dirac equation \eqref{equ-3-6-3} are respectively  given in Figures \ref{fig-3-10-5} and \ref{fig-3-10-6}, where we take one-humped soliton solutions \eqref{equ-3-6-1} with $\omega=0.8$, $x_{0}=-5, v=0.2$ as initial conditions. Fourth order uniform accuracy of the proposed integrators and long time conservations for the symmetric integrator are clearly demonstrated.

%In the following density evolution diagrams, we uniformly take the step size as $\triangle t=0.1$. In Figure \ref{fig-3-6-1}, we take $\omega=0.8$ and $\omega=0.3$ with $x_{0}=-5, v=0.2$ to investigate how a single wave evolutes in the nonrelativistic regime. As shown in the figures, both waves propagate to the right at oscillating heights, and one-humped waves propagate in a similar one-humped profile, while the profile of two-humped waves changes over the course of their evolution.

\begin{figure}[t!]
$$\begin{array}{cc}
\psfig{figure=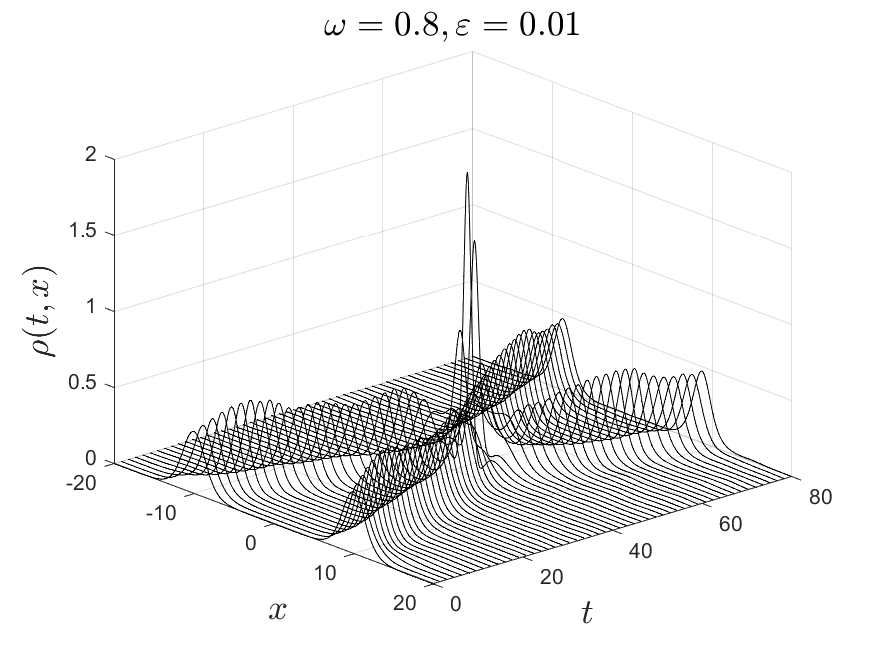,height=3.5cm,width=6cm}
\psfig{figure=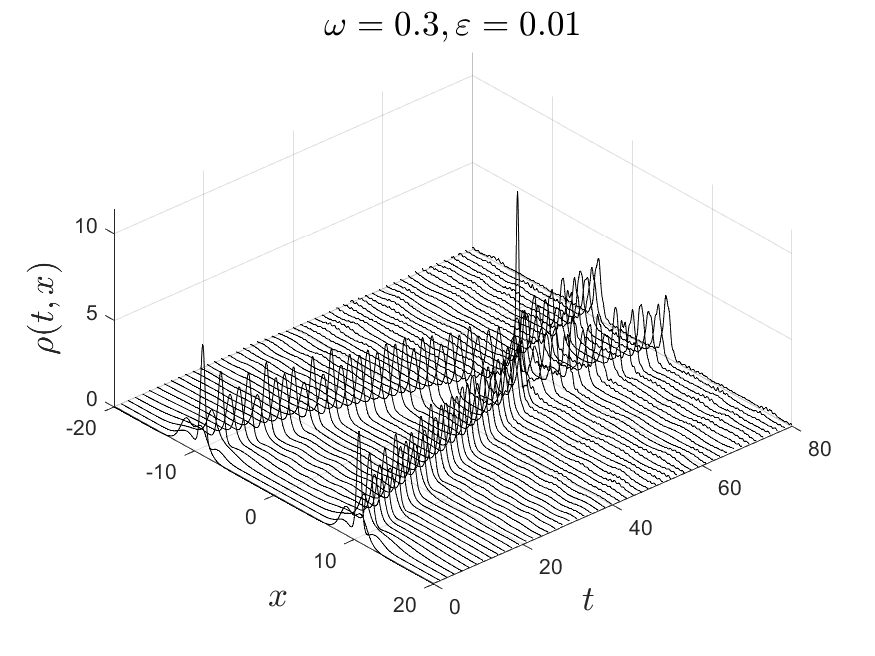,height=3.5cm,width=6cm}
\end{array}$$
\caption{Problem 2. Evolution of the density $\rho(t,x)=|\Phi(t,x)|^{2}$ of NLDE \eqref{equ-3-6-3} binary collision of one-humped (left) and two-huped waves (right) with $\eps=0.01$.}\label{fig-3-6-2}
\end{figure}

In Figure \ref{fig-3-6-1}, we examine the evolution of a single wave in the nonrelativistic regime by selecting $\omega=0.8$ and $\omega=0.3$ with $x_{0}=-5, v=0.2$. As depicted in the figure, both waves propagate to the right with oscillating amplitudes. The one-humped wave maintains its one-humped profile throughout its propagation, whereas the two-humped wave undergoes noticeable changes in its profile during its evolution.

In Figure \ref{fig-3-6-2}, we numerically solve the NLDE \eqref{equ-3-6-3} using the initial condition
$
\Phi(x,0)=\Phi_{l}^{ss}(x-x_{l},0)+\Phi_{r}^{ss}(x-x_{r},0),
$
where $x_{r}=-x_{l}=-10$, $v_{l}=-v_{r}=0.2$, $\omega_{r}=\omega_{l}=0.8  \ (0.3)$. These parameters describe two initial two-humped solitons with identical shapes and velocities but opposite directions \cite{ST}. In the nonrelativistic regime with $\eps=0.01$, two identical one-humped solitons ($\omega=0.8$) first collide and merge into a single wave, subsequently separating into two traveling waves with one-humped profiles. A similar behavior is observed for the collision of two-humped solitons.

\vspace{0.15cm}\textbf{Problem 3. Linear Case with magnetic potential.}
 In this  test, we demonstrate the applicability of the proposed integrators  to the linear Dirac equation  with magnetic potential (\cite{BCJT})
\begin{equation}\label{equ-5}
\begin{aligned}
&i\partial_{t}\Phi=\bigg(-\frac{\ii}{\varepsilon}\sum\limits_{j=1}^{d}\sigma_{j}\partial_{j}+\frac{1}{\varepsilon^2}\sigma_{3}\bigg)\Phi
+\big(V(\mathbf{x})I_{2}-\sum\limits_{j=1}^{d}A_{j}(\mathbf{x})\sigma_{j}\big)\Phi, \
 \Phi(0,\mathbf{x})=\Phi_{0}(\mathbf{x}), \ \mathbf{x}\in\mathbb{R}^{d},   \ t>0,
\end{aligned}
\end{equation}
where $I_{n}$ is the $n\times n$ identity matrix for $n\in\mathbb{N}$,
$V:=V(\mathbf{x})$ and $\mathbf{A}:=\mathbf{A}(\mathbf{x})=(A_{1}(\mathbf{x}),\ldots,A_{d}(\mathbf{x}))^{\intercal}$ are the real-valued electrical potential and magnetic potential vector.
The system \eqref{equ-5} conserves the total energy
\begin{align*}
E(t):&=\int_{\mathbb{R}^{d}}\bigg(-\frac{\ii}{\varepsilon}\sum\limits_{j=1}^{d}\Phi^{*}\sigma_{j}\partial_{j}\Phi+\frac{1}{\varepsilon^{2}}\Phi^{*}\sigma_{3}\Phi +V(\mathbf{x})|\Phi|^{2}-\sum\limits_{j=1}^{d}A_{j}(\mathbf{x})\Phi^{*}\sigma_{j}\Phi \bigg)d\mathbf{x}  \equiv E(0), \ t\geq 0.
\end{align*}
and the total mass \eqref{equ-6}. %Similar to the numerical analysis of the NLDE \eqref{equ-1}, we find that the EP-TSs integrator can also solve the LDE \eqref{equ-5}. Numerical experiments on the accuracy and long time behavior of linear Dirac equation \eqref{equ-5} solutions are given below.

\begin{figure}[t!]
$$\begin{array}{cc}
\psfig{figure=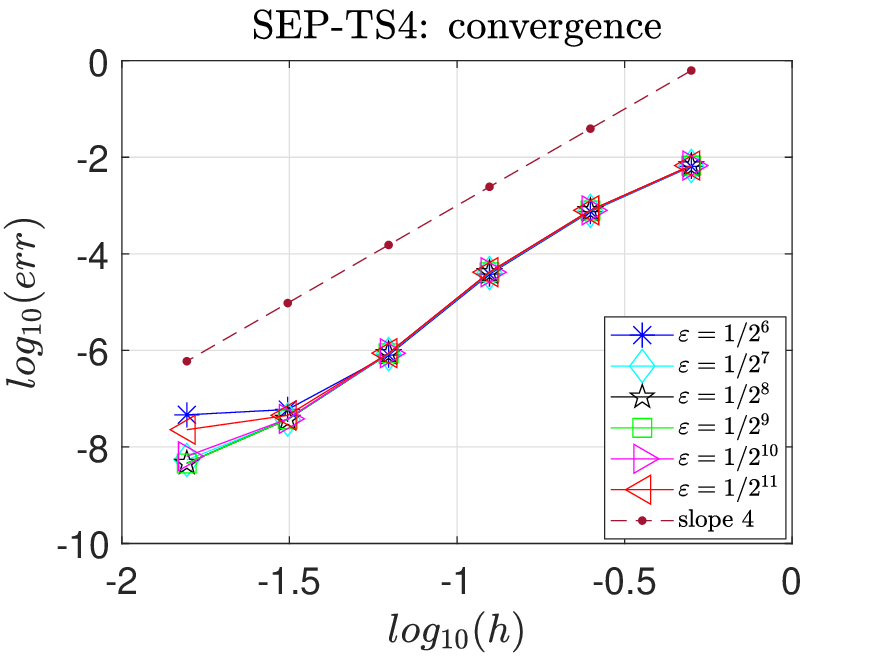,height=3cm,width=6cm}
\psfig{figure=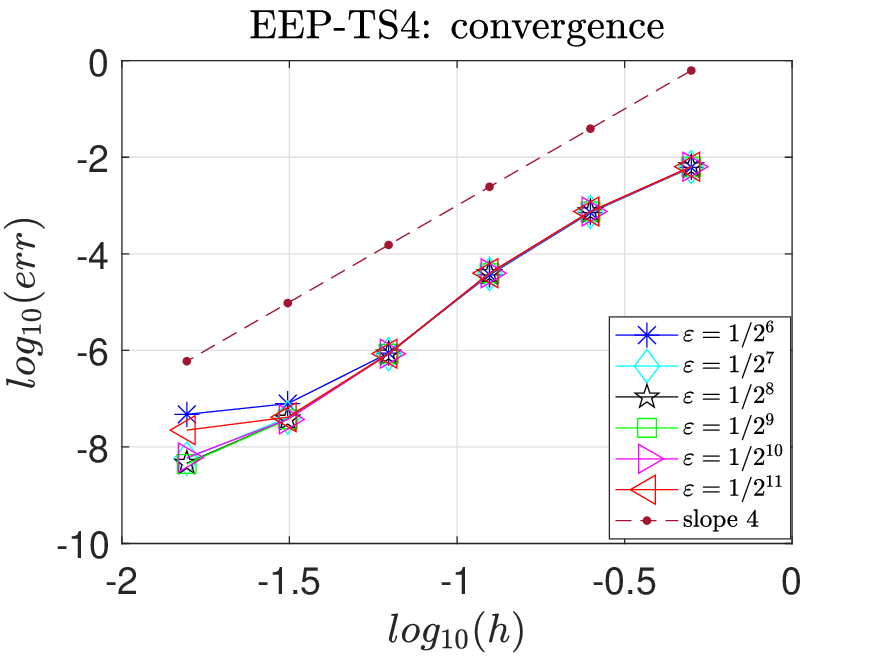,height=3cm,width=6cm}
\end{array}$$
\caption{Problem 3. Temporal error of LDE \eqref{equ-5} in 1D at $t=1$ under different $\eps$.}\label{fig-6}
\end{figure}

\begin{figure}[t!]
$$\begin{array}{cc}
\psfig{figure=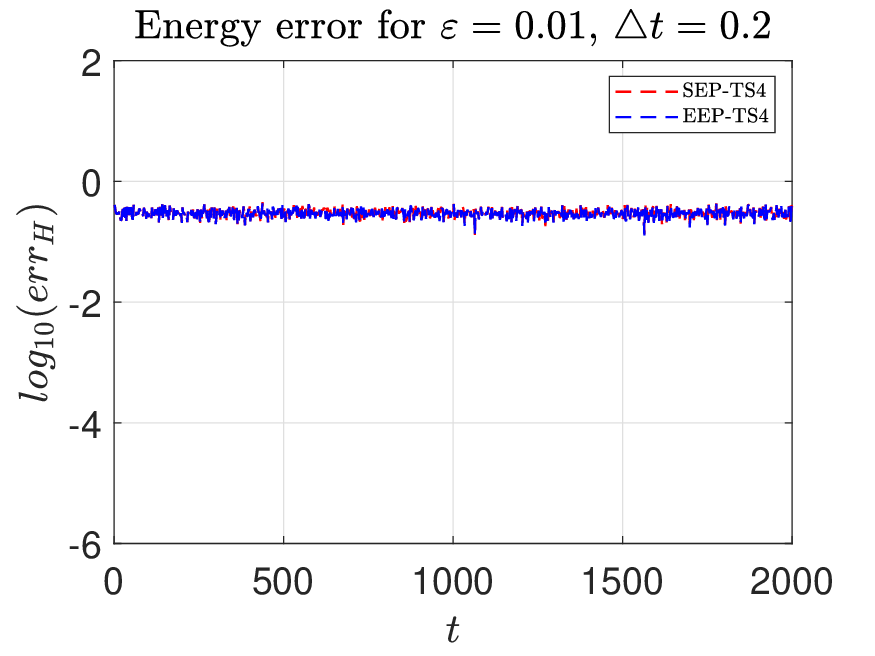,height=3cm,width=5cm}
\psfig{figure=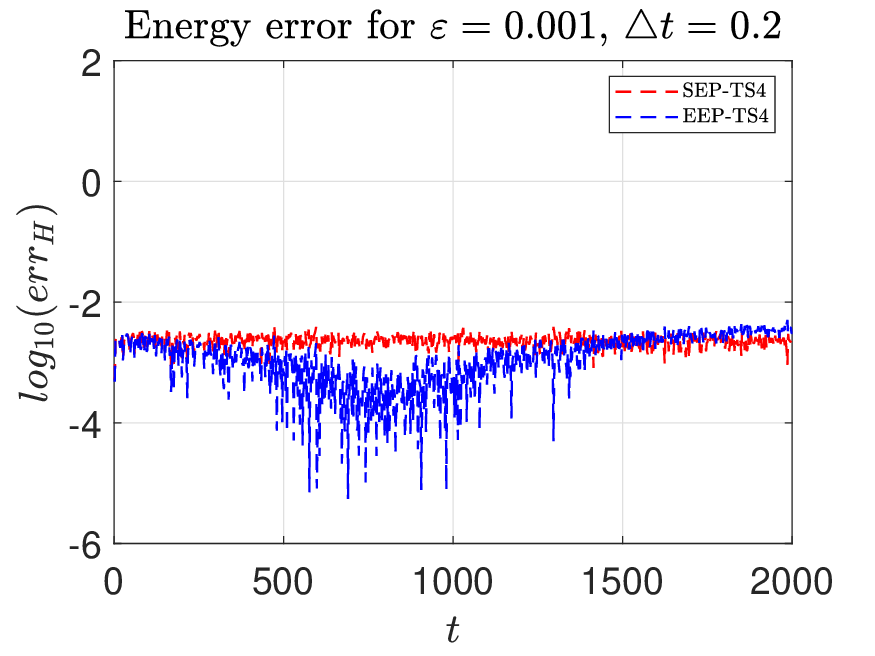,height=3cm,width=5cm}
\psfig{figure=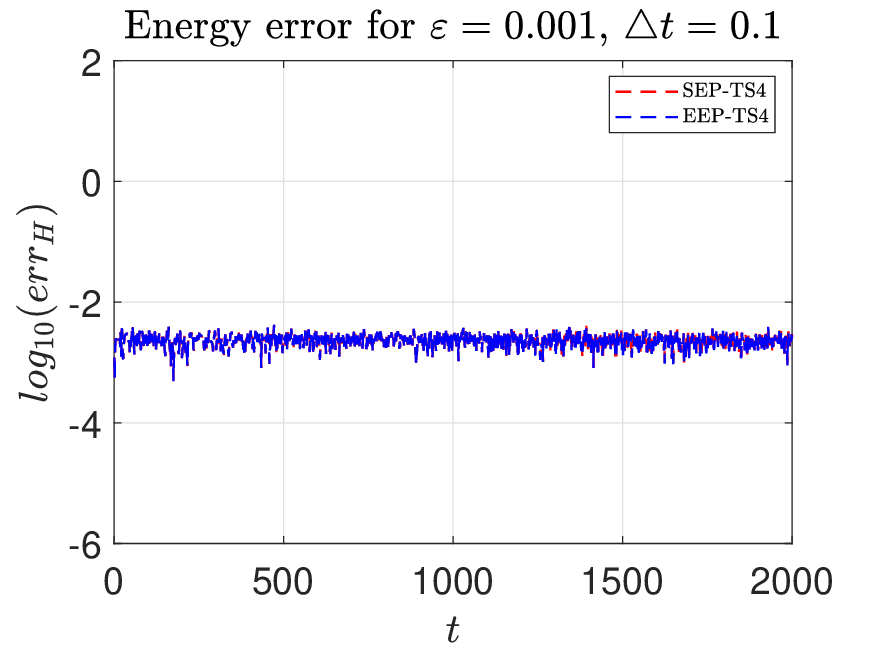,height=3cm,width=5cm}\\
\psfig{figure=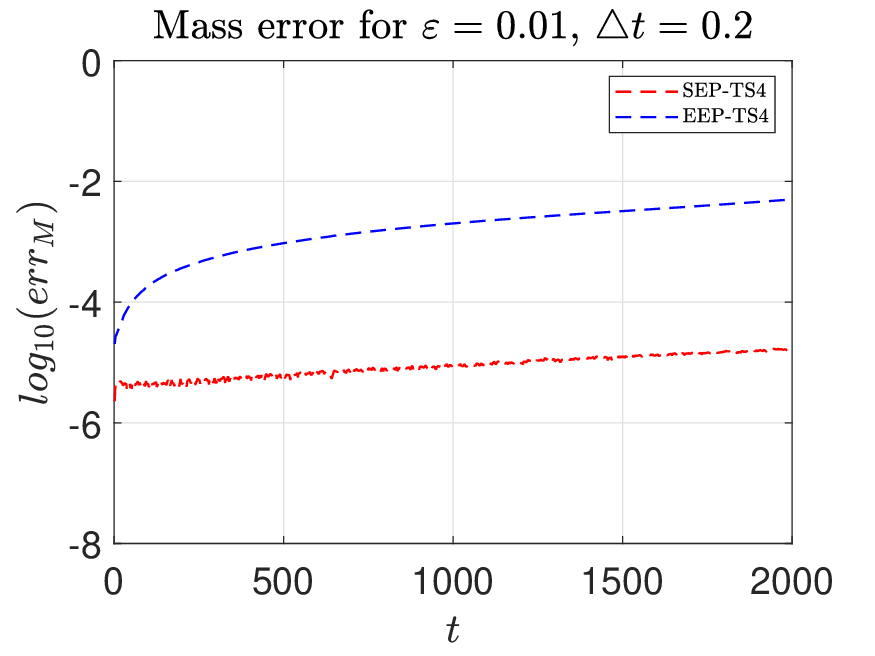,height=3cm,width=5cm}
\psfig{figure=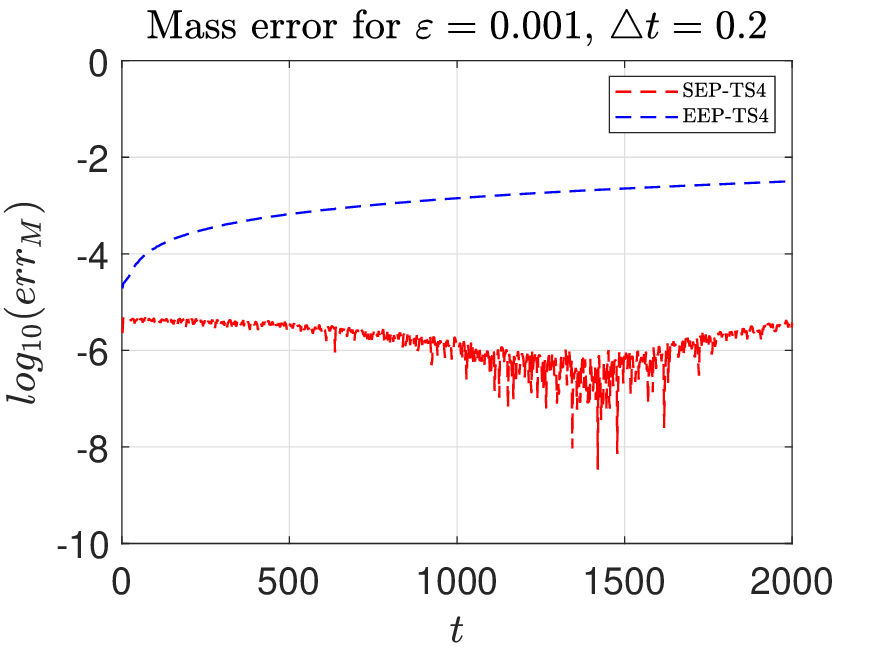,height=3cm,width=5cm}
\psfig{figure=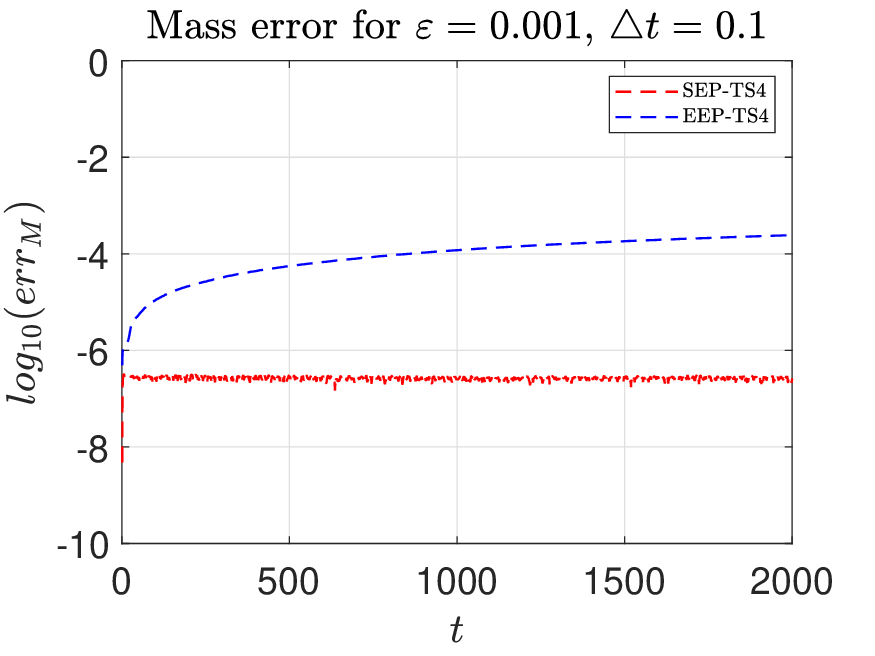,height=3cm,width=5cm}
\end{array}$$
\caption{Problem 3. Energy error (top) and mass error (bottom) of LDE \eqref{equ-5} in 1D  under different $\eps$ and $\triangle t$.}\label{fig-7}
\end{figure}

We choose $d=1$, the bounded domain $\Omega=(-32,32)$, the electric potential and magnetic potential as
$
V(x)=\frac{1-x}{1+x^{2}},  A(x)=\frac{(x+1)^{2}}{1+x^{2}},
$
and the initial data $\Phi_{0}=(\phi_{1},\phi_{2})$ with
$
\phi_{1}(0,x)=\fe^{-\frac{x^{2}}{2}},  \phi_{2}(0,x)=\fe^{-\frac{3(x-1)^{2}}{2}}.
$
The temporal errors  at $t=1$ under different $\eps$ are given in Figure \ref{fig-6} and
the long time energy and mass performances   are shown in Figure \ref{fig-7}.
The numerical phenomena observed are analogous to those of the first two problems.

\subsection{Two-dimensional Dirac equation}\label{sec-2d}\
This subsection is devoted to the numerical experiments on the two-dimensional Dirac equation.

\vspace{0.15cm} \textbf{Problem 4. Nonlinear case.}
In the following numerical experiments, we take $\Omega=(-16,16)^{2}$, $\triangle x=1/8$, $N_{\tau}=2^{6}$, and a honeycomb lattice potential
\begin{align}\label{equ-3-2-1}
V(\mathbf{x})=\cos\left(\frac{4\pi}{\sqrt{3}}\mathbf{e}_{1}\cdot\mathbf{x}\right)+\cos\left(\frac{4\pi}{\sqrt{3}}\mathbf{e}_{2}\cdot\mathbf{x}\right)
+\cos\left(\frac{4\pi}{\sqrt{3}}\mathbf{e}_{3}\cdot\mathbf{x}\right), \ \mathbf{x}\in\mathbb{R}^{2}, \ t\geq 0,
\end{align}
where $\mathbf{e}_{1}=(-1,0)^{\intercal}$, $\mathbf{e}_{2}=(1/2,\sqrt{3}/2)^{\intercal}$, $\mathbf{e}_{3}=(1/2,-\sqrt{3}/2)^{\intercal}$. Choose the initial data $\Phi_{0}(\mathbf{x})=\left(\phi_{1}(\mathbf{x}),\phi_{2}(\mathbf{x})\right)^{\intercal}$ as
 $
\phi_{1}(\mathbf{x})=\fe^{-\frac{x^{2}+y^{2}}{2}},   \phi_{2}(\mathbf{x})=\fe^{-\frac{(x-1)^{2}+y^{2}}{2}},  \mathbf{x}=(x,y)^{\intercal} \in\mathbb{R}^{2}.$

\begin{figure}[t!]
$$\begin{array}{cc}
\psfig{figure=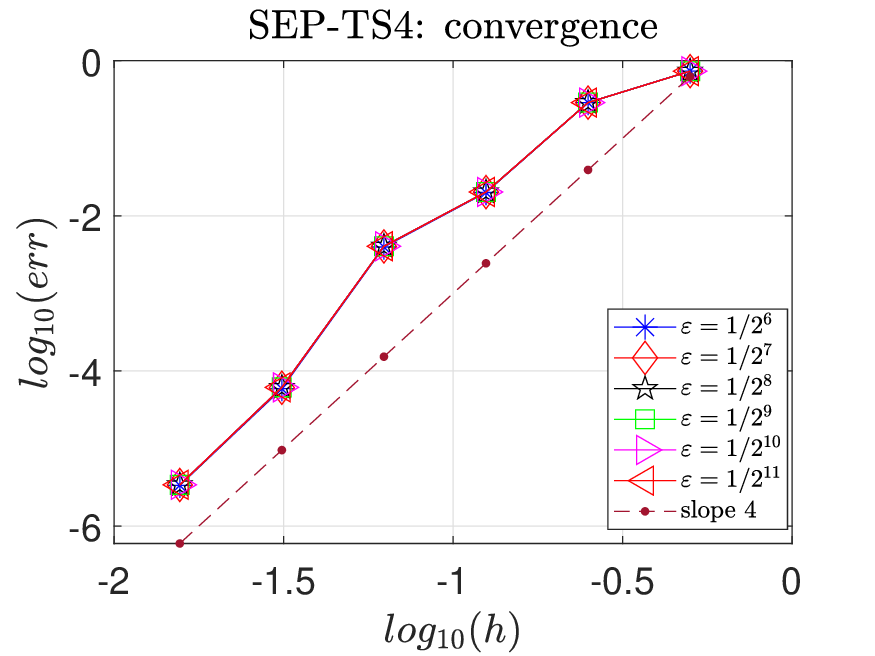,height=3cm,width=6cm}
\psfig{figure=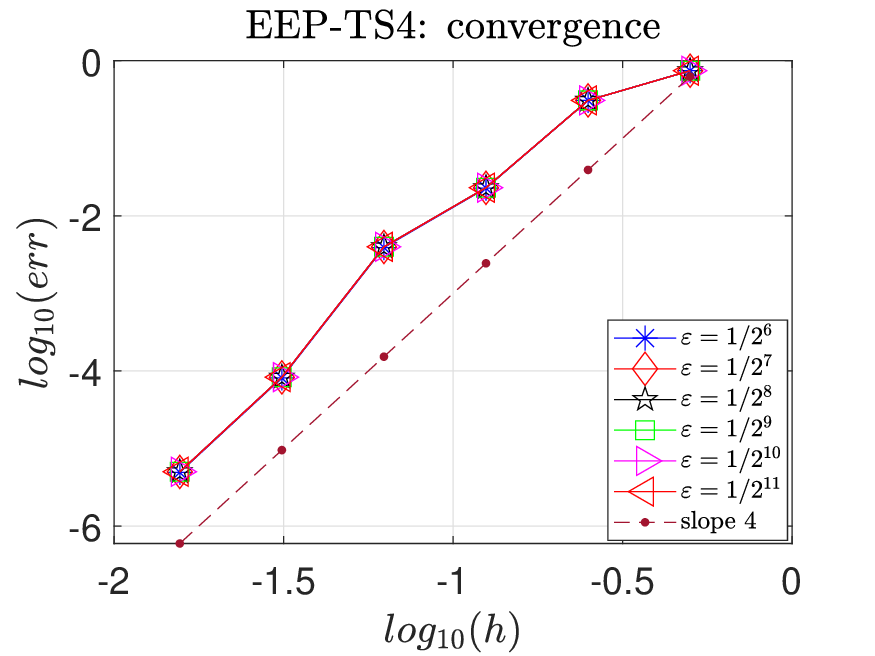,height=3cm,width=6cm}
\end{array}$$
\caption{Problem 4. Temporal error of NLDE \eqref{equ-1} in 2D at $t=1$ under different $\eps$.}\label{fig-3-10-1}
\end{figure}

\begin{figure}[t!]
$$\begin{array}{cc}
\psfig{figure=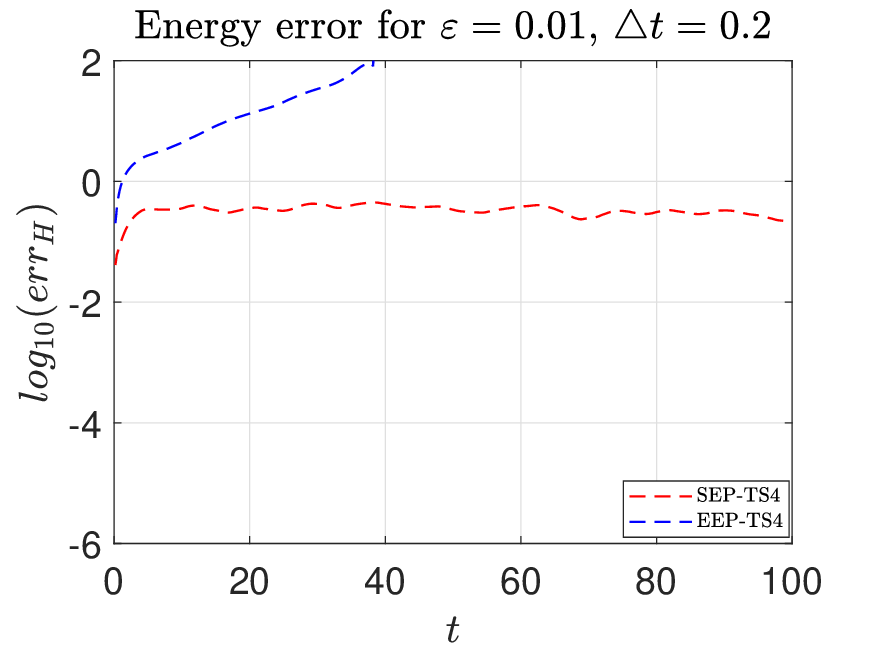,height=3cm,width=5cm}
\psfig{figure=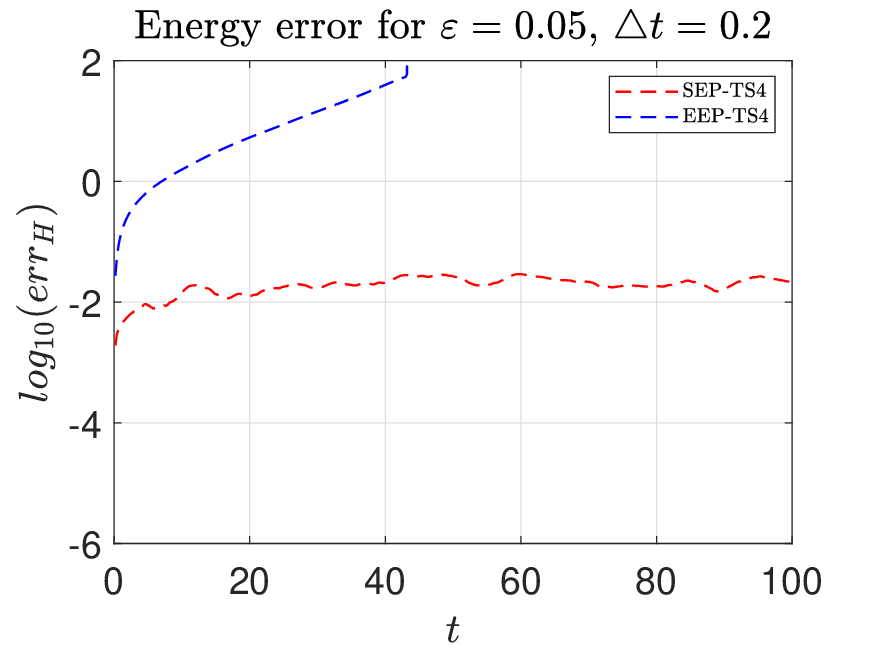,height=3cm,width=5cm}
\psfig{figure=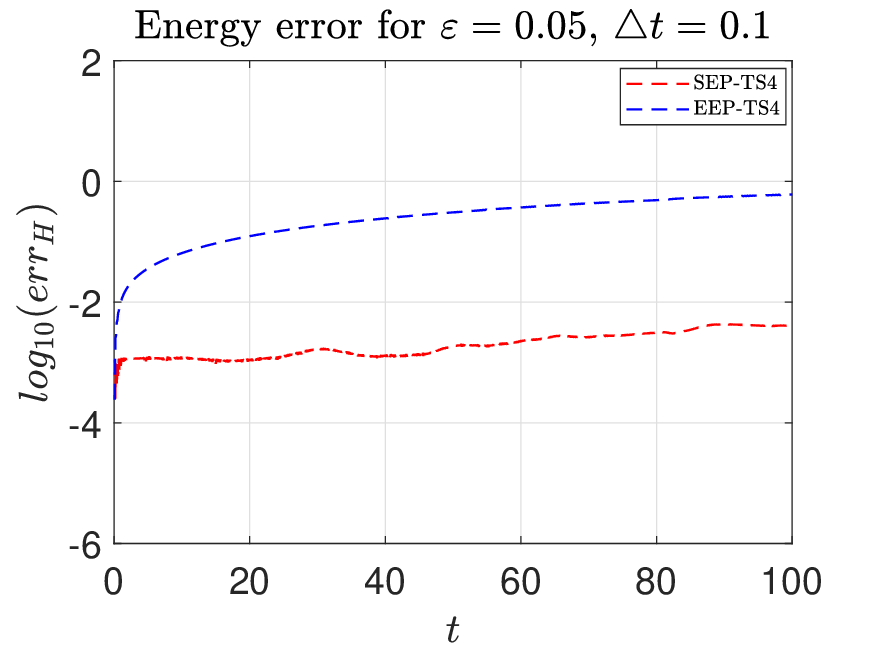,height=3cm,width=5cm}\\
\psfig{figure=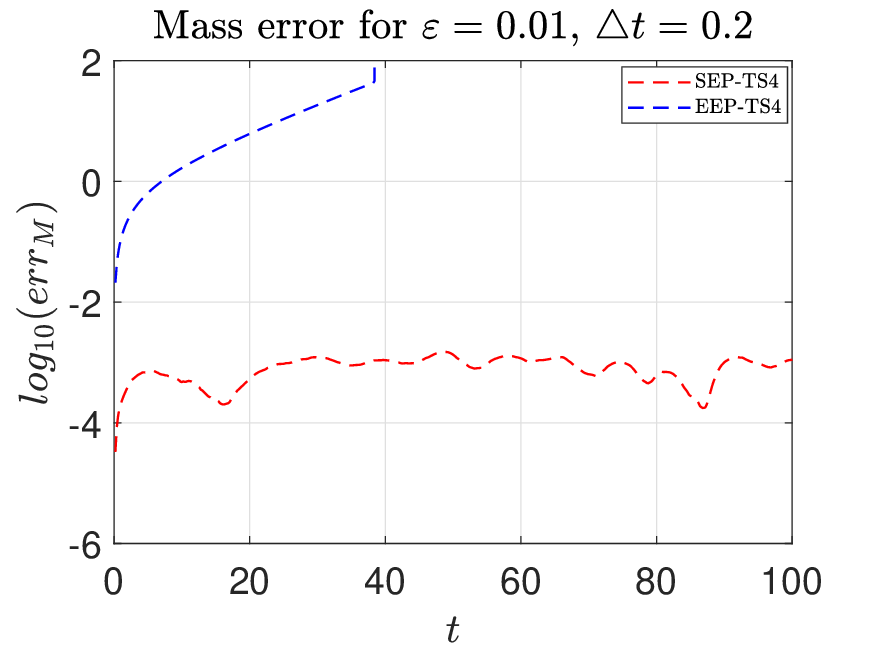,height=3cm,width=5cm}
\psfig{figure=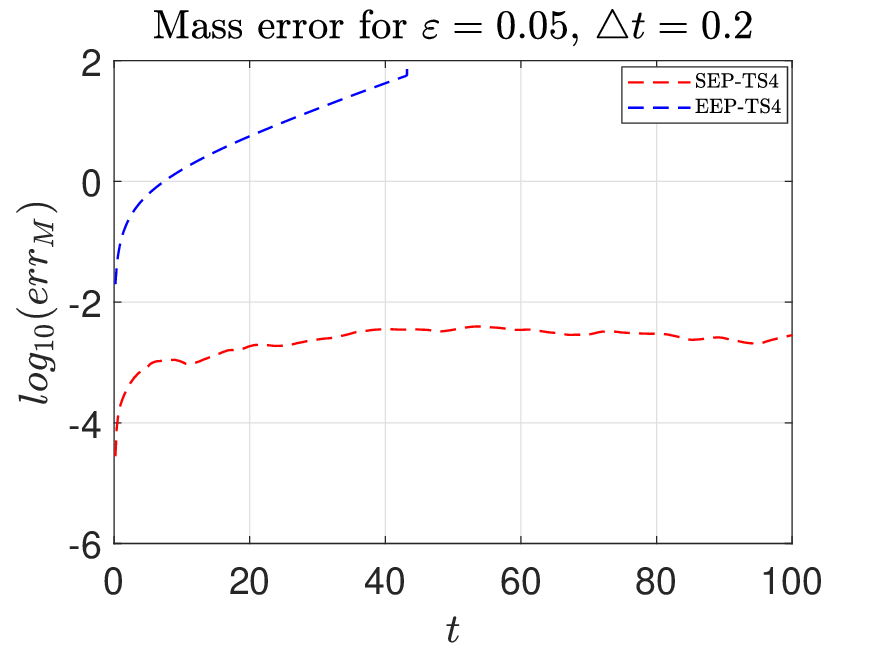,height=3cm,width=5cm}
\psfig{figure=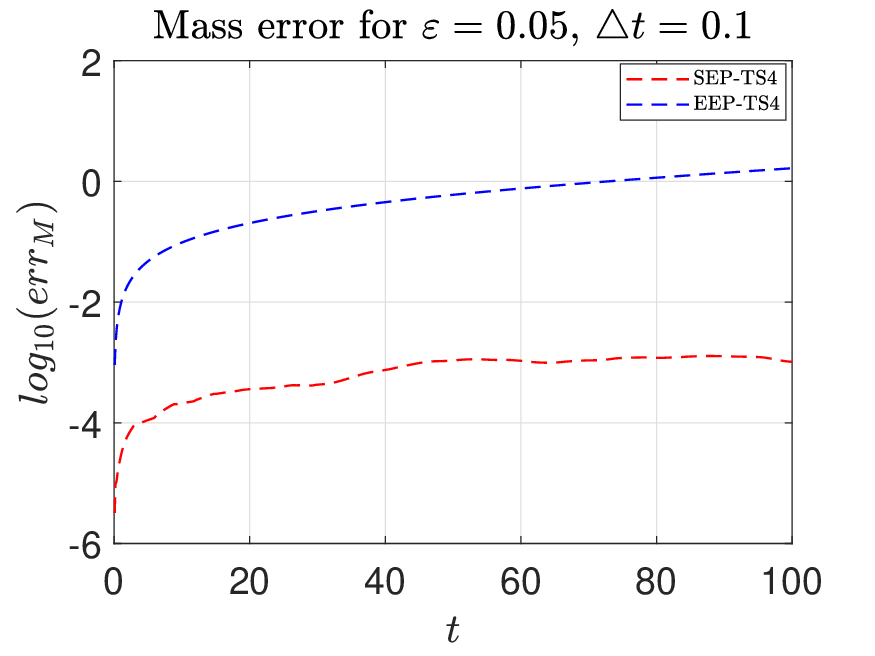,height=3cm,width=5cm}
\end{array}$$
\caption{Problem 4. Energy error (top) and mass error (bottom) of NLDE \eqref{equ-1} in 2D  under different $\eps$ and $\triangle t$.}\label{fig-3-10-2}
\end{figure}
\begin{figure}[t!]
$$\begin{array}{cc}
\psfig{figure=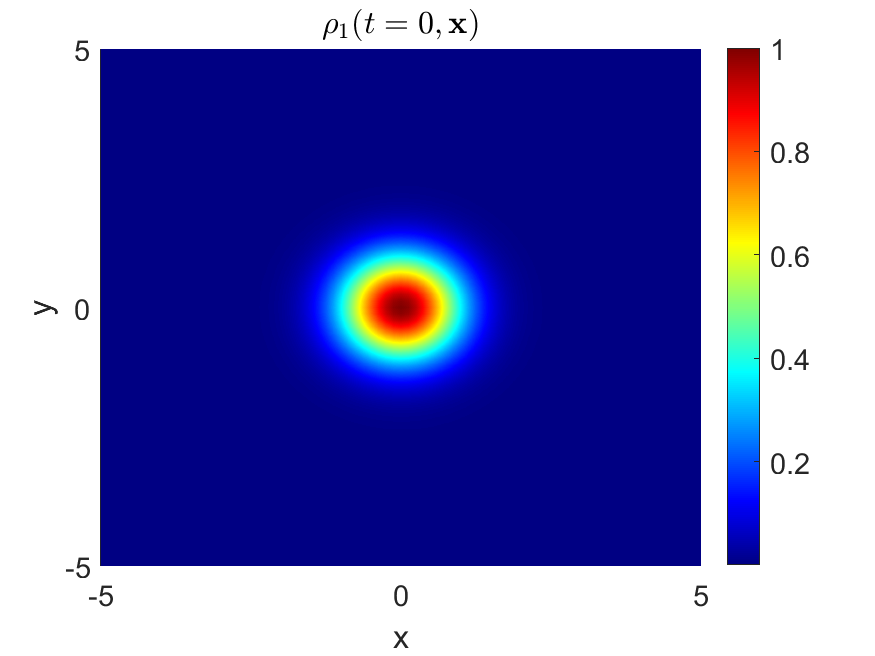,height=2.8cm,width=3.5cm}
\psfig{figure=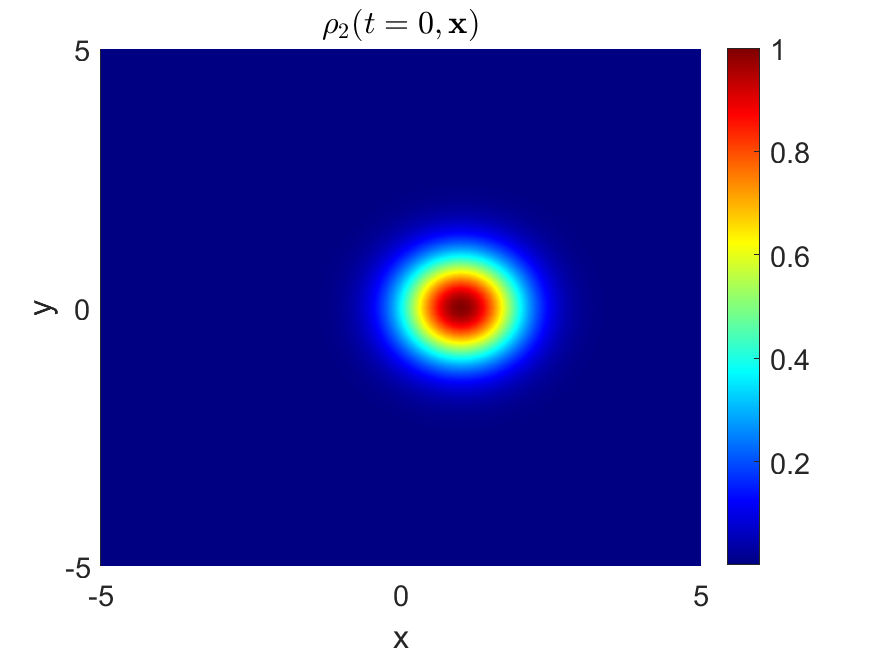,height=2.8cm,width=3.5cm}
\psfig{figure=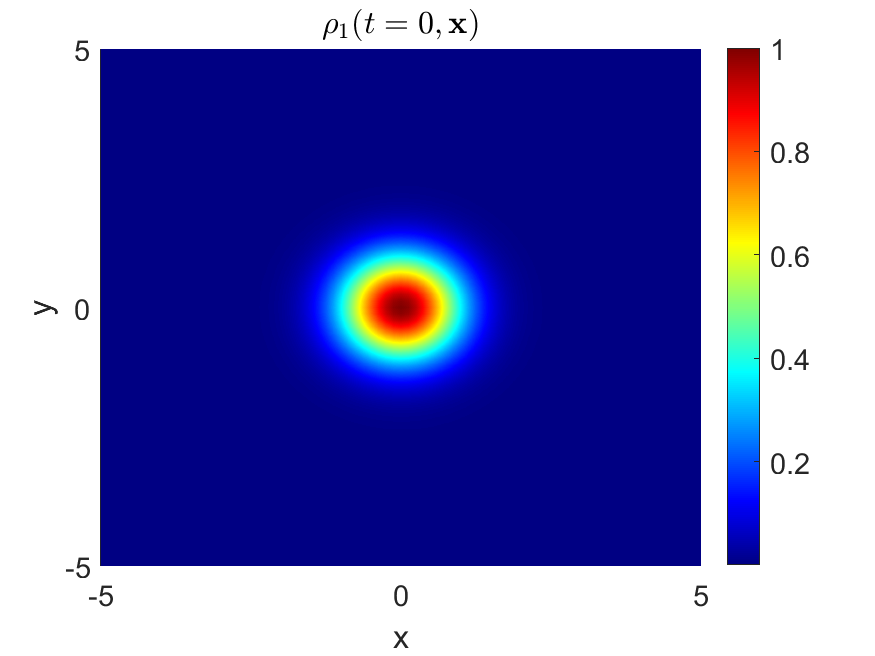,height=2.8cm,width=3.5cm}
\psfig{figure=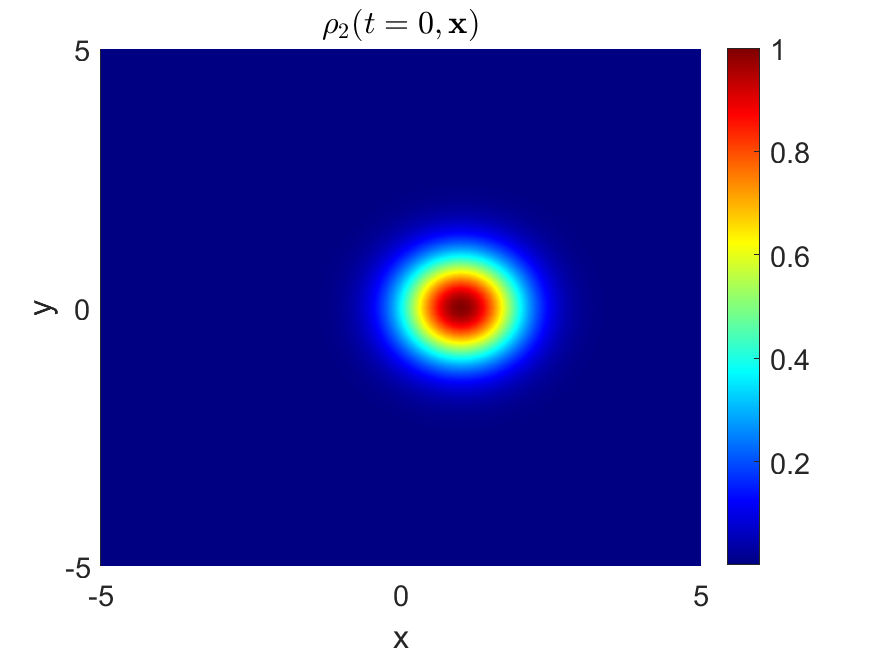,height=2.8cm,width=3.5cm}\\
\psfig{figure=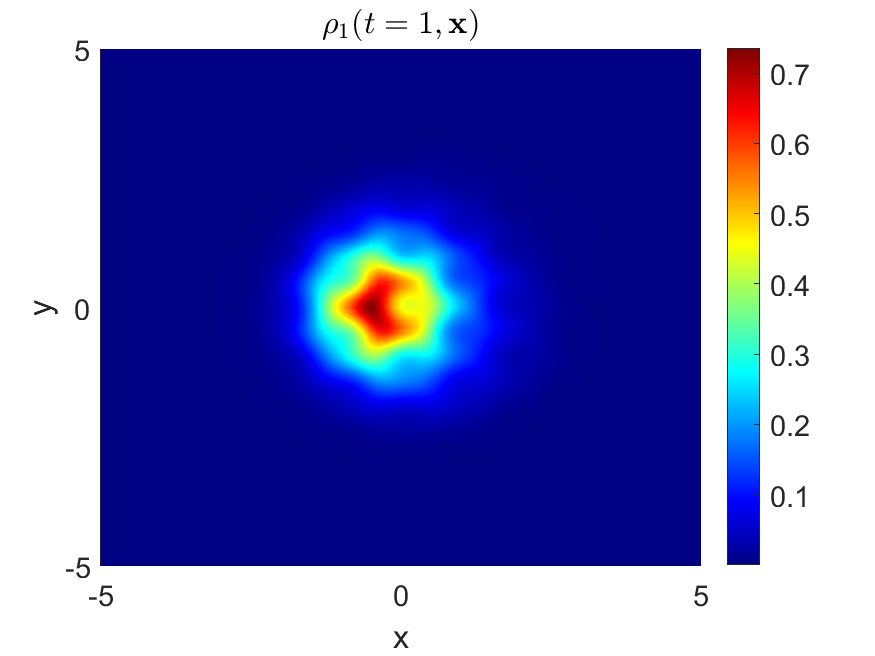,height=2.8cm,width=3.5cm}
\psfig{figure=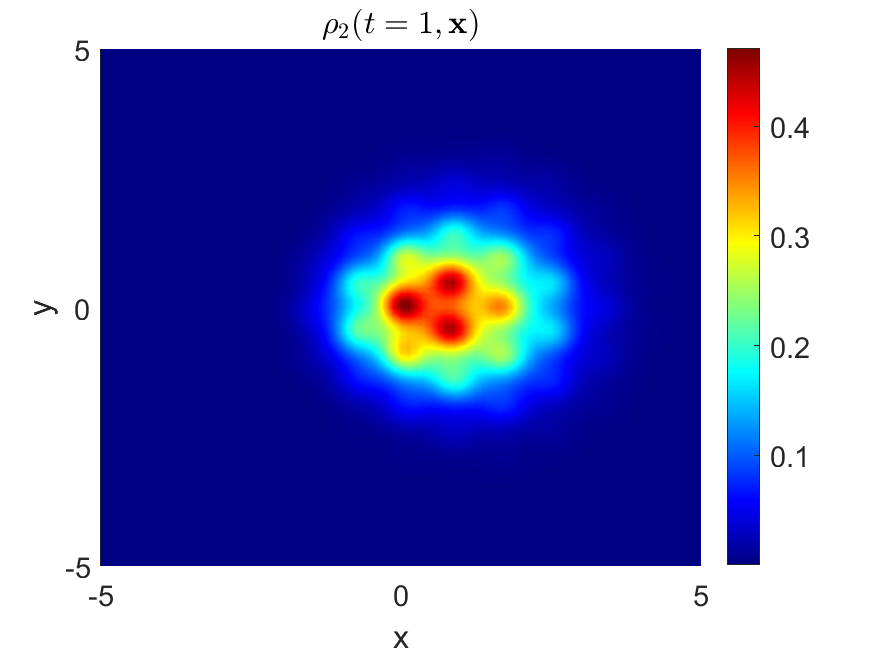,height=2.8cm,width=3.5cm}
\psfig{figure=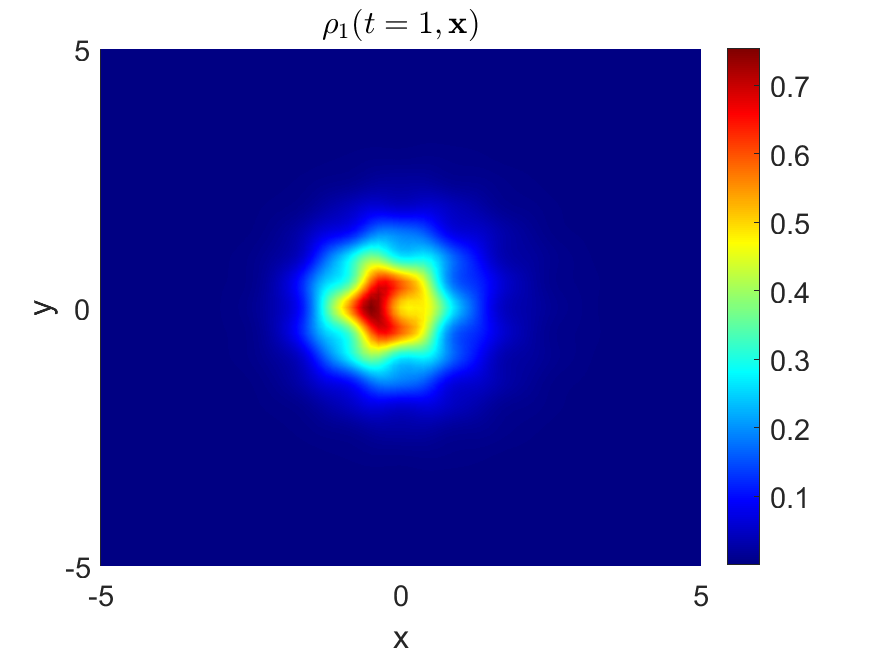,height=2.8cm,width=3.5cm}
\psfig{figure=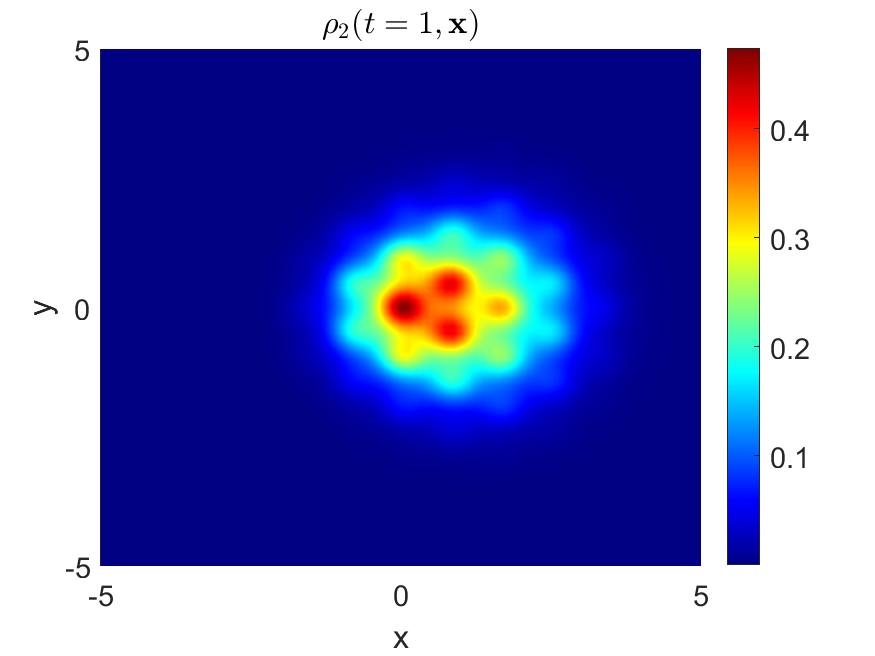,height=2.8cm,width=3.5cm}\\
\psfig{figure=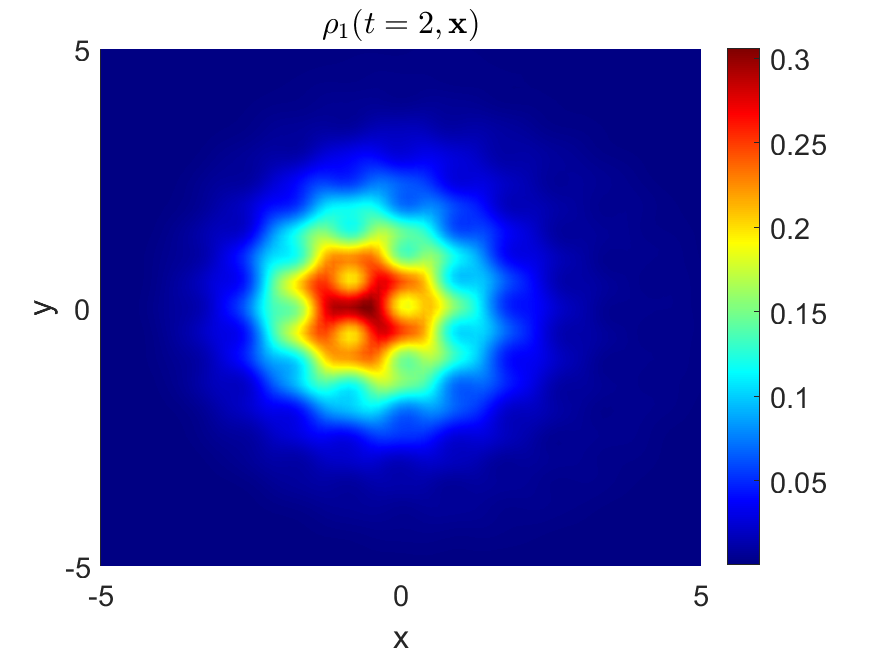,height=2.8cm,width=3.5cm}
\psfig{figure=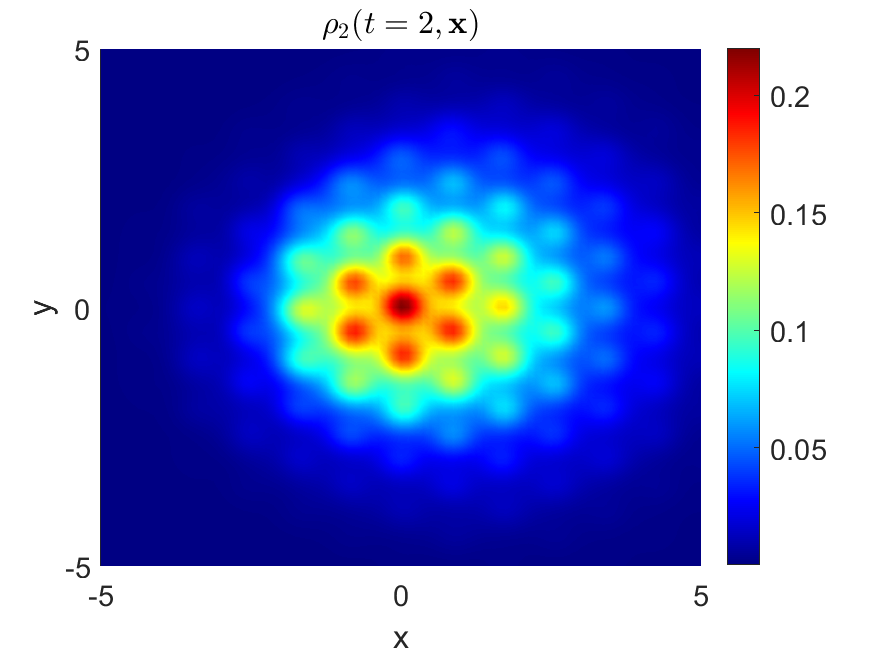,height=2.8cm,width=3.5cm}
\psfig{figure=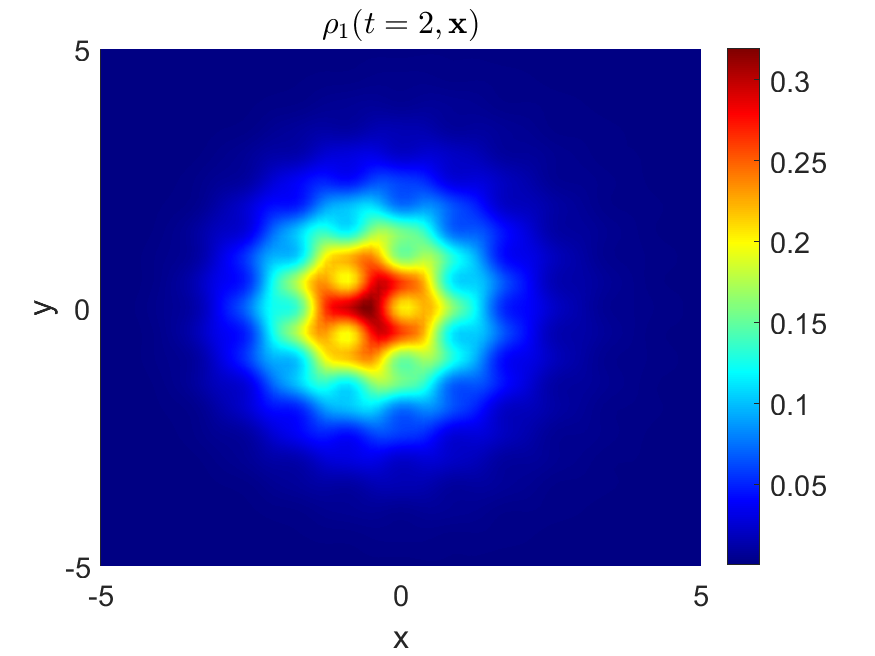,height=2.8cm,width=3.5cm}
\psfig{figure=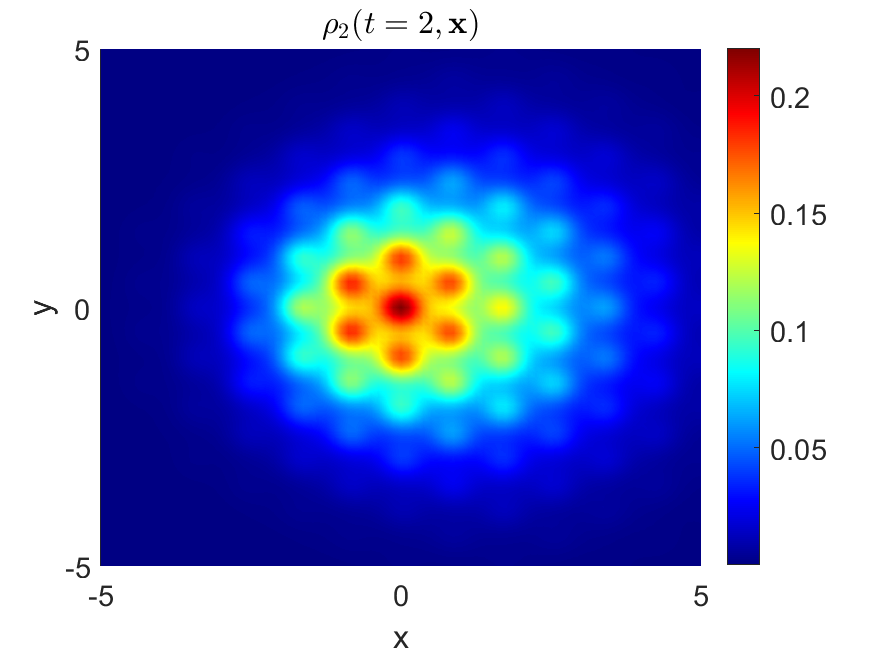,height=2.8cm,width=3.5cm}\\
\psfig{figure=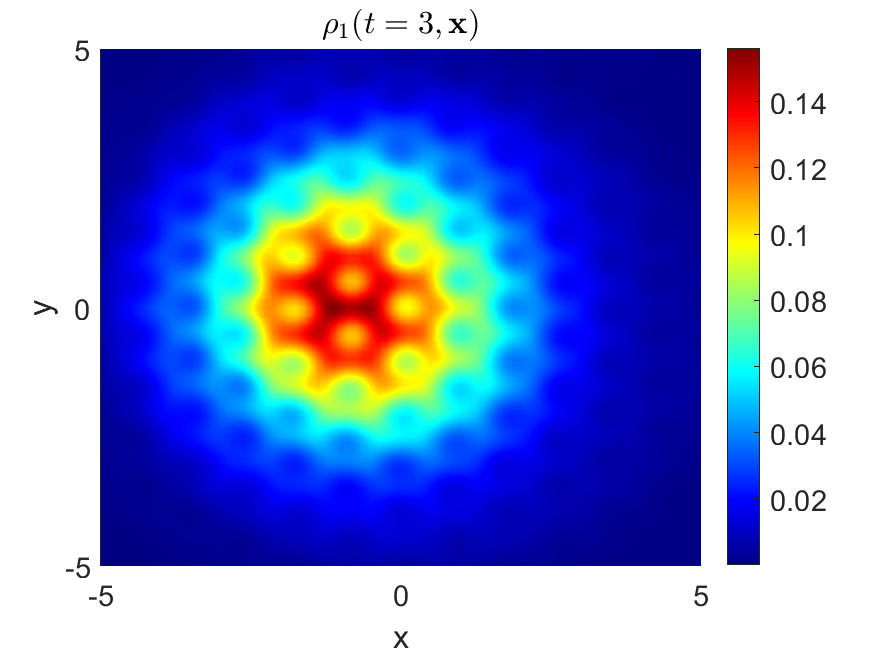,height=2.8cm,width=3.5cm}
\psfig{figure=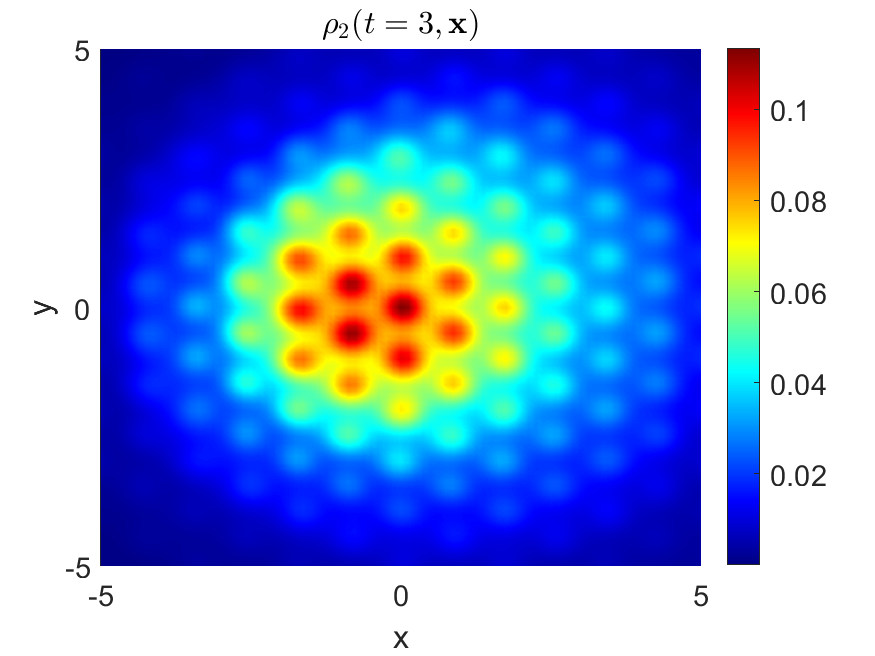,height=2.8cm,width=3.5cm}
\psfig{figure=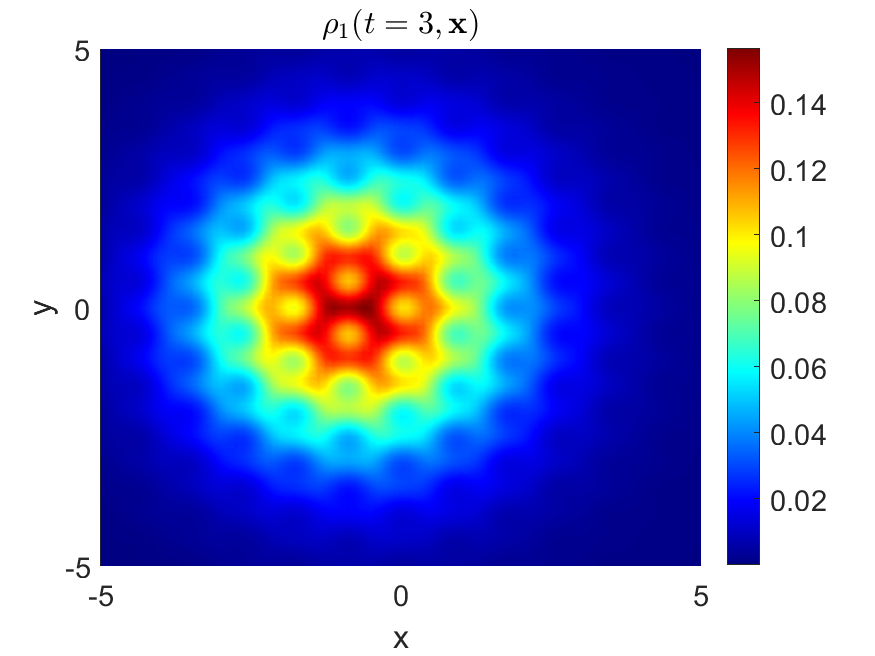,height=2.8cm,width=3.5cm}
\psfig{figure=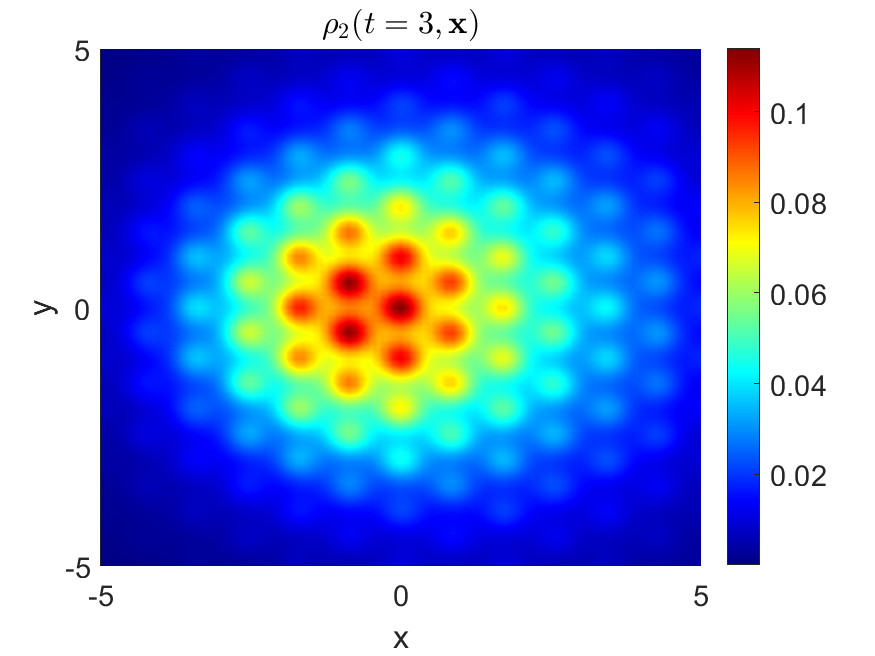,height=2.8cm,width=3.5cm}\\
\end{array}$$
\caption{Problem 4. Contour plots of the densities $\rho_{i}(t,\mathbf{x})=|\phi_{i}(t,\mathbf{x})|^{2},\ i=1,2$ for $\eps=0.1 \ (left \ two \ columns)$ and $\eps=0.01 \ (right \ two \ columns)$ of NLDE \eqref{equ-1} in 2D under the nonlinearity $\lambda_{1}=-1$, $\lambda_{2}=0$ with a honeycomb potential.}\label{fig-3-1-2}
\end{figure}

Firstly,  Figures \ref{fig-3-10-1} and \ref{fig-3-10-2} respectively display the temporal errors and long-term performance of the 2D Dirac equation \eqref{equ-1}. The reference solution is given by the 2D Strang splitting method with $\triangle t=10^{-5}, h=1/8$.
Fourth order uniform accuracy   and long time  behaviour  are clearly observed.

Then we give the dynamics of the nonlinear Dirac equation \eqref{equ-1} in 2D (d=2) by SEP-TS4 with $\triangle t=0.01$.
Figures \ref{fig-3-1-2} and \ref{fig-3-1-3} illustrate the densities $\rho_{j}(t,\mathbf{x})=|\phi_{j}(t,\mathbf{x})|^{2}, j=1,2$, for $\eps=0.1$ and $\eps=0.01$ under varying nonlinearities. As observed, for both $\eps=0.1$ and $\eps=0.01$, the densities exhibit a smoother distribution across the lattice potential, and the dynamics closely resemble those of the Schr\"{o}dinger-type limiting system.

\begin{figure}[t!]
$$\begin{array}{cc}
\psfig{figure=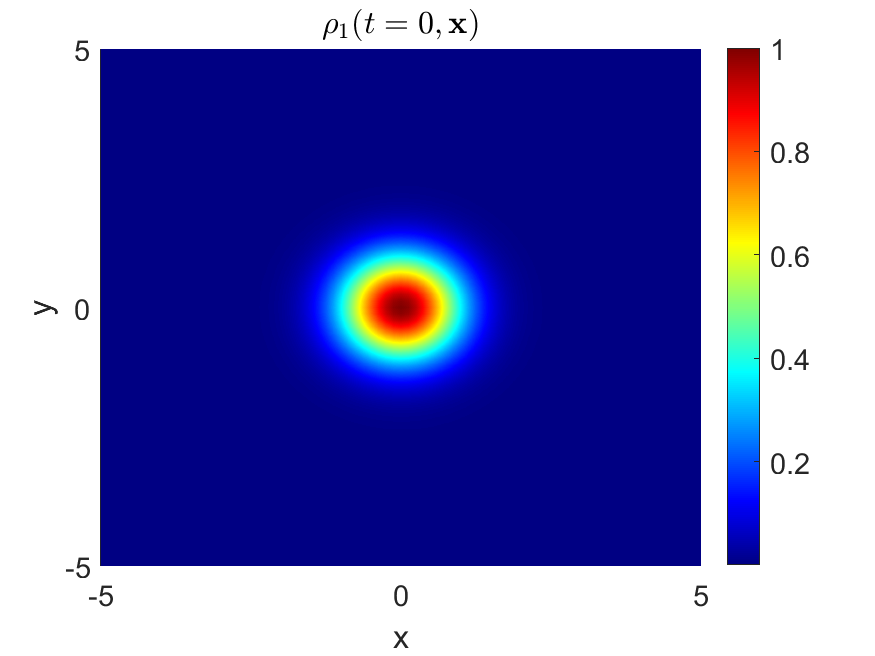,height=2.8cm,width=3.5cm}
\psfig{figure=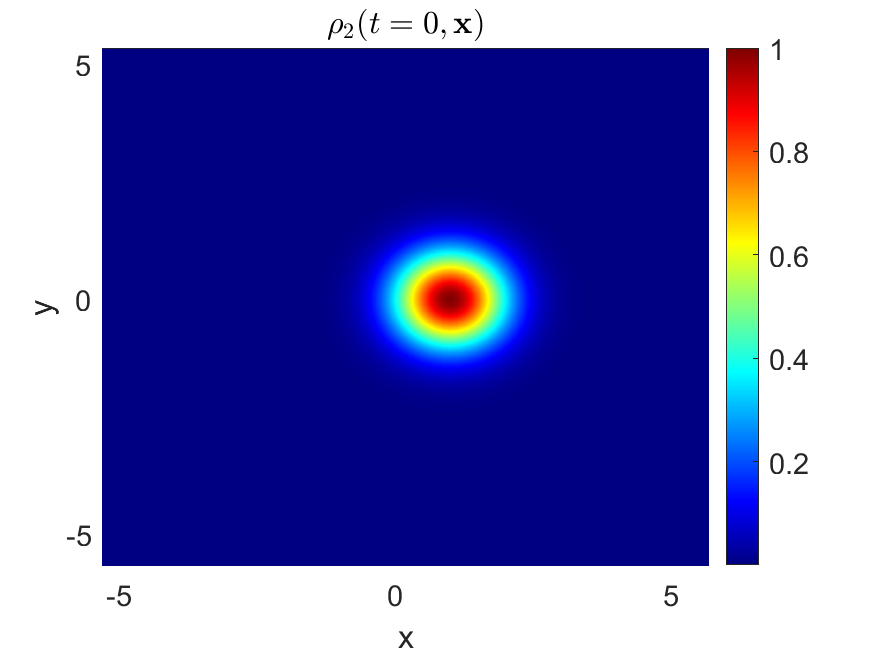,height=2.8cm,width=3.5cm}
\psfig{figure=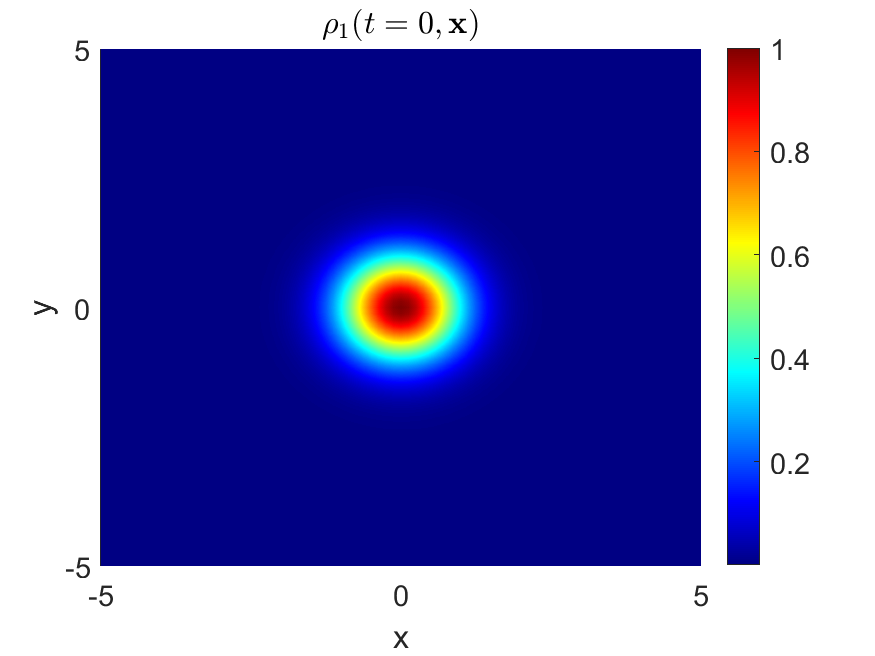,height=2.8cm,width=3.5cm}
\psfig{figure=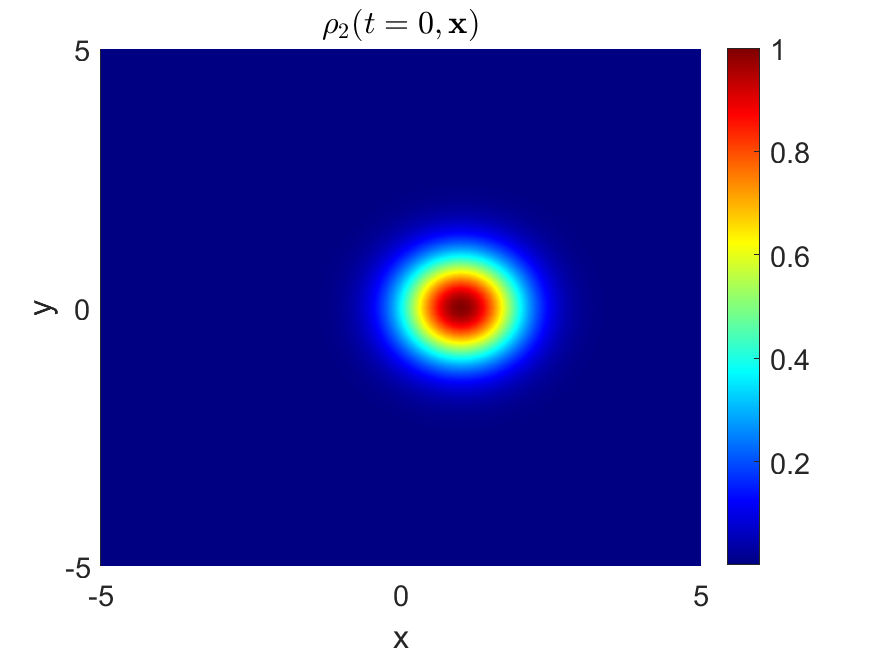,height=2.8cm,width=3.5cm}\\
\psfig{figure=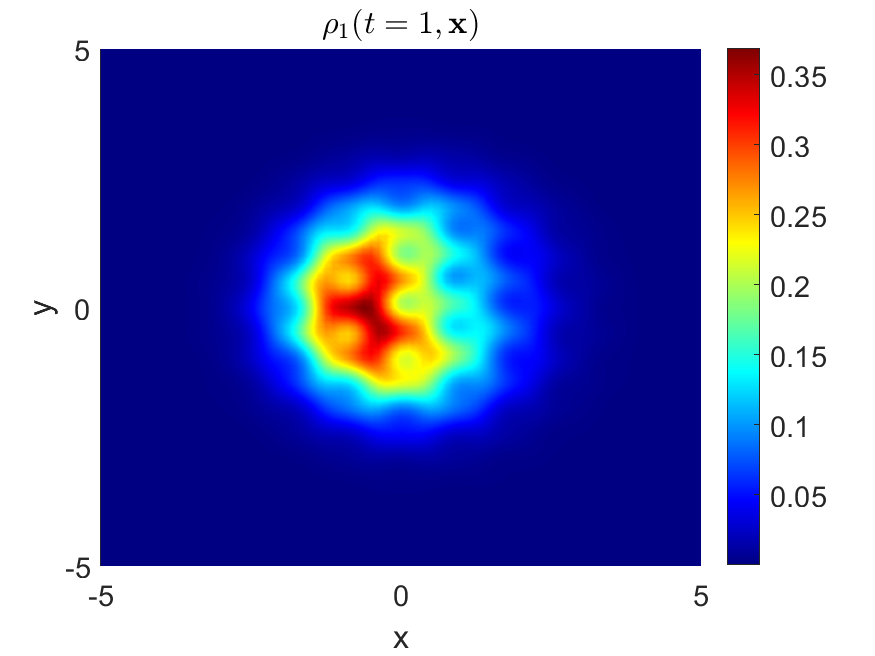,height=2.8cm,width=3.5cm}
\psfig{figure=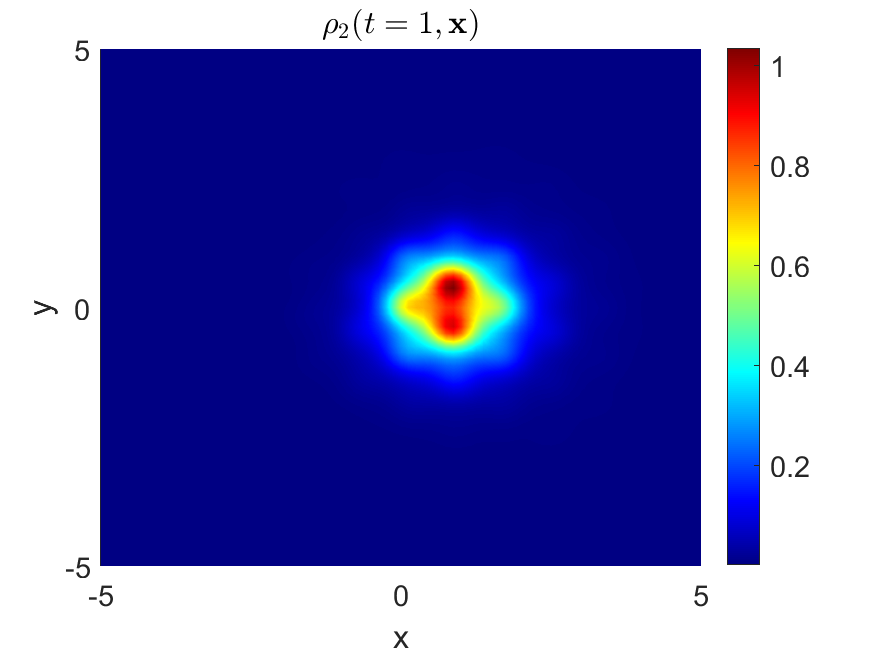,height=2.8cm,width=3.5cm}
\psfig{figure=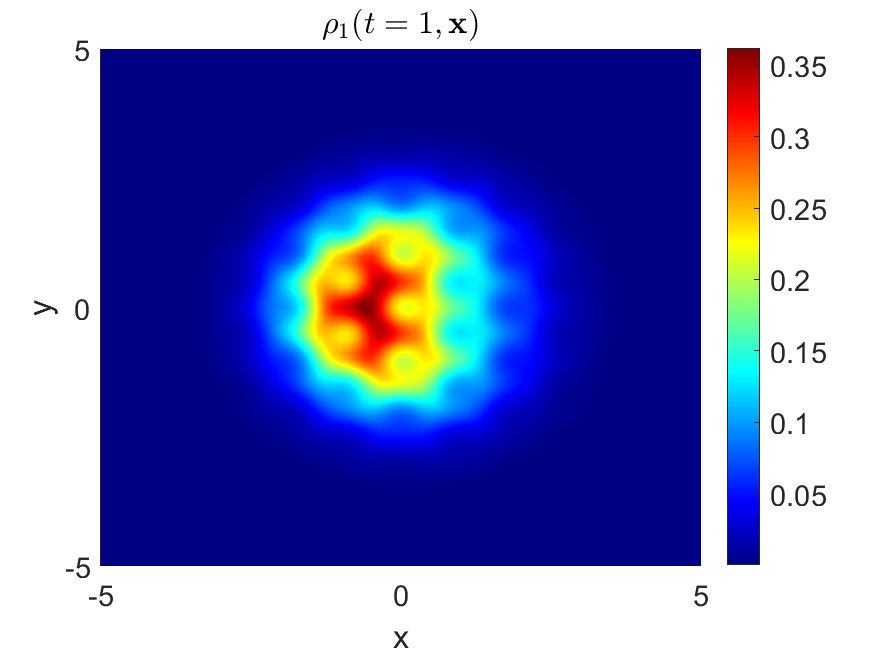,height=2.8cm,width=3.5cm}
\psfig{figure=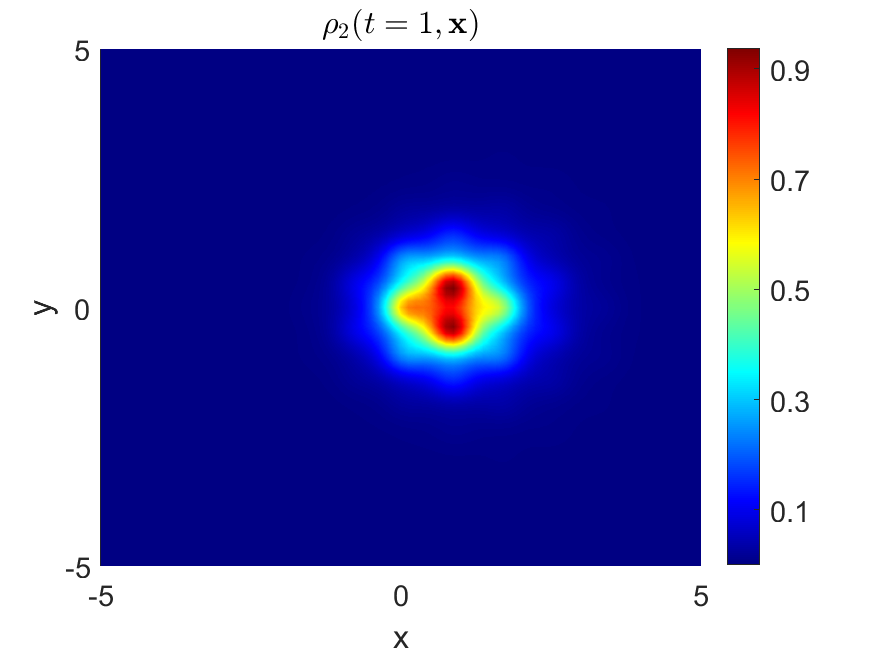,height=2.8cm,width=3.5cm}\\
\psfig{figure=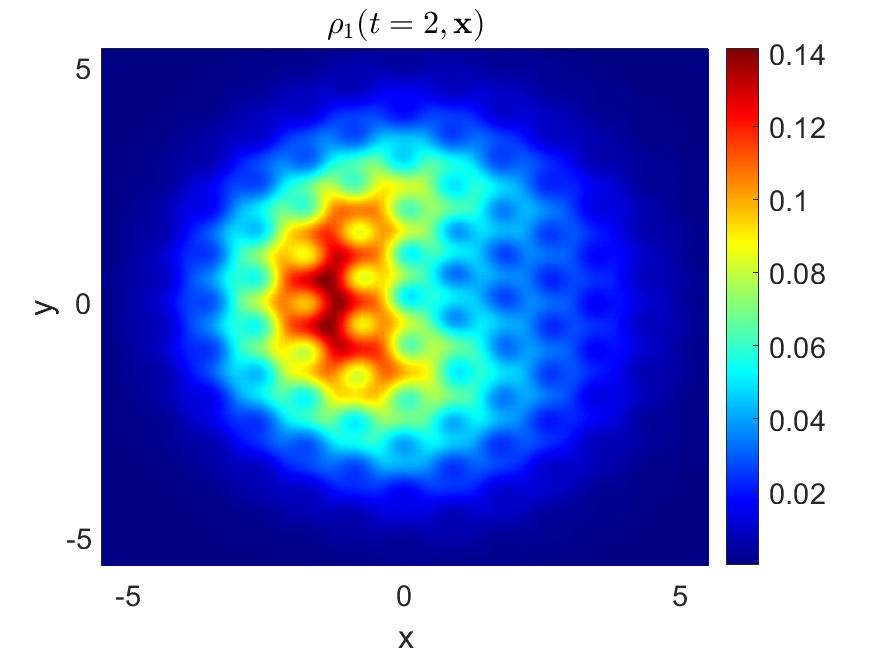,height=2.8cm,width=3.5cm}
\psfig{figure=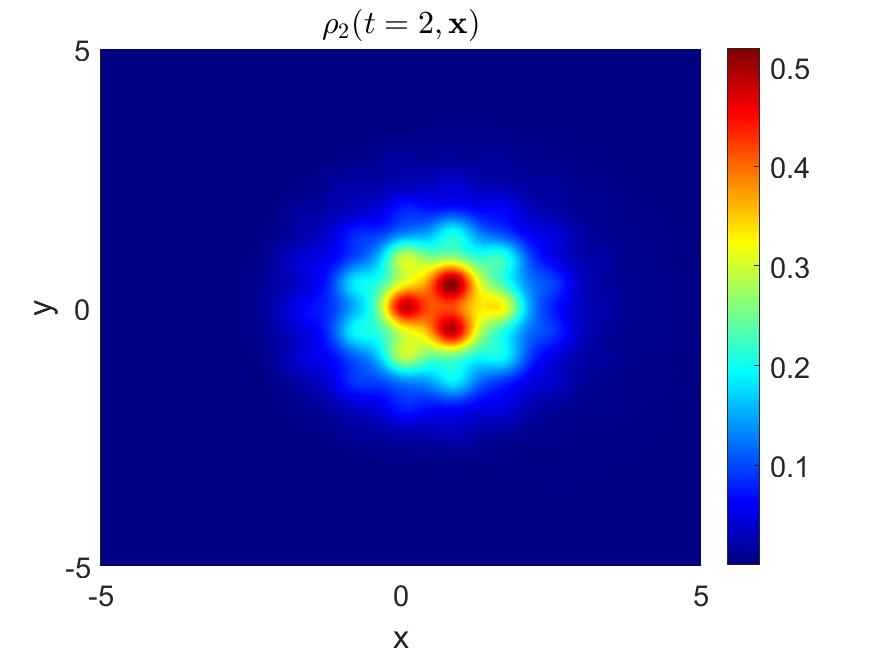,height=2.8cm,width=3.5cm}
\psfig{figure=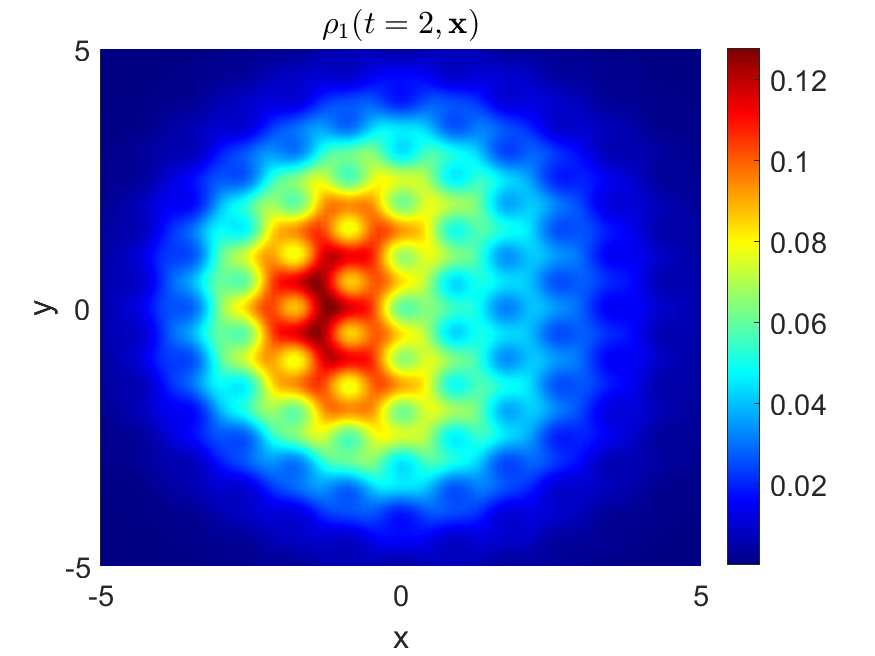,height=2.8cm,width=3.5cm}
\psfig{figure=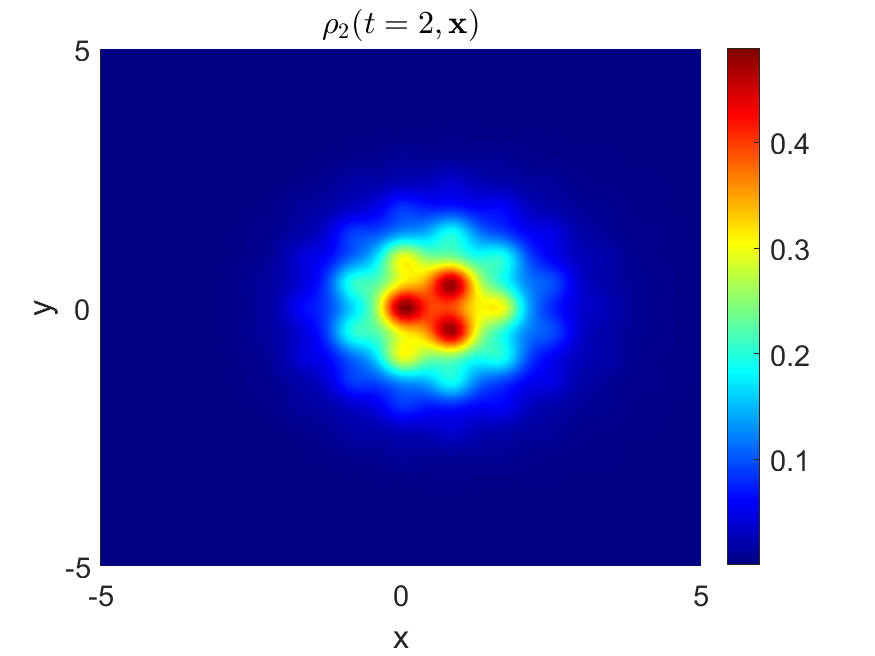,height=2.8cm,width=3.5cm}\\
\psfig{figure=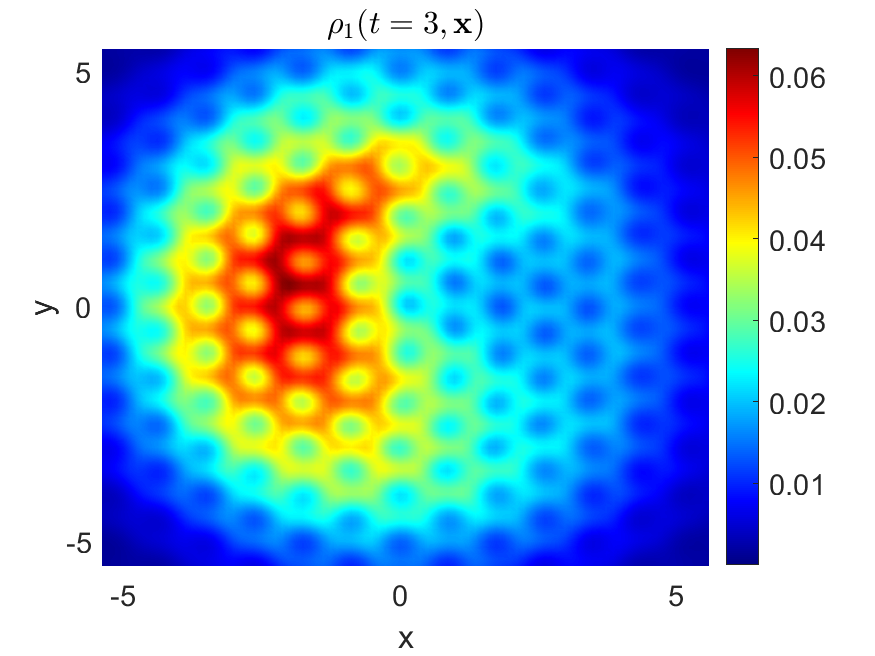,height=2.8cm,width=3.5cm}
\psfig{figure=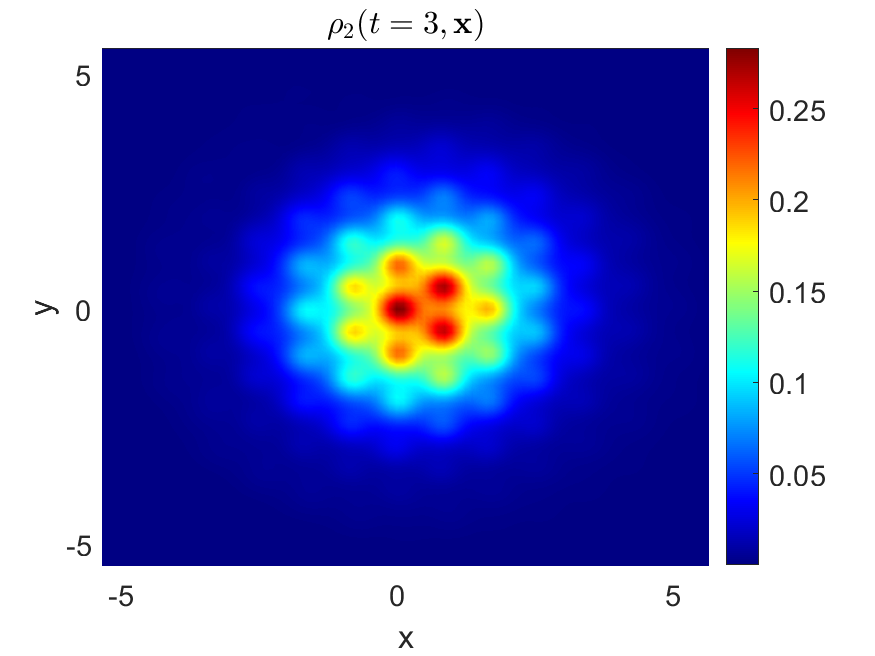,height=2.8cm,width=3.5cm}
\psfig{figure=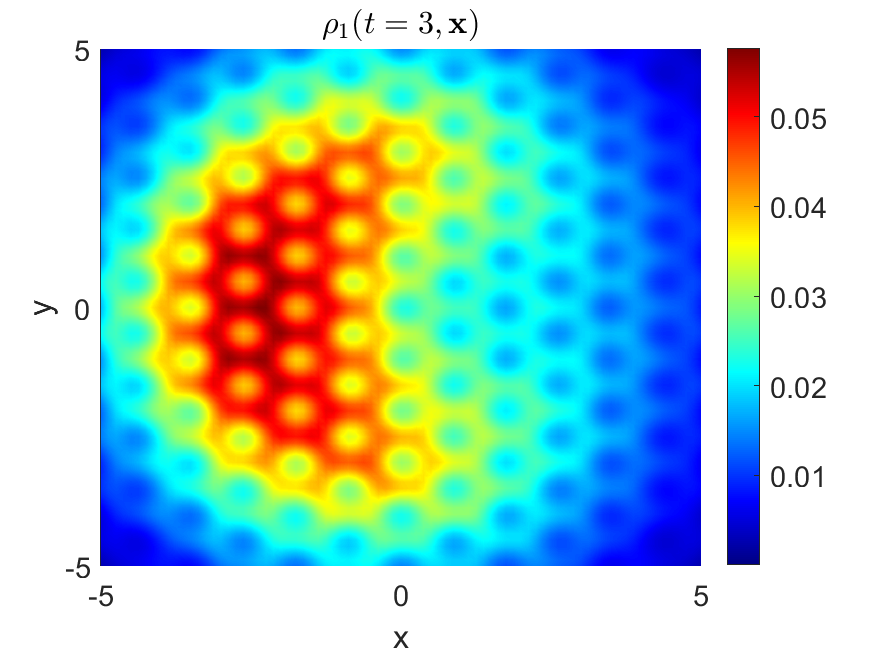,height=2.8cm,width=3.5cm}
\psfig{figure=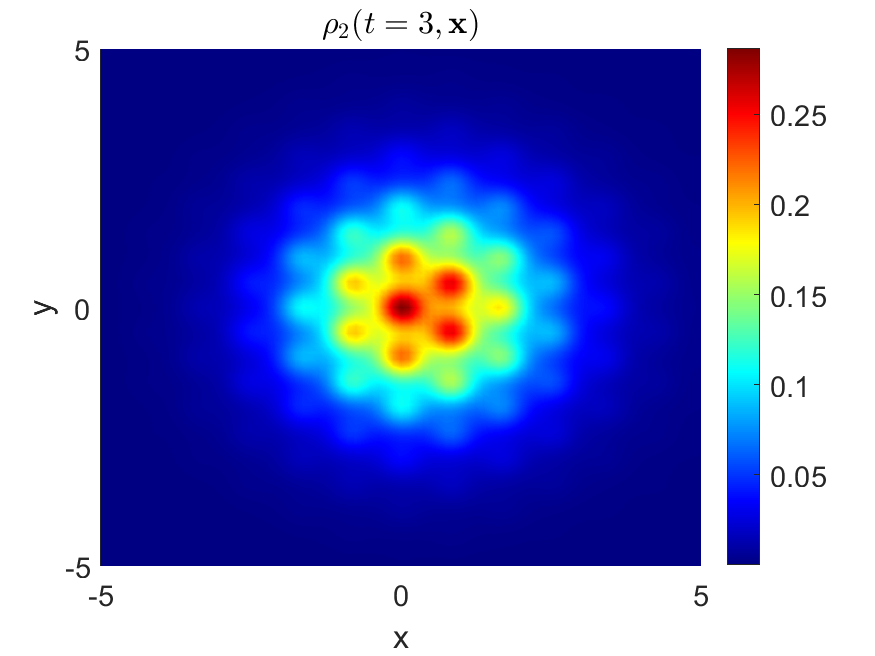,height=2.8cm,width=3.5cm}\\
\end{array}$$
\caption{Problem 4. Contour plots of the densities $\rho_{i}(t,\mathbf{x})=|\phi_{i}(t,\mathbf{x})|^{2},\ i=1,2$ for $\eps=0.1 \ (left \ two \ columns)$ and $\eps=0.01 \ (right \ two \ columns)$ of NLDE \eqref{equ-1} in 2D under the nonlinearity $\lambda_{1}=0$, $\lambda_{2}=1$ with a honeycomb potential.}\label{fig-3-1-3}
\end{figure}

 \vspace{0.15cm}\textbf{Problem 5. Linear case with magnetic potential.}
As the last numerical test,  we study numerically the dynamics of the linear Dirac equation \eqref{equ-5} with a honeycomb lattice potential, i.e. we take $d=2$, $A_{1}(\mathbf{x})=A_{2}(\mathbf{x})\equiv 0$, $V(\mathbf{x})$ and initial data $\Phi_{0}(\mathbf{x})$ are still taken as before. The problem is solved numerically on $\Omega=[-10,10]^{2}$ with mesh size $\triangle x=1/16$, $N_{\tau}=2^{6}$. Firstly we test the time error of LDE \eqref{equ-5} and the long-term performance, and the corresponding results are presented in Figures \ref{fig-3-10-3}-\ref{fig-3-10-4}.
Figure \ref{fig-3-1-5} depicts the densities $\rho_{j}(t,\mathbf{x})=|\phi_{j}(t,\mathbf{x})|^{2} (j=1,2)$ for $\eps=0.2$ and $\eps=0.01$ by the SEP-TS4 schemes with time step $\triangle t=0.01$.
As $\eps\rightarrow 0^{+}$, the relativistic effects gradually vanish, and the Dirac equation asymptotically reduces to the Schr\"{o}dinger equations, leading to a smoother spatial distribution of the densities across the lattice potential.
\begin{figure}[t!]
$$\begin{array}{cc}
\psfig{figure=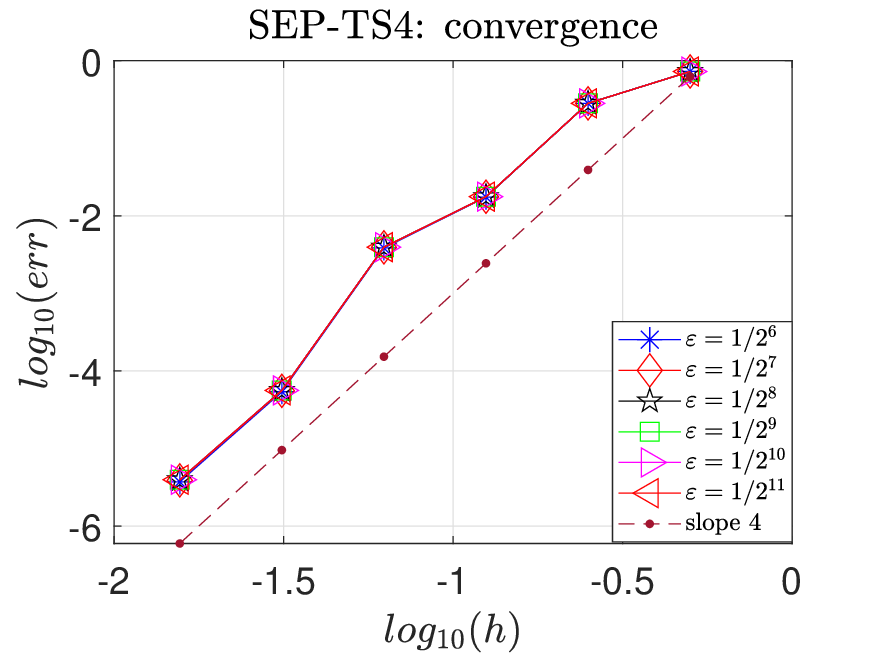,height=3cm,width=6cm}
\psfig{figure=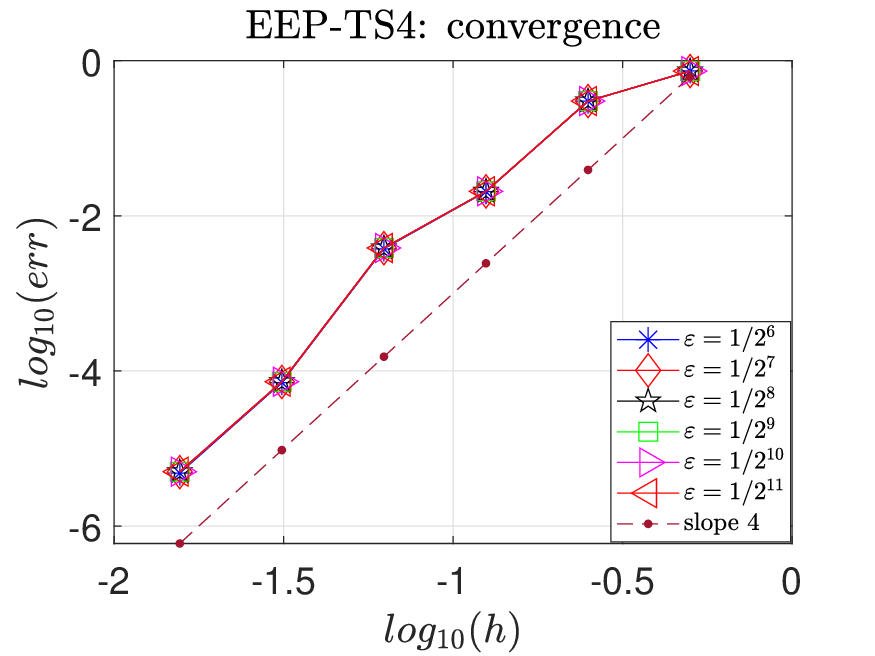,height=3cm,width=6cm}
\end{array}$$
\caption{Problem 5. Temporal error of LDE \eqref{equ-5} in 2D at $t=1$ under different $\eps$.}\label{fig-3-10-3}
\end{figure}

\begin{figure}[t!]
$$\begin{array}{cc}
\psfig{figure=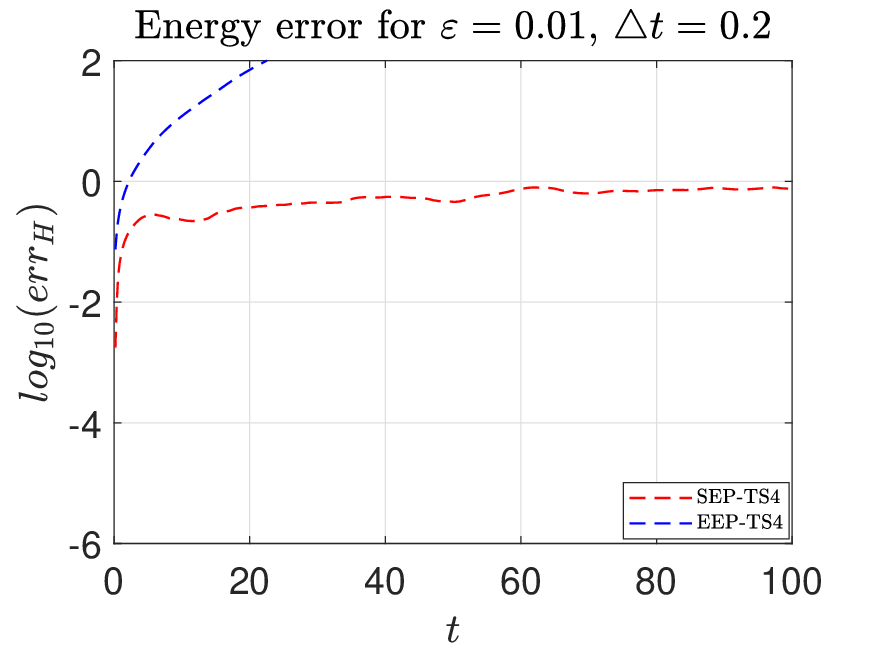,height=3cm,width=5cm}
\psfig{figure=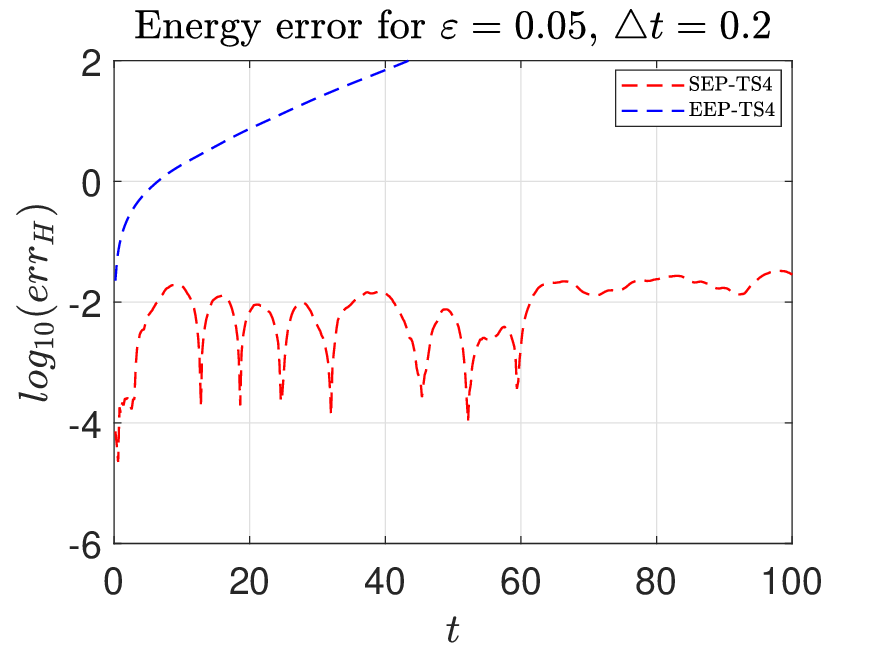,height=3cm,width=5cm}
\psfig{figure=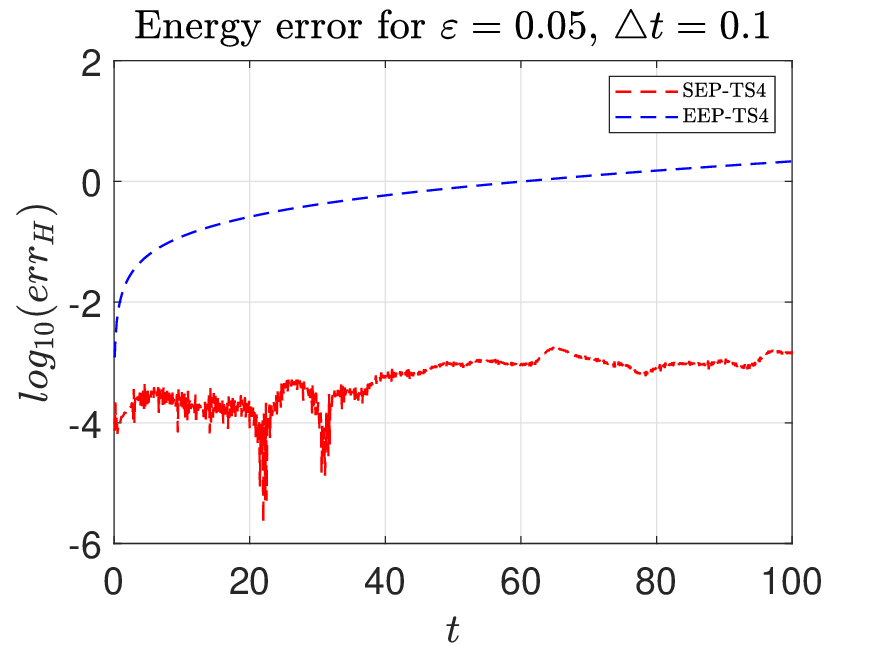,height=3cm,width=5cm}\\
\psfig{figure=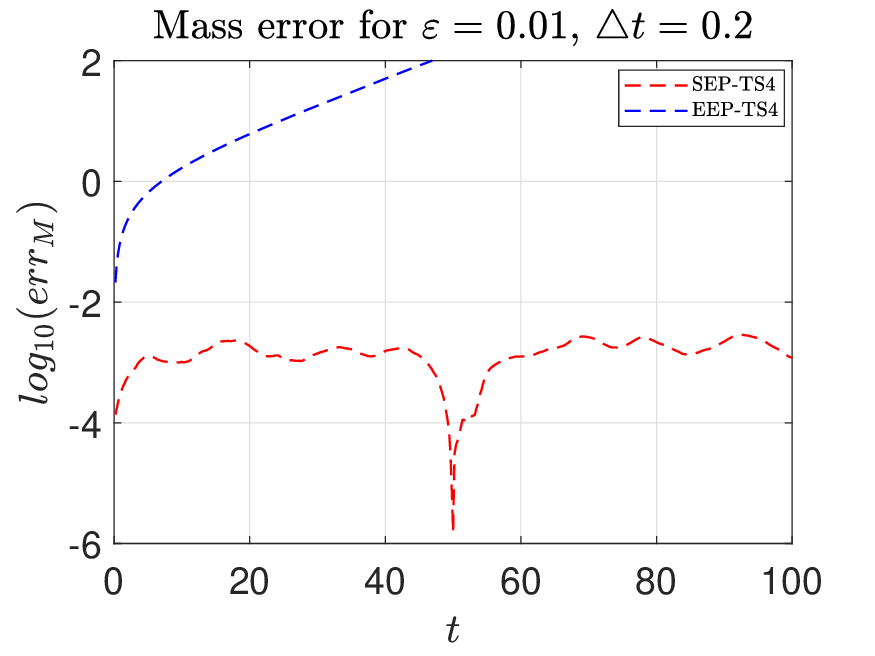,height=3cm,width=5cm}
\psfig{figure=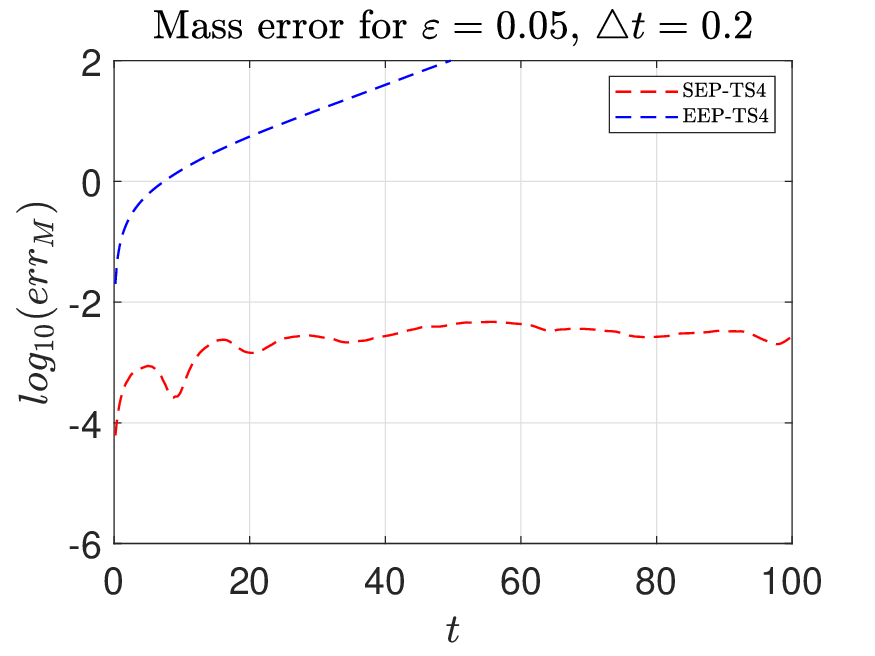,height=3cm,width=5cm}
\psfig{figure=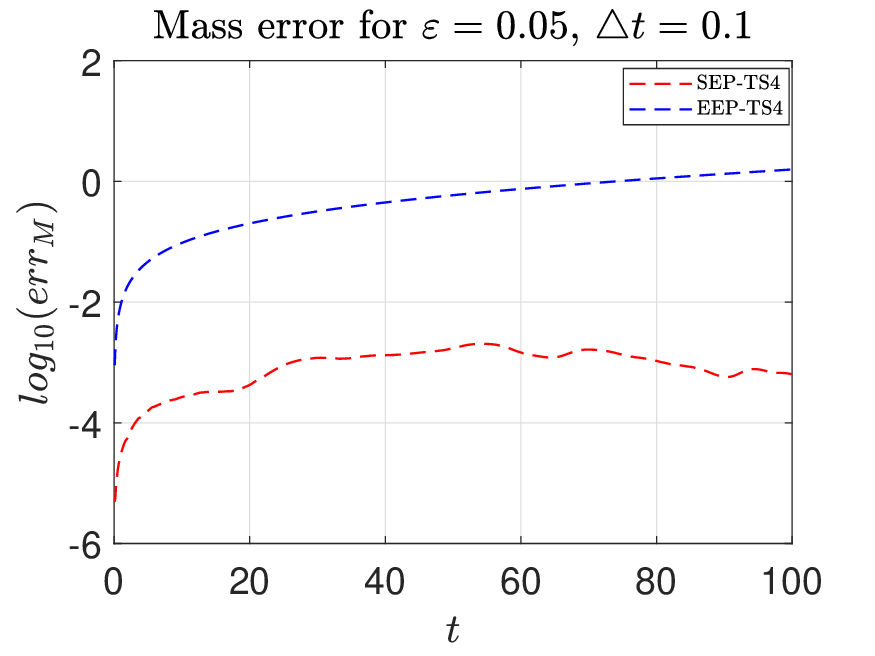,height=3cm,width=5cm}
\end{array}$$
\caption{Problem 5. Energy error (top) and mass error (bottom) of LDE \eqref{equ-5} in 2D under different $\eps$ and $\triangle t$.}\label{fig-3-10-4}
\end{figure}

\begin{figure}[t!]
$$\begin{array}{cc}
\psfig{figure=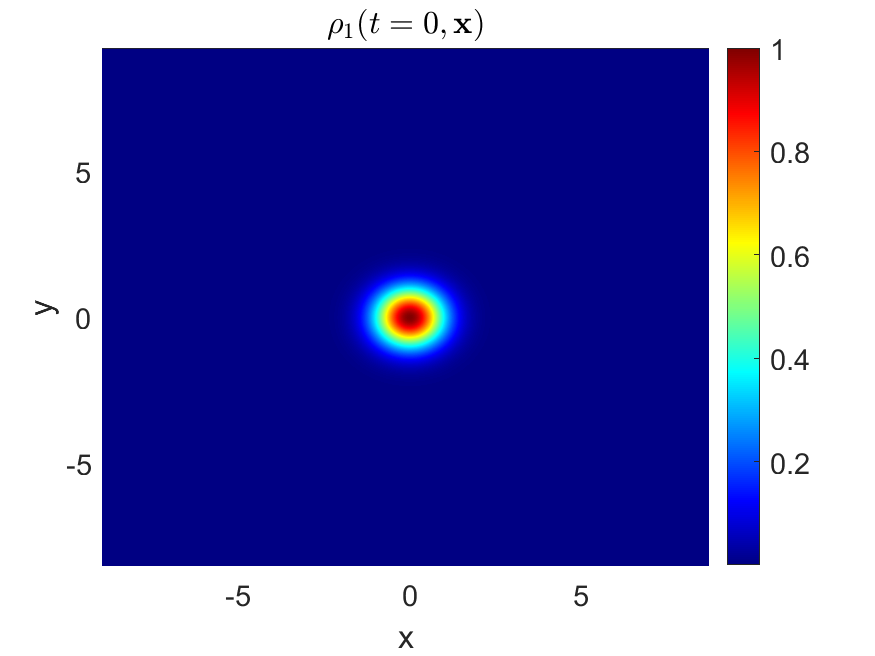,height=2.8cm,width=3.5cm}
\psfig{figure=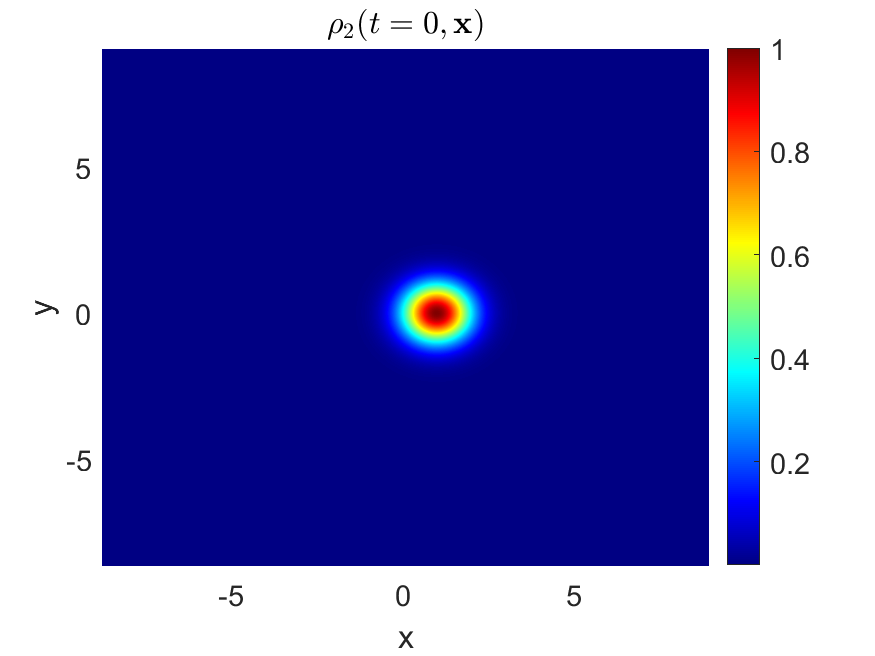,height=2.8cm,width=3.5cm}
\psfig{figure=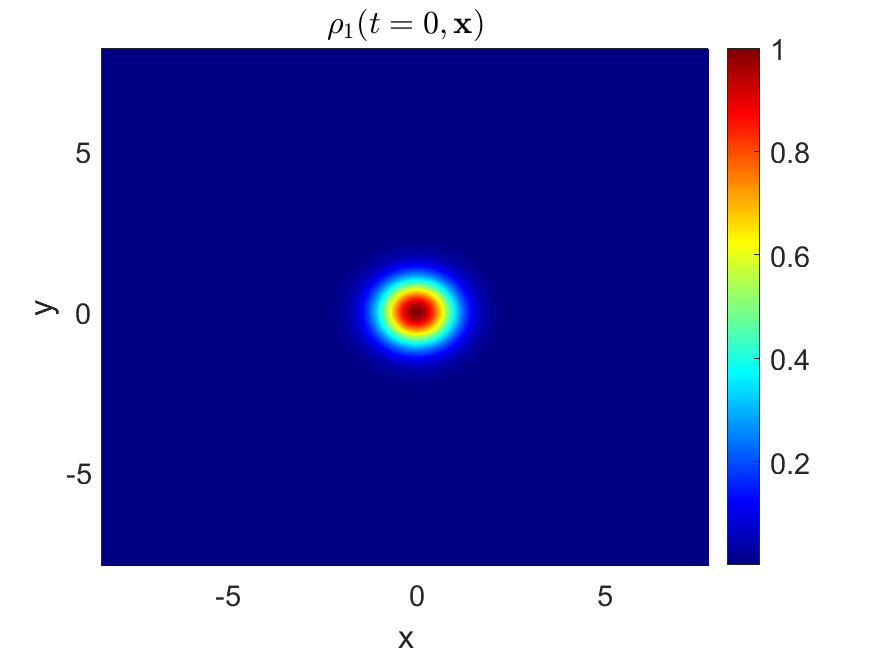,height=2.8cm,width=3.5cm}
\psfig{figure=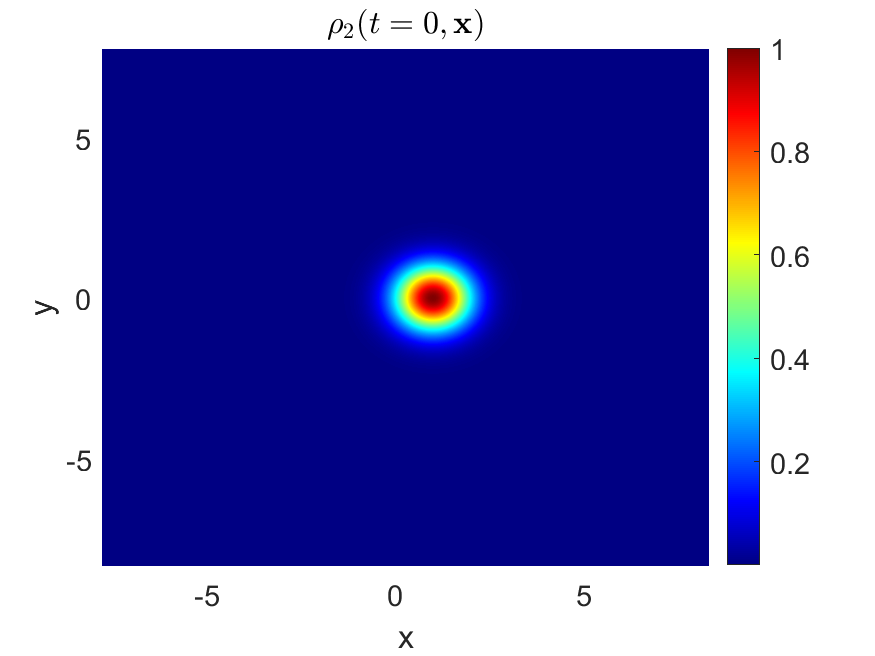,height=2.8cm,width=3.5cm}\\
\psfig{figure=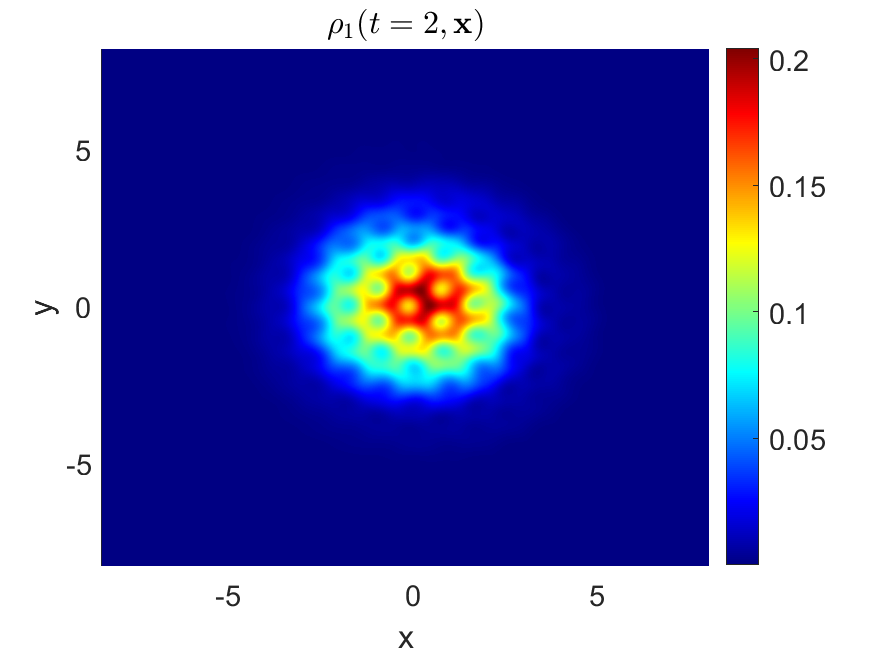,height=2.8cm,width=3.5cm}
\psfig{figure=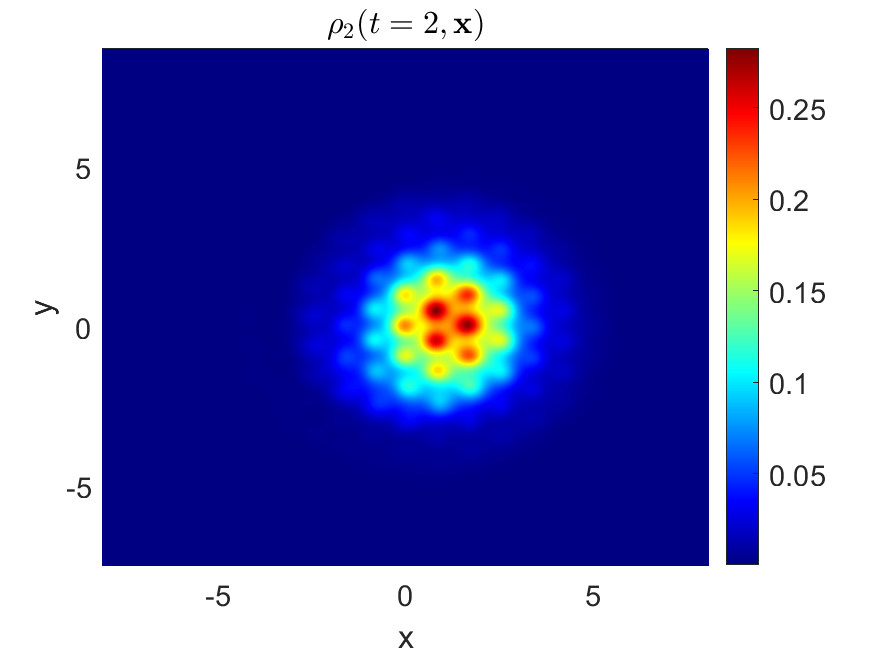,height=2.8cm,width=3.5cm}
\psfig{figure=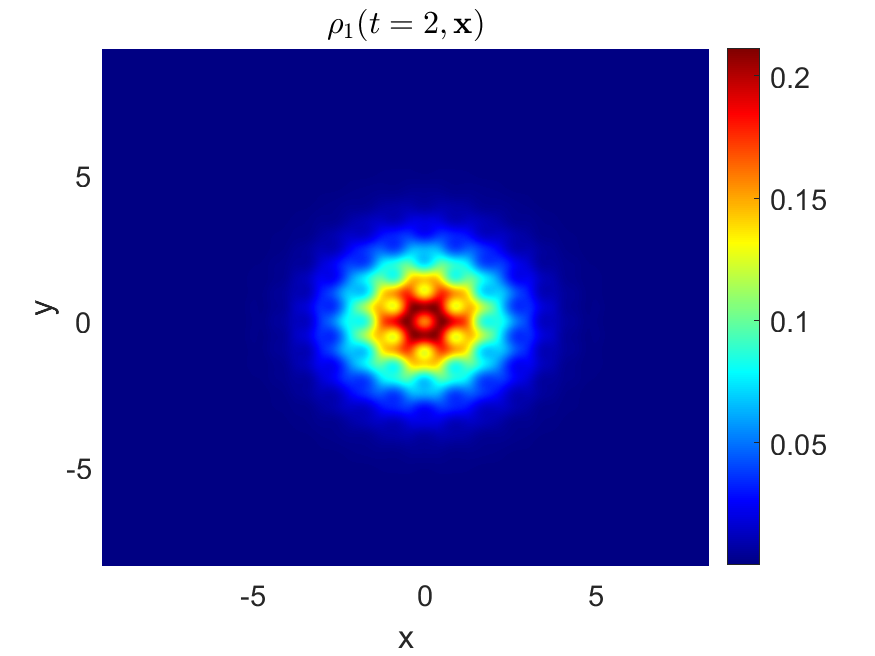,height=2.8cm,width=3.5cm}
\psfig{figure=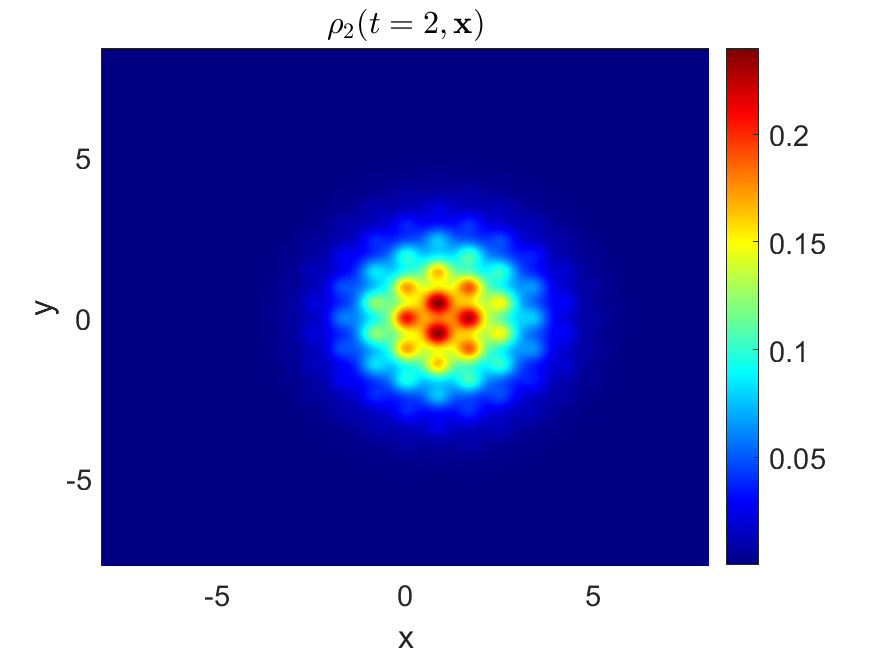,height=2.8cm,width=3.5cm}\\
\psfig{figure=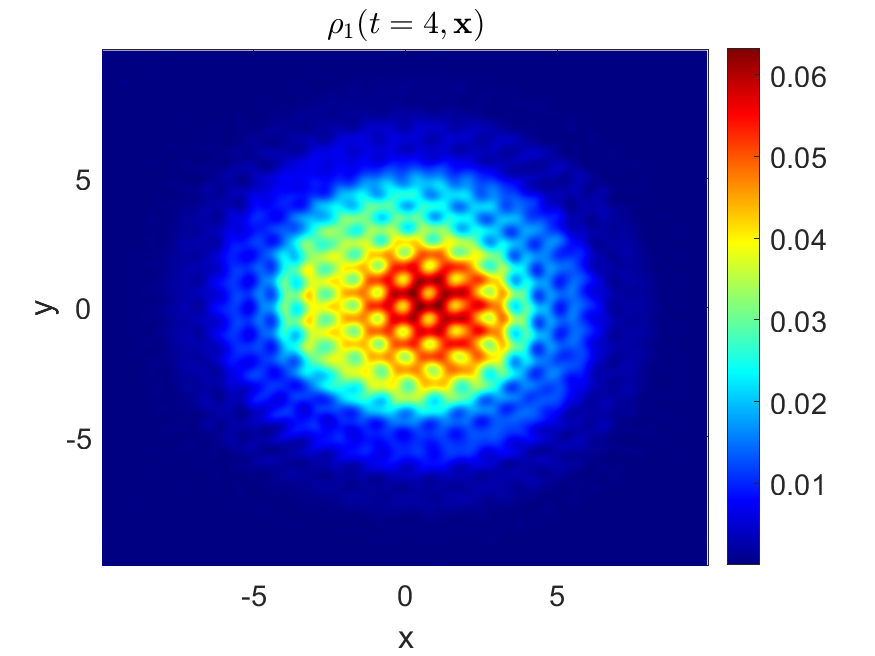,height=2.8cm,width=3.5cm}
\psfig{figure=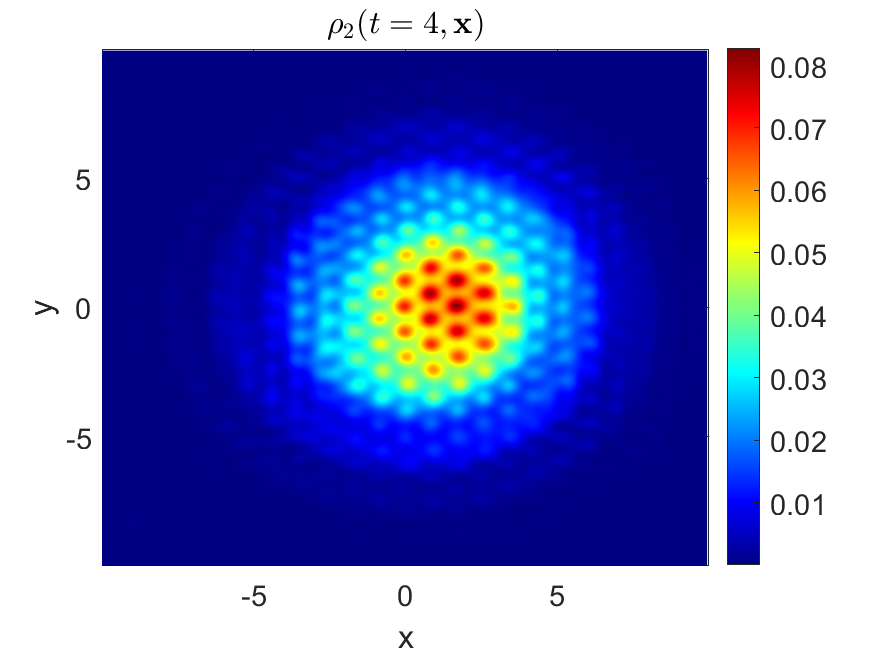,height=2.8cm,width=3.5cm}
\psfig{figure=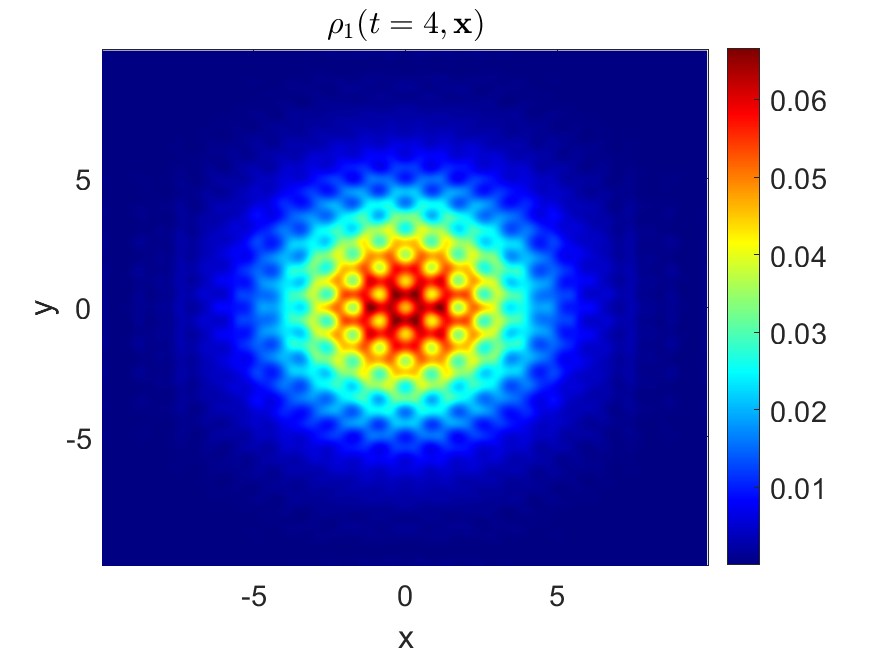,height=2.8cm,width=3.5cm}
\psfig{figure=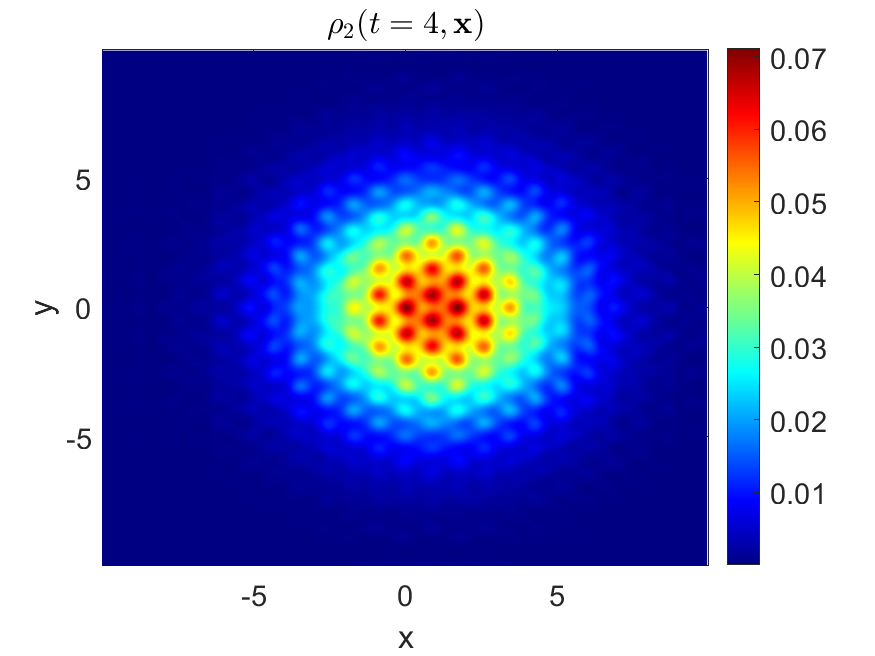,height=2.8cm,width=3.5cm}\\
\psfig{figure=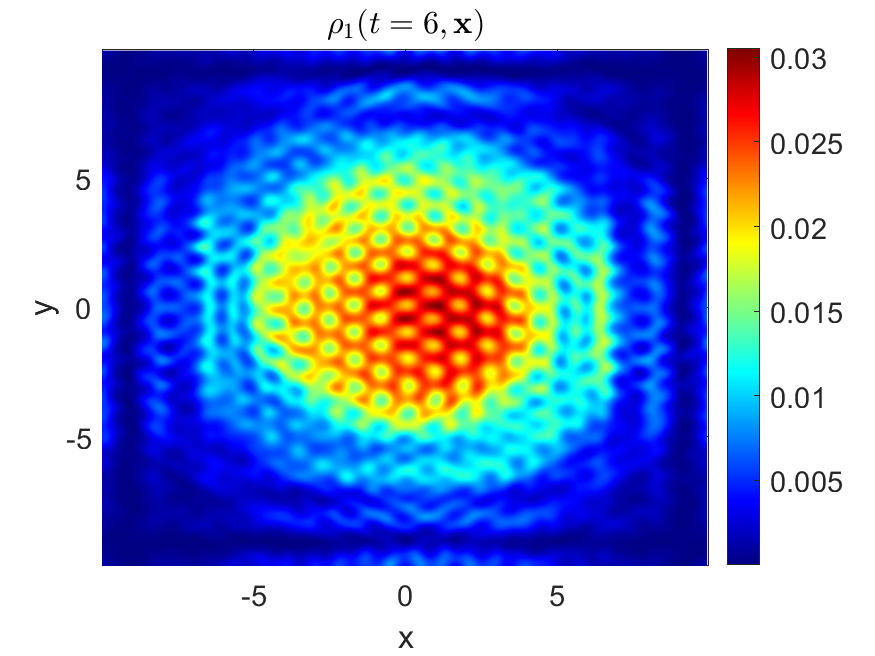,height=2.8cm,width=3.5cm}
\psfig{figure=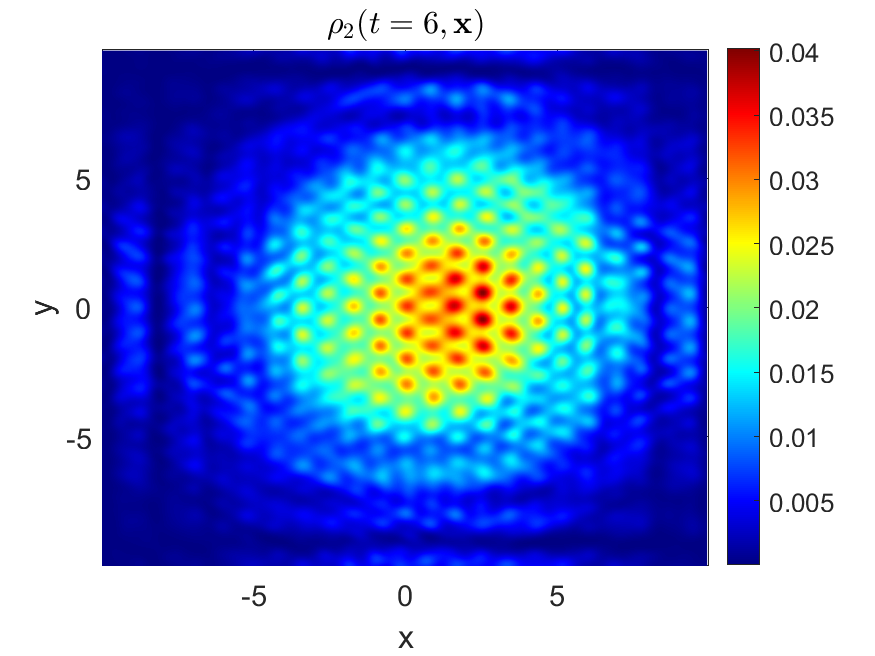,height=2.8cm,width=3.5cm}
\psfig{figure=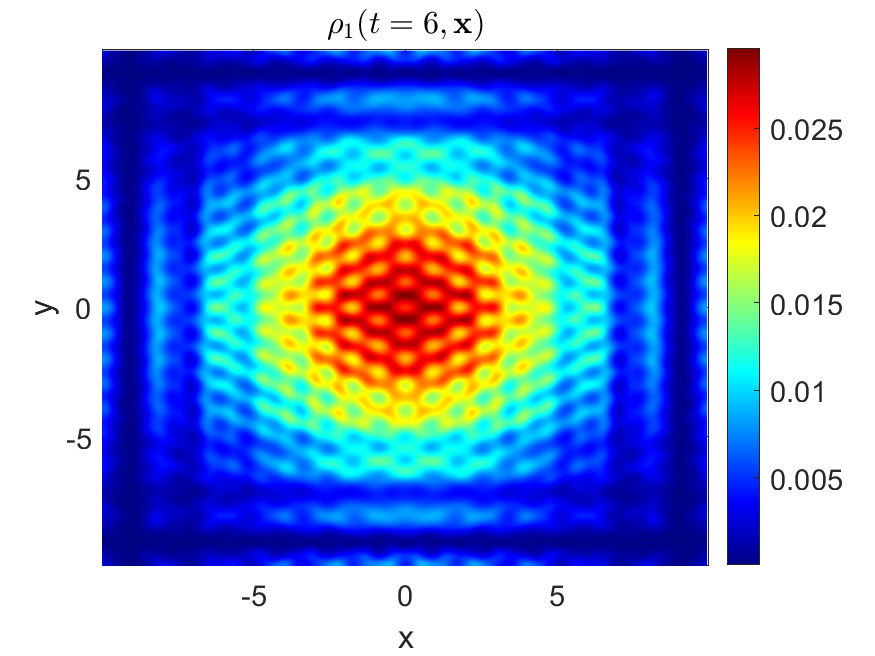,height=2.8cm,width=3.5cm}
\psfig{figure=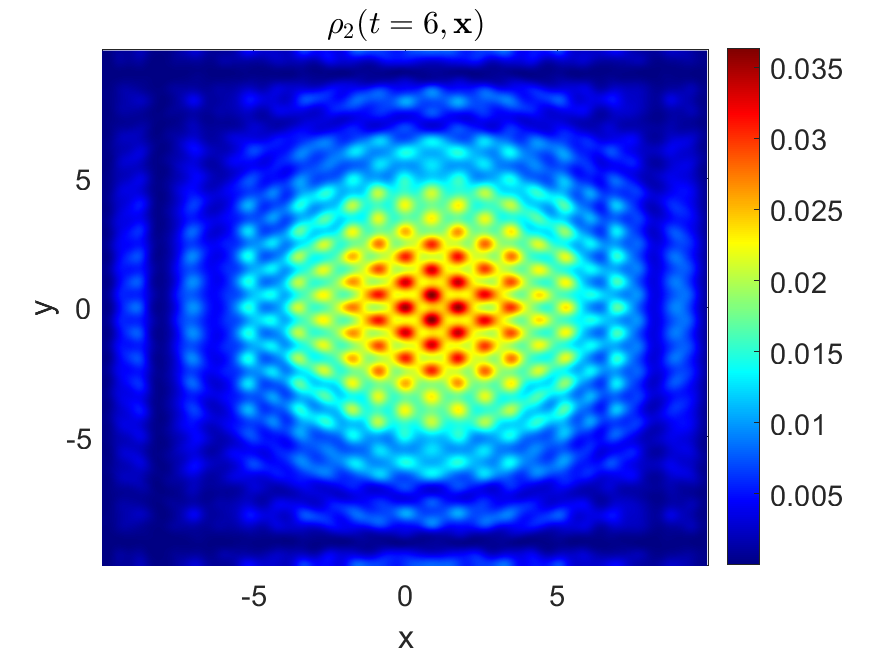,height=2.8cm,width=3.5cm}\\
\end{array}$$
\caption{Problem 5. Contour plots of the densities $\rho_{i}(t,\mathbf{x})=|\phi_{i}(t,\mathbf{x})|^{2},\ i=1,2$ of LDE \eqref{equ-5} in 2D for $\eps=0.2 \ (left \ two \ columns)$ and $\eps=0.01 \ (right \ two \ columns)$ with a honeycomb potential.}\label{fig-3-1-5}
\end{figure}

\section{Conclusion}\label{sec:con}
In this work, we  developed and analyzed two novel fourth-order integrators for solving the nonlinear Dirac equation in the nonrelativistic regime. Through a sequence of   transformations of the original system, application of the two-scale formulation approach, spectral semi-discretization, and fourth-order exponential integrators, we   constructed a symmetric integrator and an explicit scheme.
These two integrators were rigorously proved to achieve fourth-order uniform accuracy. For the symmetric scheme,   long-term energy conservation was shown  by using modulated Fourier expansion. Moreover, the proposed integrators were applied to different kinds of Dirac equation. Several numerical experiments were conducted to illustrate the superior performance  of the proposed  integrators.

% \section*{Acknowledgements}
%The authors are grateful to   the  anonymous reviewers for their very valuable suggestions, which help improve
%this paper significantly.

\bibliographystyle{model-num-names}

\end{document}